\documentclass[11pt]{amsart}
\usepackage[applemac]{inputenc}

\usepackage{fullpage}
\usepackage{amsmath,amsthm,amssymb,latexsym,epsfig,graphicx}

\usepackage{
  ulem, amsfonts}
 \usepackage{amsmath}
\usepackage{amsmath,esint, amssymb, graphicx, amsthm,mathbbol,upgreek,exscale}

 \usepackage{mathtools}

\usepackage{graphicx}
\usepackage{array,booktabs,calc}

\usepackage[colorlinks=true, pdfstartview=FitV, linkcolor=blue, citecolor=blue,
 urlcolor=blue]{hyperref}
\usepackage{color}

\date{}

\newcommand\N{\mathbb{N}} 
\newcommand\Z{\mathbb{Z}} 
\newcommand\R{\mathbb{R}} 
\newcommand\C{\mathbb{C}} 
\newcommand\T{\mathbb{T}} 
\newcommand\D{\mathbb{D}}

\newcommand\EE{\varepsilon}

\newcommand\xx{{\boldsymbol x}}        
\newcommand\bxx{\bar{\boldsymbol x}}   
\newcommand\uu{{\boldsymbol u}}        
\newcommand\tuu{\tilde{\boldsymbol u}} 
\newcommand\bcdot{{\boldsymbol \cdot}}        
\newcommand\CC{\mathcal{C}}            
\newcommand\der{\textrm{d}}            
\newcommand\vt{\vartheta}              
\newcommand\gperp{\nabla^{\perp}}      

\newcolumntype{C}[1]{>{\centering\arraybackslash}m{#1}}

\theoremstyle{plain}
\numberwithin{equation}{section}
\newtheorem{theorem}{Theorem}[section]
\newtheorem{proposition}[theorem]{Proposition}
\newtheorem{lemma}[theorem]{Lemma}
\newtheorem{remark}[theorem]{Remark}

\usepackage[textmath,displaymath,floats,graphics,delayed]{preview}
  \title[]{Imperfect  bifurcation  for  the quasi-geostrophic shallow-water equations}
\author[D. G.  Dritschel]{David Gerard Dritschel}
\address{Mathematical Institute, University of St Andrews\\
St Andrews KY16 9SS\\ UK}
\email{david.dritschel@st-andrews.ac.uk}

\author[T. Hmidi]{Taoufik Hmidi}
\address{Univ Rennes, CNRS, IRMAR - UMR 6625, F-35000 Rennes, France}
\email{thmidi@univ-rennes1.fr }

 \author[C. Renault]{Coralie Renault}
\address{Univ Rennes, CNRS, IRMAR - UMR 6625, F-35000 Rennes, France}
\email{Coralie.Renault@ens-rennes.fr}

\begin{document}

\begin{abstract}

We study analytical and numerical aspects of the bifurcation diagram
of simply-connected rotating vortex patch equilibria for the
quasi-geostrophic shallow-water (QGSW) equations.  The QGSW equations
are a generalisation of the Euler equations and contain an additional
parameter, the Rossby deformation length $\EE^{-1}$, which enters in
the relation between streamfunction and (potential) vorticity.  The
Euler equations are recovered in the limit $\EE\to 0$.  We prove,
close to circular (Rankine) vortices, the persistence of the
bifurcation diagram for arbitrary Rossby deformation length.  However
we show that the two-fold branch, corresponding to Kirchhoff ellipses
for the Euler equations, is never connected even for small values
$\EE$, and indeed is split into a countable set of disjoint connected
branches.  Accurate numerical calculations of the global structure of
the bifurcation diagram and of the limiting equilibrium states are
also presented to complement the mathematical analysis.
   
\end{abstract}

\newpage 

\maketitle{}
\tableofcontents

\section{Introduction}
In this paper we investigate new aspects of simply-connected vortex
patch relative equilibria satisfying the quasi-geostrophic shallow
water (QGSW) equations.  These equations are derived asymptotically
from the rotating shallow water equations, in the limit of rapid
rotation and weak variations of the free surface \cite{Vallis}.  A key
property of these equations, and of the parent shallow water
equations, is the {material} conservation of `potential
vorticity', $q$, a quantity which remains unchanged following fluid
particles.  The QGSW equations are given by
\begin{equation}\label{sqg}
\left\lbrace
\begin{array}{l}
\partial_t q + v \cdot \nabla q=0, \text{ }(t,x)\in \R_+ \times \R^2, \\
v=\nabla^\perp\psi,\\
\psi=(\Delta-\EE^2)^{-1}q,\\
q_{|t=0}=q_0
\end{array}
\right.
\end{equation}
where $v$ refers to the velocity field, $\psi$ is the streamfunction, 
$\nabla^\perp=(-\partial_2,\partial_1)$ and $\EE\in \R$.  The
 parameter $\EE$,  when it is positive, is known as the inverse `Rossby deformation length', a
natural length scale arising from a balance between rotation and
stratification.  Small $\EE$ physically corresponds to a free surface
which is nearly rigid.  For $\EE=0$, we recover the two-dimensional
Euler equations.\\

Historically, substantially more is known about vortex patch relative
equilibria for the Euler equations than for the QGSW equations.  The
most famous example is the exact analytical solution for a rotating
ellipse, found by Kirchhoff in 1876 \cite{Kirchhoff}.  This is in fact
a non-trivial {family} of solutions bifurcating from Rankine's
vortex, the circular patch (all axisymmetric vorticity distributions
are in equilibrium, by symmetry).  Kirchhoff's elliptical vortex is
linearly stable when its aspect ratio $\lambda\in[1/3,1]$, i.e.\ when
it is not strongly deformed from a circle, as was shown in 1893 by
Love \cite{Love}.  Love also discovered that there is a sequence of
instabilities, at $\lambda=\lambda_m$ for $m=3,\,4,\,\ldots$, ordered
by the elliptical coordinate azimuthal wavenumber $m$.  That is,
$\lambda_3>\lambda_4>\ldots$, with $\lambda_m\to0$ as $m\to\infty$
\cite{D86}.  Later, it was discovered that new branches of vortex
patch equilibria bifurcate from each of these instability points
\cite{Cerr,Luz}, revealing that the equilibrium solutions of Euler's
equations are exceedingly rich and varied. Analytical proofs and results on the boundary regularity were given in \cite{3,4,20,19}.\\

The ellipse is not the only solution which bifurcates from a circle.
Deem and Zabusky in 1978 \cite{11} discovered, by numerical methods,
$m$-fold symmetric vortex patch solutions for $m>2$ which are the
generalizations of Kirchhoff's ellipse.  The near-circular solutions
resemble small-amplitude boundary waves.  Their existence was proved
analytically by Burbea \cite{2} using a conformal mapping technique and bifurcation theory.
At larger wave amplitudes, the outward protruding crests of the waves
sharpen, ultimately limiting in a shape with right-angle corners.
These corners coincide with hyperbolic stagnation points in the
rotating reference frame in which the patch is steady.  
From an analytical point of view, this problem is still open and
some progress has been recently made in \cite{HMW}.\\

In the present study, we generalize these vortex patch solutions
further by exploring how they are altered for the QGSW equations when
$\EE\in \R^*$ (the Euler case corresponds to $\EE=0$).  We focus on
$\EE\ll1$ where analytical progress can be made, but also present
numerical solutions that confirm the analysis and extend it to larger
positive $\EE$.  A surprising discovery is a new branch of $m$-fold symmetric
solutions which do {not} bifurcate from the circle (below we
exhibit the case $m=3$).  This is an isolated branch and exists even
for $\EE=0$.\\

Some relevant vortex patch solutions for the QGSW equations are
already available in the literature.  Polvani (1988) \cite{Polvani88}
and Polvani, Zabusky and Flierl (1989) \cite{Polvani89} computed the
generalization of Kirchhoff's ellipse for various values of $\EE$ (as
well as for doubly-connected patches and for multi-layer flows).
Later, P{\l}otka and Dritschel (2012) \cite{pd12} carried out a more
comprehensive analysis of the generalized Kirchhoff ellipse solutions,
including linear stability and nonlinear evolution.  In these studies,
solutions were obtained numerically, starting near the known circular
patch solution.  In all cases, the limiting states were found to be
dumbbell-shaped, specifically two symmetrical teardrop-shaped patches
connected at a single point.  This is in stark contrast with
Kirchhoff's ellipse, which continues as a solution at arbitrarily
small aspect ratio $\lambda$.\\

In fact, as shown in this study, there are other two-fold symmetric
vortex patch solutions of the QGSW equations.  However, these are
{not} on the branch of solutions connected to the circular patch.
We demonstrate numerically, then prove mathematically, that the
Kirchhoff branch for $\EE=0$ splits up into many disconnected branches
for any $\EE>0$.  The limiting dumbbell state found in the above
studies is directly related to the limiting state of {just one}
of the solution branches bifurcating from the Kirchhoff ellipse at
$\lambda=\lambda_4$ found by Luzzatto-Fegiz and Williamson (2010)
\cite{Luz} for $\EE=0$.  The other branch lies on a disconnected
branch of solutions when $\EE>0$.  The same behavior appears to occur
near all other even bifurcations, i.e.\ near
$\lambda=\lambda_6,\,\lambda_8,\,\ldots$, though we only have
numerical evidence for this at present.\\

The plan of the paper is as follows.  In section \ref{sec:num1}, we
describe the numerical method employed to compute the vortex patch
solutions.  Results are then presented in section \ref{sec:num2}, both
for the two-fold and the three-fold singly-connected vortex patch
solutions.  The remainder of the paper is devoted to proving the
existence of $m$-fold solutions (in section \ref{sec:mfold}), and
proving that the two-fold solution branch splits near
$\lambda=\lambda_4$ when $|\EE | \ll1$ (in section \ref{sec:2fold}).  We
conclude in section \ref{sec:conc} and suggest ideas for future work.

\section{Numerical approach}
\label{sec:num1}

To navigate along equilibrium solution branches up to limiting states
having corners on the patch boundaries, it is necessary to employ a
robustly convergent numerical method.  Here we follow Luzzatto-Fegiz
and Williamson (2011) \cite{Luzjcp}, employing a Newton iteration 
to find corrections to the boundary shape from the exact condition
\begin{equation}\label{psi1}
\psi(\xx)-\tfrac12\Omega|\xx|^2=C
\end{equation}
expressing the fact that the streamfunction is constant in a frame
of reference rotating at the angular velocity $\Omega$.  Here $C$ is
a constant, and in this section only $\xx$ refers to the vector
position of a point with coordinates $(x,y)$.  For a 
single patch of uniform (potential) vorticity $q=2\pi$, the
streamfunction $\psi$ is determined from the contour integral
\begin{equation}\label{psi2}
\psi(\xx)=\oint_{\CC}H(\EE r)\left((x'-x)\der{y'}-(y'-y)\der{x'}\right)
\end{equation}
where $\CC$ refers to the boundary of the patch,
$r=|\xx'-\xx|$ is the distance between $\xx'$ and $\xx$, 
and the function $H$ is given by
\begin{equation}\label{psi2bis}
H(z)=\frac{z K_1(z)-1}{z^2}
\end{equation}
where $K_1$ is the modified Bessel function of order 1 \cite{pd12}.
In the limit $z\to 0$, relevant to the Euler equations, $H(\EE r)$
reduces to $\tfrac14(\ln{r^2}-1)$ plus an unimportant constant.\\

We start with a guess $\xx=\bxx(\vt)$ for the shape of $\CC$.  Here,
following \cite{D95}, we use a special coordinate $\vt$ proportional
to the travel time of a fluid particle around $\CC$ from some
pre-assigned starting point.  This coordinate simplifies the
equations which must be solved at each step of the iterative
procedure correcting the boundary shape.  The travel time $t$ is
computed from
\begin{equation}\label{tt1}
t=\int_0^s \frac{\der{s'}}{|\tuu(s')|}
\end{equation}
where $s$ is arc length measured from the starting point, and
$\tuu=\gperp(\psi-\tfrac12\Omega|\xx|^2)$ is the velocity in the
rotating frame of reference
(here $\gperp=(-\partial/\partial{y},\partial/\partial{x})$ is
the skewed gradient operator).  The integral around the entire
boundary, $T_p$, gives the particle orbital period, from which
we define the particle frequency $\Omega_p=2\pi/T_p$, a particularly
useful diagnostic in which to characterize the equilibrium patch
solutions.  From $\Omega_p$, we define the travel-time coordinate
as $\vt=\Omega_p t$, where $t$ is given in (\ref{tt1}).  Notably,
at equilibrium, $\tuu$ is tangent to $\CC$.  During the iteration
to find an equilibrium, this will not be exactly true, but the small
discrepancy has no significant effect on the convergence of the
numerical procedure.\\

We correct the previous guess $\bxx$ at each step of the iteration
by taking
\begin{equation}\label{newguess}
\xx=\bxx+\frac{\hat\eta(\vt)(\bar{y}_\vt,-\bar{x}_\vt)}{|\bxx_\vt|^2}
\end{equation}
where a $\vt$ subscript denotes differentiation with respect to $\vt$.
This represents a normal perturbation to the previous boundary, though
the scalar function $\hat\eta$ has units of area.  This choice of
perturbation leads to the simplest form for the linearized
approximation to (\ref{psi1}) and (\ref{psi2}),
\begin{equation}\label{newguessbis}
  \Omega_p\hat\eta(\vt)
  -\int_0^{2\pi}\hat\eta(\vt')K_0(\EE|\bxx(\vt')-\bxx(\vt)|)\der\vt'
  -\tfrac12\hat\Omega|\bxx(\vt)|^2  
  =C-\bar\psi(\vt)+\tfrac12\bar\Omega|\bxx(\vt)|^2
  \equiv R(\vt)
\end{equation}
where terms only up to first order in $\hat\eta$ have been retained, and
we have additionally included a perturbation $\hat\Omega$ in the rotation
rate $\Omega=\bar\Omega+\hat\Omega$.  Above, $K_0$ is the modified Bessel
function of order 0, and notably $K_0(\EE r)$ reduces to $-\ln{r}$
plus an unimportant constant in the limit $\EE\to 0$.  On the right hand
side of (\ref{newguessbis}), $\bar\psi$ refers to the streamfunction
evaluated using the previous guess $\bxx$, i.e.\ using $\bxx$ in
place of $\xx$ in (\ref{psi2}).  This is a linear integral equation for 
$\hat\eta$ and possibly $\hat\Omega$, but its solution requires
additional constraints.  First of all, we require that the vortex
patch area $A$ remains constant, and without loss of generality we
may take $A=\pi$.  This constraint gives rise to an additional equation
after linearising the expression for area
\begin{equation}\label{area}
A=\frac12\oint_{\CC}x\der{y}-y\der{x}
\end{equation}
leading to
\begin{equation}\label{areapert}
\int_0^{2\pi}\hat\eta(\vt)\der\vt=A-\bar{A}
\end{equation}
where $\bar{A}$ is the area of the previous guess found by using $\bxx$
in place of $\xx$ in (\ref{area}).  If $\Omega$ is held fixed during the
iteration ($\hat\Omega=0$), no further constraints are necessary.  However,
holding $\Omega$ fixed does not allow one to negotiate turning points in
the equilibrium solution branches \cite{Luzjcp}.  If we let $\Omega$ vary
(and thus be determined as part of the solution), we need to impose a
further constraint.  The most natural is angular impulse:
\begin{equation}\label{angi}
J=\frac14\oint_{\CC}|\xx|^2(x\der{y}-y\der{x})
\end{equation}
whose linearisation leads to the additional equation
\begin{equation}\label{angipert}
\int_0^{2\pi}|\bxx(\vt)|^2\hat\eta(\vt)\der\vt=J-\bar{J}
\end{equation}
where $\bar{J}$ is the angular impulse of the previous guess found by
using $\bxx$ in place of $\xx$ in (\ref{angi}).\\

Numerically, the vortex patch boundary for an $m$-fold equilibrium
is represented by $n=400m$ boundary
nodes, approximately equally spaced in $\vt$.  Upon each iteration, the
travel time coordinate is recomputed and the nodes are redistributed
to be equally-spaced in $\vt$, to within numerical discretisation
error.  The algorithm for node redistribution is otherwise the same as
that introduced in \cite{D88} and uses local cubic splines for high
accuracy.  Also, the calculation of the streamfunction $\psi$ and
velocity $\uu=\gperp\psi$ by contour integration is described in
\cite{D88}, and further in \cite{D89} for the QGSW equations (as used
by \cite{pd12} to determine the branch of 2-fold QGSW equilibria
bifurcating from the Rankine vortex).\\

The perturbation function $\hat\eta(\vt)$ is represented as the truncated
Fourier series
\begin{equation}\label{pertseries}
\hat\eta(\vt)=\sum_{j=0}^N a_j\cos(jm\vt)
\end{equation}
which imposes even symmetry.  The same symmetry is imposed for the
node redistribution, so only $201$ boundary nodes are unique.  Not all
vortex patch equilibria have such symmetry \cite{Luzjcp}, but we
restrict attention to symmetric equilibria in this study.  A
truncation of $N=32$ was found sufficient to produce results accurate
to within the plotted line widths below.  Accuracy is not
significantly improved when using larger $N$ because ultimately the
highest modes $\cos(jm\vt)$ for $j$ near $N$ are poorly represented by
the boundary nodes.  In general, we find $N \sim 0.08n/m$ ensures the
highest modes are adequately resolved, as judged by the decay of the
Fourier coefficients $a_j$ for $j$ large.\\

Note the coefficient $a_0$ is directly determined by area
conservation.  From (\ref{areapert}), we find
\begin{equation}\label{a0area}
a_0=\frac{A-\bar{A}}{2\pi}.
\end{equation}
This means that the streamfunction constant $C$ is determined by
averaging (\ref{newguessbis}) over $\vt$, and thus $C$ is generally
determined from all of the $a_n$ and $\hat\Omega$.  However, $C$ is
not needed to find an equilibrium, so this calculation need not be
done.\\

We solve (\ref{newguessbis}) for $a_n$ ($n>0$) --- together with
(\ref{angipert}) for $\hat\Omega$ when $\Omega$ is allowed to vary ---
by substituting (\ref{pertseries}) into (\ref{newguessbis}), then
multiplying both sides by $\pi^{-1}\cos(im\vt)$, and finally integrating
over $\vt$.  This results in the linear system
\begin{equation}\label{linsys}
\sum_{j=1}^N A_{ij}a_j-B_i\hat\Omega=C_i, \qquad i=1,\,2,\,...,\,N
\end{equation}
where the matrix elements $A_{ij}$ and vector components $B_i$ and $C_i$
are given by
\begin{align}\label{linsysbis}
  A_{ij} &= \Omega_p\delta_{ij}
  -\frac{1}{\pi}\int_0^{2\pi}\int_0^{2\pi}\cos(im\vt)\cos(jm\vt')
  K_0(\EE|\bxx(\vt')-\bxx(\vt)|)\der\vt'\der\vt
  \nonumber \\
  B_i &= \frac{1}{2\pi}\int_0^{2\pi}|\bxx(\vt)|^2\cos(im\vt)\der\vt  \\
  C_i &= \frac{1}{\pi}\int_0^{2\pi}R(\vt)\cos(im\vt)\der\vt  \nonumber
\end{align}
where $\delta_{ij}$ is the Kronecker delta.  Notably, $A_{ij}=A_{ji}$.
Care is taken to avoid the logarithmic singularity in $K_0(\EE r)$ by
separating this function into a singular part
$S=-\ln(1-\cos(\vt'-\vt))$, which can be integrated analytically (and
contributes $-\pi\delta_{ij}/(im)$ to $A_{ij}$), and a non-singular
remainder $K_0-S$ which is integrated by two-point Gaussian
quadrature.  The same numerical quadrature is used to compute $B_i$
(if $\hat\Omega\neq 0$) and $C_i$.  The area constant $a_0$ from
(\ref{a0area}) is ignored at this stage and is added after the linear
system (\ref{linsys}) is solved.\\

When $\Omega$ is held fixed, we have $\hat\Omega=0$ in (\ref{linsys})
above, and then (\ref{linsys}) may be directly solved for the coefficients
$a_j$, $j=1,\,2,\,...,\,N$.  Otherwise, we need to add another equation
in order to also obtain $\hat\Omega$.  We do this by fixing the angular
impulse and thus use (\ref{angipert}).  Inserting (\ref{pertseries}) into
(\ref{angipert}) and dividing by $2\pi$ results in the further linear
equation
\begin{equation}\label{linsysangi}
\sum_{j=1}^N B_j a_j=\frac{J-\bar{J}}{2\pi}
\end{equation}
(again ignoring the area constant $a_0$).  If we tag $-\hat\Omega$ to
the end of the vector of coefficients $(a_1,a_2,...,a_N)$, we obtain a
symmetric linear system for this vector, which is easily solved by
standard numerical linear algebra packages.  Only at this stage do we
add $a_0$ from (\ref{a0area}) and fully determine $\hat\eta(\vt)$ from
the sum in (\ref{pertseries}).  We then obtain a new guess for the
vortex boundary shape from (\ref{newguess}), and accept this as the
converged solution if the maximum value in $|\xx-\bxx| < 10^{-7}$.
Otherwise, we use $\xx$ as the next guess $\bxx$ and repeat the
above procedure.\\

The above explains how we obtain a single equilibrium state for either
fixed rotation rate $\Omega$ or fixed angular impulse $J$.  To obtain a
whole family of states or solutions, after convergence to one state, we
slightly change either $\Omega$ or $J$ and search for the next state
using the same iterative procedure described above.\\

This approach fails when we reach a turning point in either $\Omega$ or
$J$.  When this happens, we switch the control parameter
(e.g.\ instead of changing $\Omega$ we change $J$) and continue to the
next turning point.  Fortunately, the turning points for $\Omega$ and
$J$ are almost never coincident, so this is an effective strategy.  A
more elaborate strategy was followed in \cite{Luzjcp}, but the
proposed strategy has been found to be highly effective.  We always
reach limiting states containing near corners on their boundaries,
beyond which there are no other states with the same topology.\\

We note for completeness that the above procedure is readily
generalized to study multiply-connected vortex patch equilibria.  The
primary change is that the constant $C$ and all functions of $\vt$ in
(\ref{newguessbis}) acquire a $k$ subscript (denoting the $k$th patch
boundary or contour), while all functions of $\vt'$ acquire an $\ell$
subscript.  Furthermore, the integral is now summed over all contours
$\ell$ and multiplied by the uniform potential vorticity $q_\ell$.
Each contour must also satisfy an area constraint like
(\ref{areapert}), with $\eta$ and $A$ supplemented by $k$ subscripts.
If angular impulse is fixed, additionally (\ref{angipert}) must be
satisfied.  Here again all functions of $\vt$ acquire a $k$ subscript,
and the integral must be summed over $k$ and multiplied by the uniform
potential vorticity $q_k$.  The impulse is an invariant of the entire
vortex system.\\

The procedure also generalizes to cases in which the vortex patches
steadily translate rather than rotate.  In this case, without loss of
generality, we may suppose that the patches translate in the $x$
direction at speed $U$.  The only changes then required in the
procedure described above is to replace $\tfrac12\Omega|\bxx|^2$
by $-U\bar{y}$ and to impose linear impulse conservation,
\begin{equation}\label{limp}
I=\frac13\sum_k q_k\oint_{{\CC}_k}y_k(x_k\der{y}_k-y_k\der{x}_k).
\end{equation}
When searching for equilibria of a fixed linear impulse $I$, an
additional equation is required to determine the perturbation $\hat{U}$
to the speed $U$.  This is found by linearising (\ref{limp}) after
substituting (\ref{newguess}), supplemented by $k$ subscripts,
yielding
\begin{equation}\label{limppert}
\sum_k q_k\int_0^{2\pi}\bar{y}_k(\vt)\hat\eta_k(\vt)\der\vt=I-\bar{I}
\end{equation}
where $\bar{I}$ is the linear impulse of the previous guess found by
using $\bxx_k$ in place of $\xx_k$ in (\ref{limp}).\\

Various diagnostics may be used to characterize the equilibrium
vortex patch solutions.  For the symmetric $m$-fold solutions
studied here, the only other non-zero invariant besides angular
impulse is the `excess' energy $E$, as defined in \cite{pd12}.
For a single patch of area $A=\pi$, this is determined from
\begin{equation}\label{exene}
  E=\frac{\pi}{4}(\ln(\EE/2)+\gamma)-\frac{1}{4\pi}
  \oint_{\CC}\oint_{\CC}H(\EE r)
  [(\xx'-\xx)\bcdot\der\xx'][(\xx'-\xx)\bcdot\der\xx]
\end{equation}
where $\gamma=0.5772...$ is Euler's constant, $r=|\xx-\xx'|$, and
$H(z)$ is defined in (\ref{psi2bis}).  This expression for $E$ has a
finite limit as $\EE\to 0$, and equals the excess energy for the Euler
equations (see discussion in appendix B of \cite{pd12}).  For a circular
patch, $E=\pi/16$ when $\EE=0$.\\

The energy and angular impulse are important since minima or maxima of
these quantities, as a function of a control parameter like $\Omega$,
generally indicate changes in stability.  In the results below, we
have confirmed this by a direct linear stability analysis outlined in
\cite{D95} and used previously in \cite{pd12}.

\section{Numerical results}
\label{sec:num2}

\subsection{2-fold vortex patch equilibria}
\label{subsec:num2fold}

We begin by discussing 2-fold symmetric vortex patch equilibria, in
particular the structure of the solution branches for small $\EE$.
This structure is illustrated in two ways for $\EE=0.01$ in
figure 1.\\

\begin{figure}
\begin{center}
\begin{tabular}{C{0.475\textwidth} C{0.475\textwidth}} \\
(a) & (b) \\
\includegraphics[height = 0.475\textwidth]{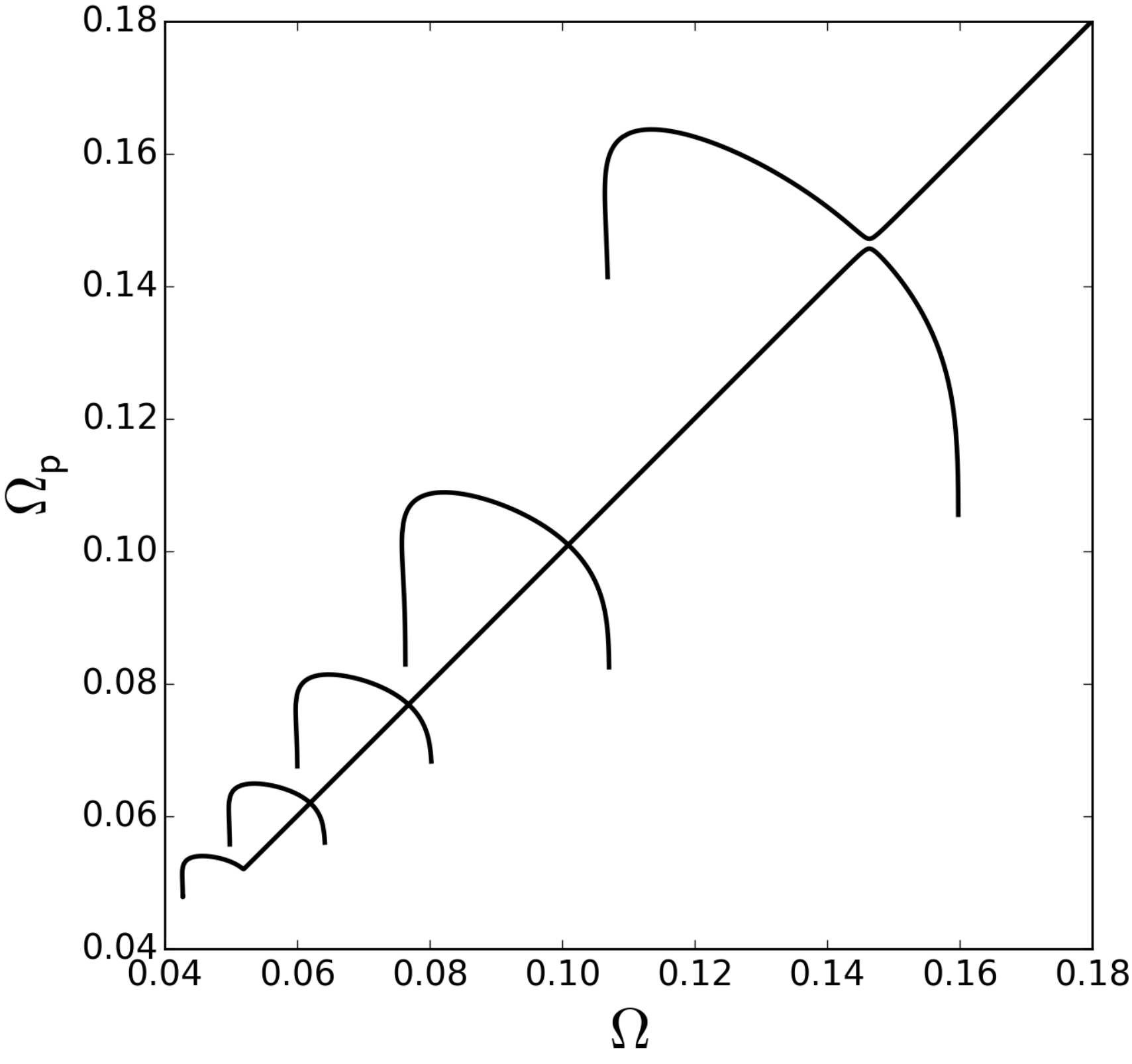}&
\includegraphics[height = 0.475\textwidth]{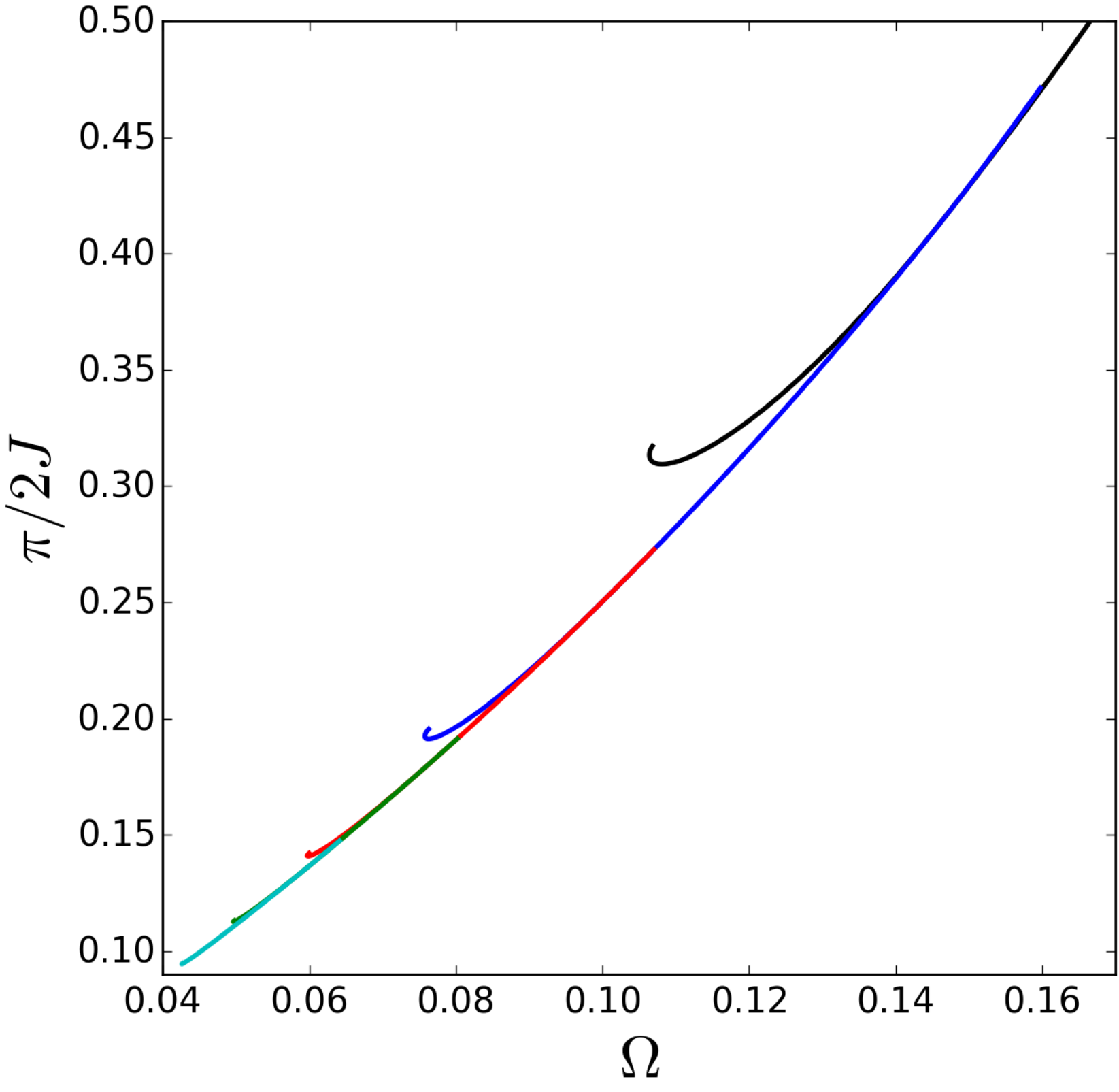} 
\end{tabular}
\label{fig:2struct}
\caption{Solution branch structure for 2-fold vortex patch equilibria
  when $\EE=0.01$.  In (a), we shown the particle frequency $\Omega_p$
  versus rotation rate $\Omega$, while in (b) we show $\pi/2J$ versus
  $\Omega$ and render separate branches in different colours.  Note,
  for a circular vortex patch, $J=\pi/2$.  The upper right branch is
  not shown in its entirety; it begins at the circular patch with
  $\Omega\approx\tfrac14$ and $\Omega_p\approx\tfrac14$.  Also, only
  the first five branches of an infinite set of them converging on
  $\Omega=\Omega_p=0$ are shown.}
\end{center}
\end{figure}


This figure shows the difficulty in distinguishing solution branches
when a conserved quantity like $J$ or $E$ is plotted versus the
control parameter $\Omega$.  On the other hand, $\Omega_p$ versus
$\Omega$ fully opens the branching structure, enabling one to see the
separation in the branch stemming from the circular vortex (upper
right) from the next branch.  In panel (b), the separate branches are
coloured, with the primary one being black, the second one blue
(partly overlaid by the third one in red), etc.  The blue, red, green
and cyan branches all have the same general form.  They rise at larger
$\Omega$ from small $\Omega_p$, reach a maximum in $\Omega_p$, then fall at smaller $\Omega$.  If we could reach the limiting states having one or
more stagnation points on the vortex boundary (where the boundary
exhibits a corner), we would find $\Omega_p=0$ since it would take an
infinite time for a particle to circulate around the boundary.\\

\begin{figure}
\begin{center}
\includegraphics[height = 0.475\textwidth]{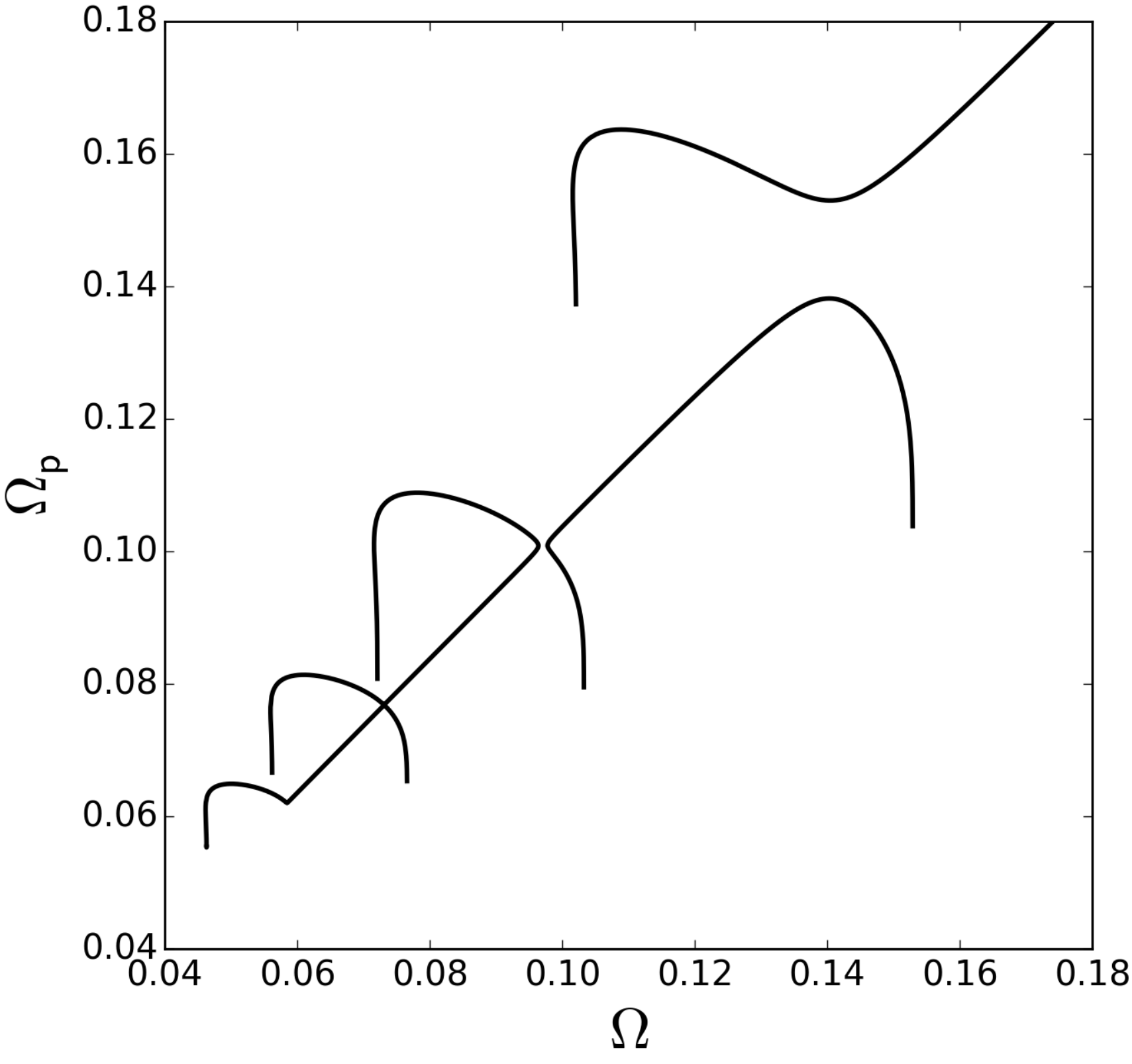}
\label{fig:2structbis}
\caption{Solution branch structure for 2-fold vortex patch equilibria
  when $\EE=0.1$.  Here, only $\Omega_p$ versus $\Omega$ is shown.}
\end{center}
\end{figure}

The separation of these branches becomes increasingly difficult to see
as $\Omega$ decreases; however, at the larger value of $\EE=0.1$
(shown in
figure 2),
we can clearly see the second
branch separating from the third.  We believe this is a generic
feature for all $\EE>0$: all branches separate.\\

The uppermost branch of solutions stemming from the circular vortex
was computed previously by \cite{pd12}, and at that time was thought
to be the only branch of 2-fold solutions.  The limiting state is a
dumbbell shaped vortex touching at a single point at the origin.
On the second branch, there are two limiting states.  The one having
the largest $\Omega$ has a rugby-ball shape with right-angled corners
at the outermost tips.  The other limiting state has the form of an
array of three vortices connected at two stagnation points.  This
pattern continues for the other branches, with the vortex become
increasingly elongated.  The (near) limiting states are illustrated in 
figures 3 and 4
together with the
streamfunction in the co-rotating frame of reference.\\

\begin{figure}
\begin{center}
\begin{tabular}{C{0.475\textwidth} C{0.475\textwidth}} \\
(a) & (b) \\
\includegraphics[width = 0.475\textwidth]{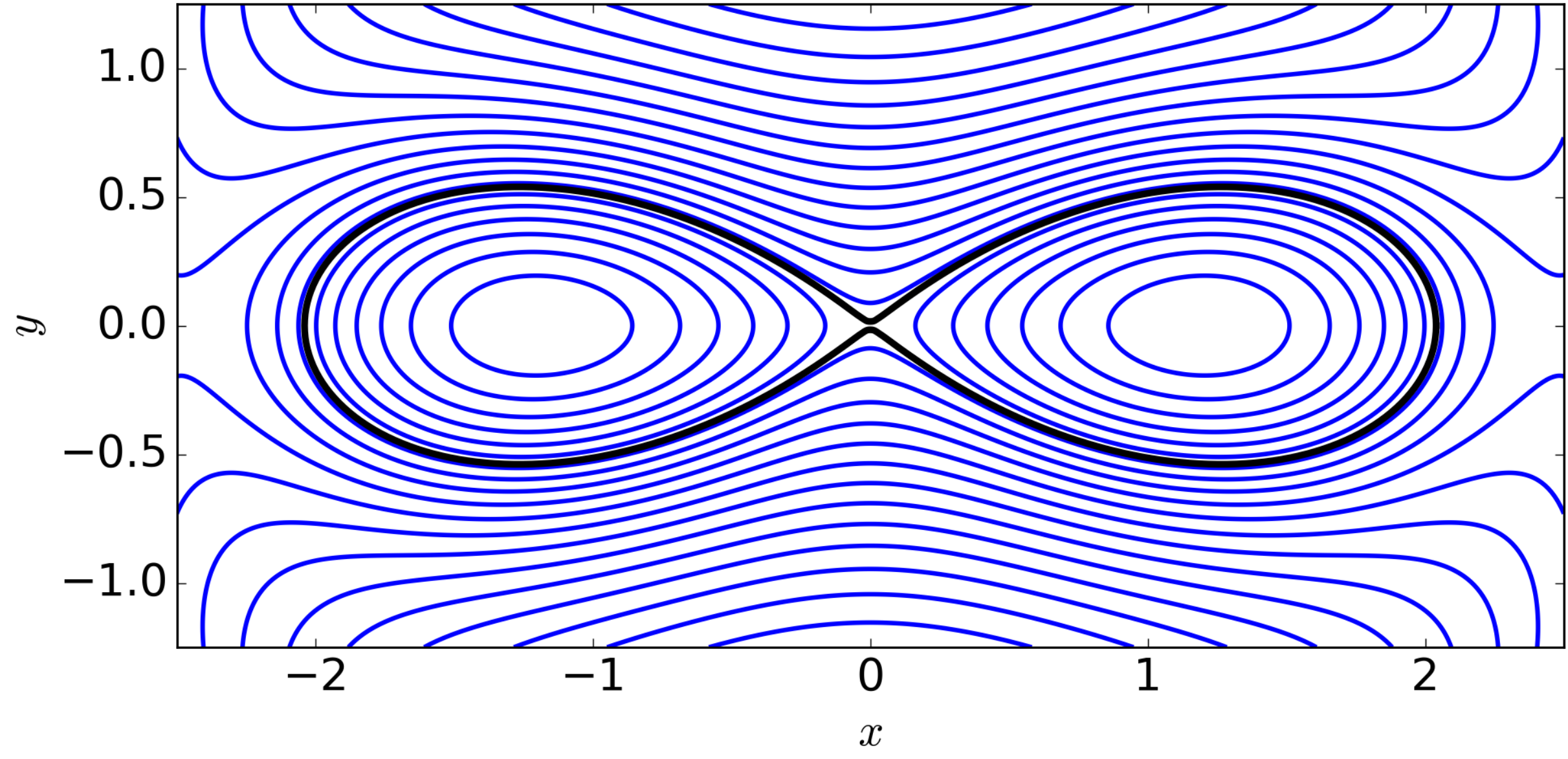} &
\includegraphics[width = 0.475\textwidth]{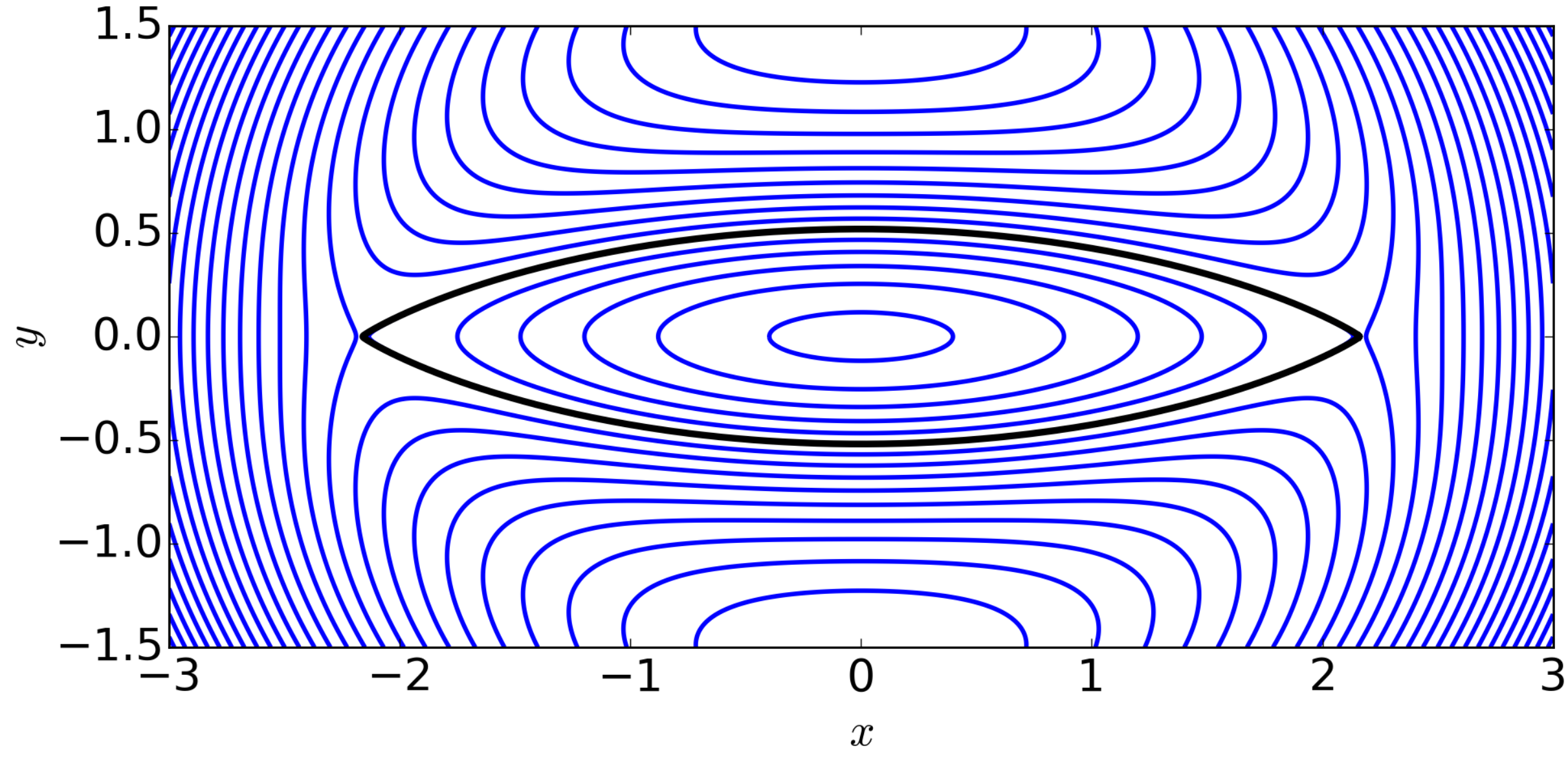} \\
(c) & (d) \\
\includegraphics[width = 0.475\textwidth]{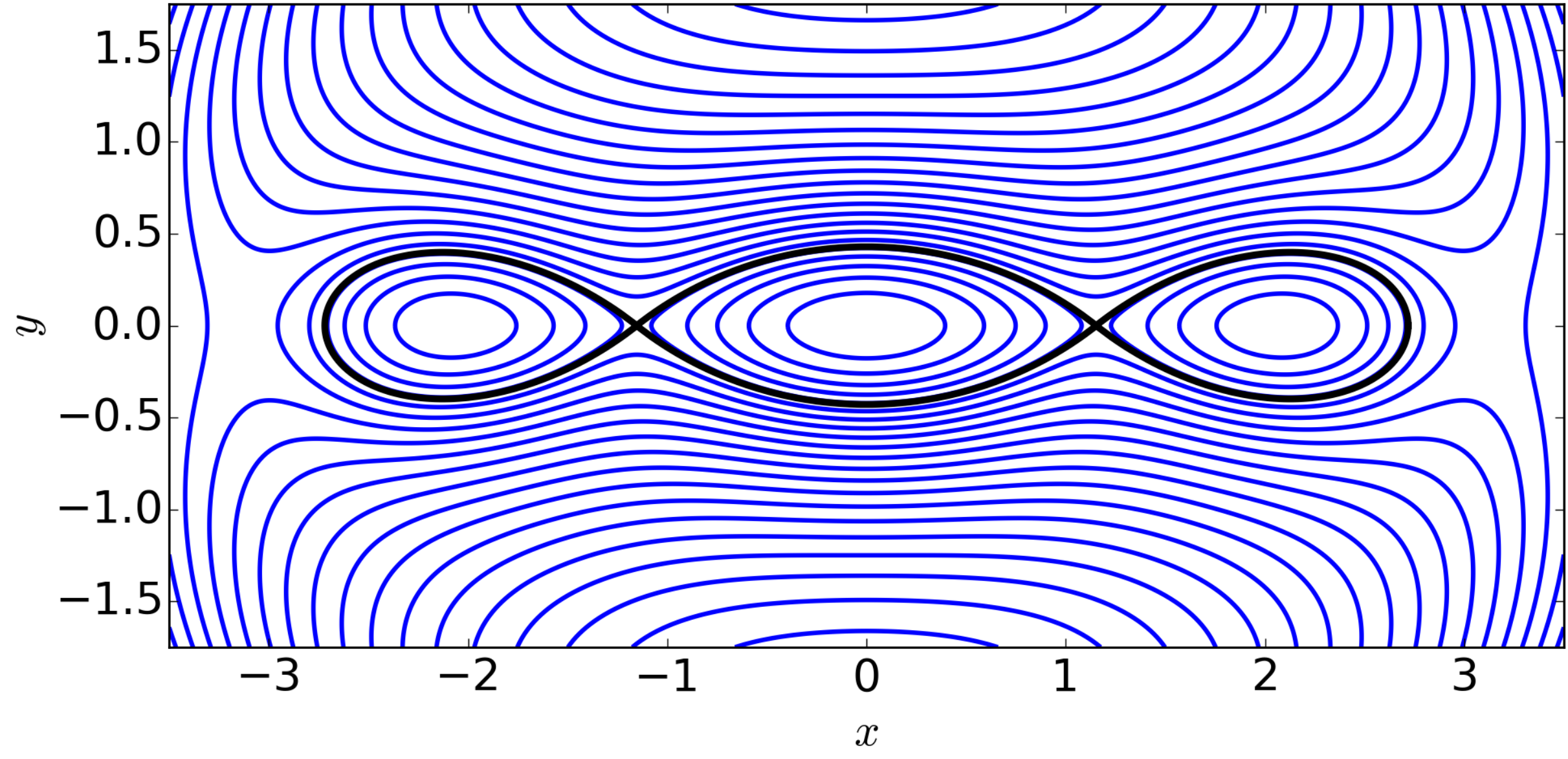} &
\includegraphics[width = 0.475\textwidth]{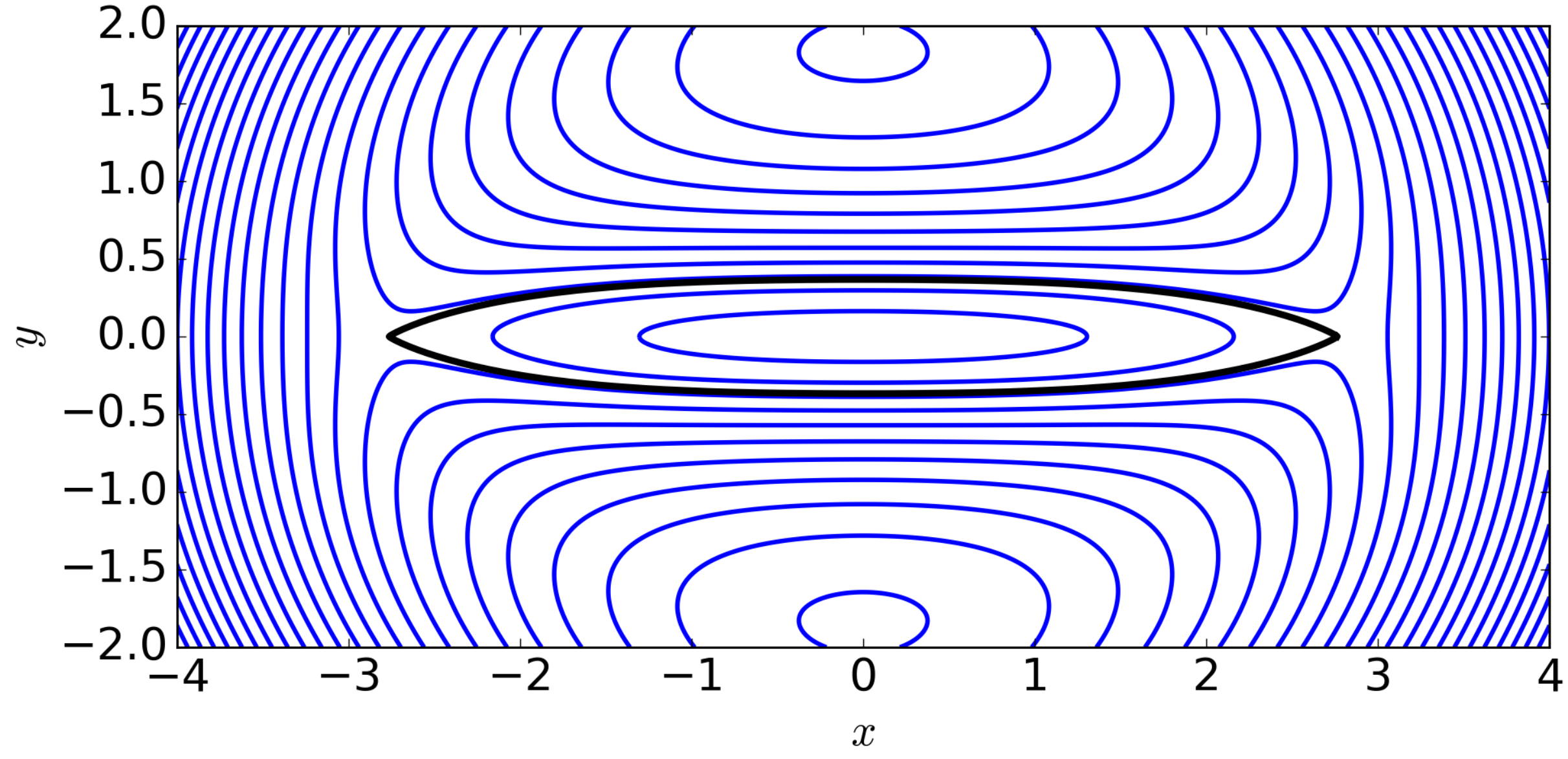}
\end{tabular}
\label{fig:2limitp1}
\caption{Form of the (near) limiting solutions (black contours) and
  the co-rotating streamfunction (blue) for (a) the end of the first
  branch, (b) the start of the second branch, (c) the end of the second
  branch, and (d) the start of the third branch.}
\end{center}
\end{figure}

\begin{figure}
\begin{center}
\begin{tabular}{C{0.475\textwidth} C{0.475\textwidth}} \\
(a) & (b) \\
\includegraphics[width = 0.475\textwidth]{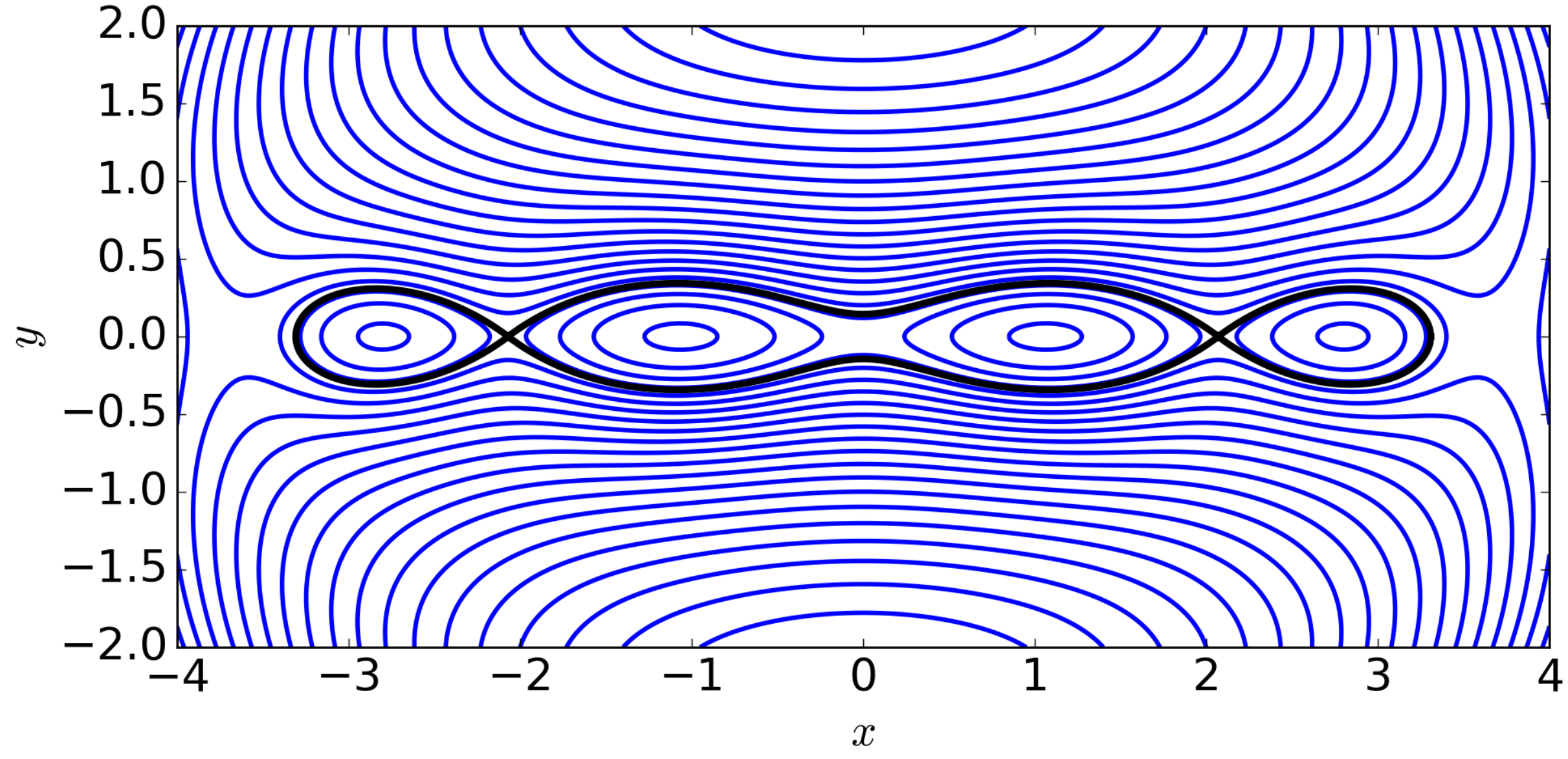} &
\includegraphics[width = 0.475\textwidth]{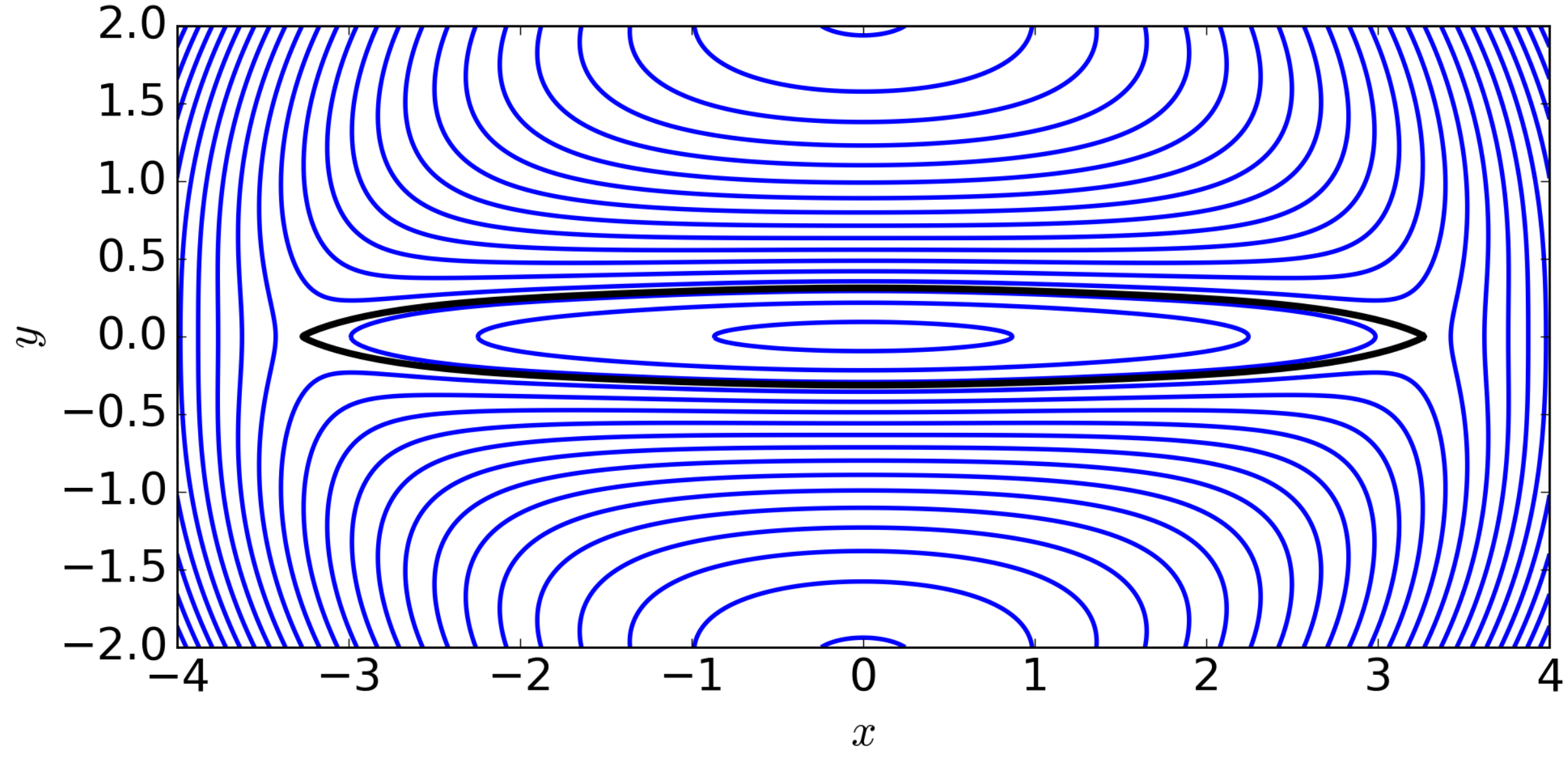} \\
(c) & (d) \\
\includegraphics[width = 0.475\textwidth]{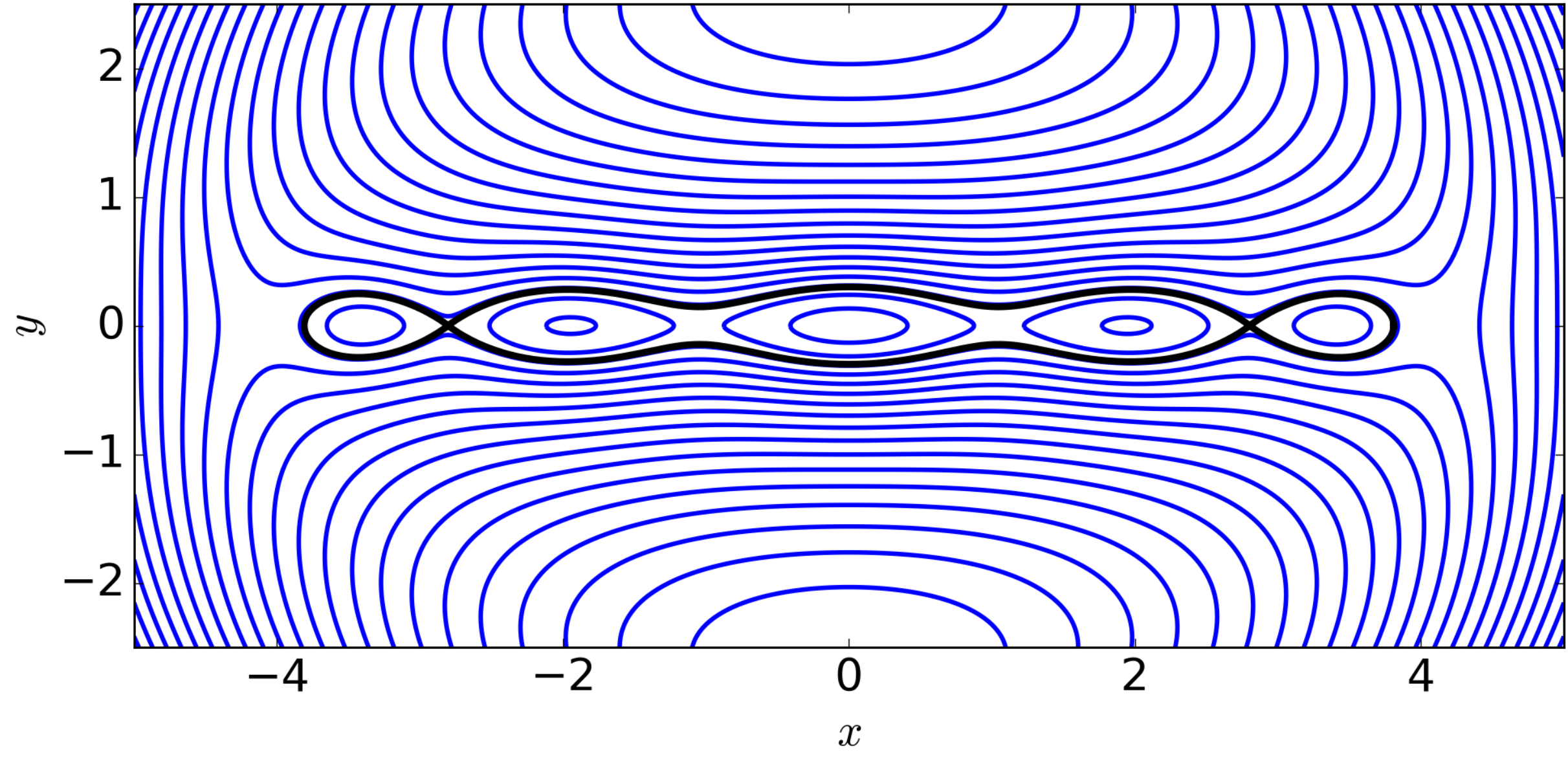} &
\includegraphics[width = 0.475\textwidth]{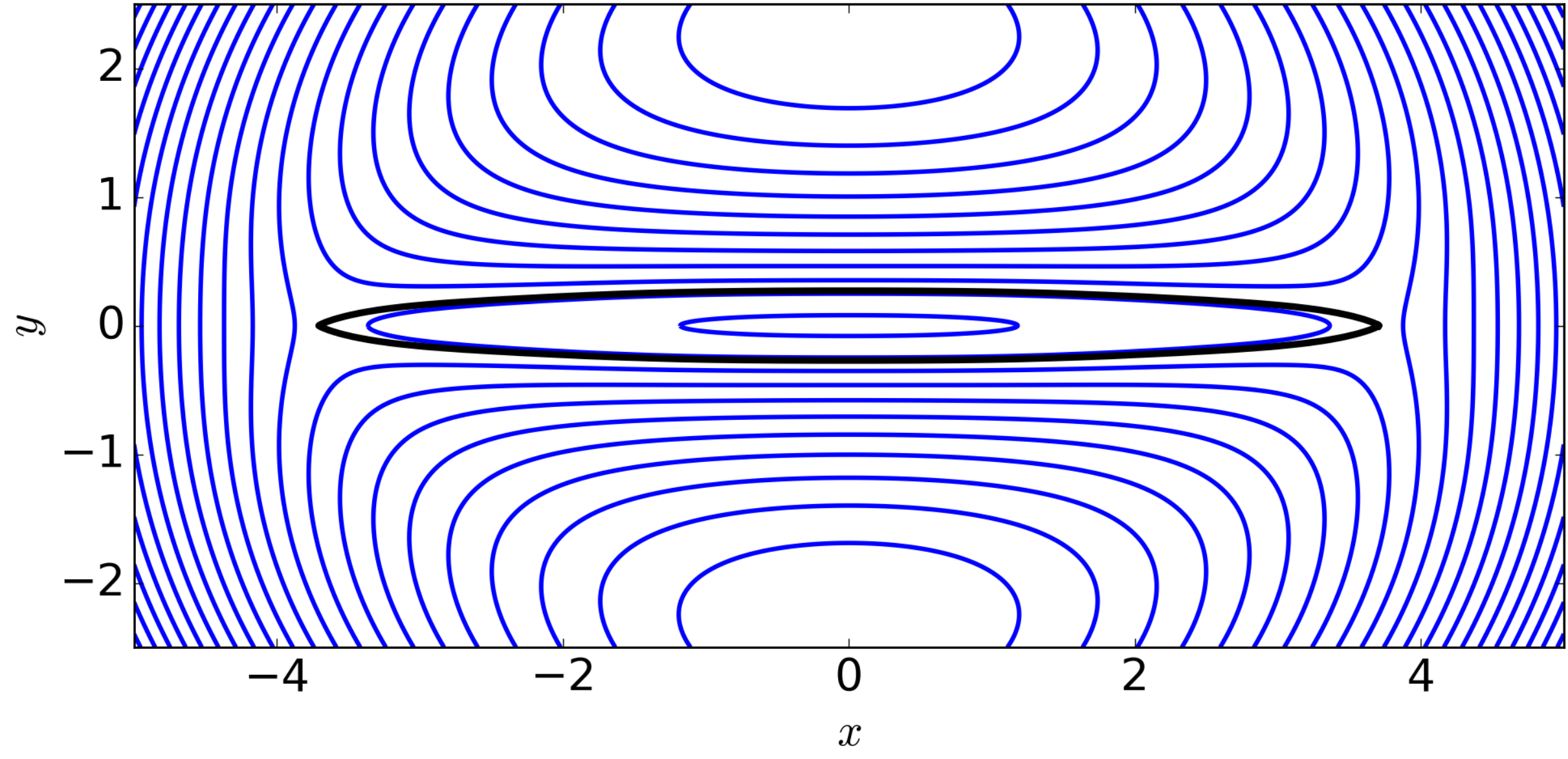}
\end{tabular}
\label{fig:2limitp2}
\caption{Form of the (near) limiting solutions (black contours) and
  the co-rotating streamfunction (blue) for (a) the end of the third
  branch, (b) the start of the fourth branch, (c) the end of the fourth
  branch, and (d) the start of the fifth branch.}
\end{center}
\end{figure}

Key properties of the two-fold limiting states for $\EE=0.01$ are
provided in Table 1.

\begin{table}
\begin{center}
\begin{tabular}{| c || r | r | r |}
  \toprule
  Branch & $\Omega$\hspace{0.6cm} & $\pi/2J$\hspace{0.3cm} & $16E/\pi$\hspace{0.3cm} \\
  \midrule
  1a & $0.250000$ & $1.000000$ & $ 1.000000$ \\
  1b & $0.106827$ & $0.317246$ & $-0.456852$ \\
  2a & $0.159754$ & $0.471293$ & $ 0.110399$ \\
  2b & $0.076297$ & $0.195486$ & $-1.231673$ \\
  3a & $0.107035$ & $0.272772$ & $-0.692599$ \\
  3b & $0.059980$ & $0.142097$ & $-1.777402$ \\
  4a & $0.080209$ & $0.191364$ & $-1.268668$ \\
  4b & $0.049782$ & $0.113129$ & $-2.181805$ \\
  5a & $0.064135$ & $0.147404$ & $-1.714846$ \\
  5b & $0.042700$ & $0.094682$ & $-2.504112$ \\
  \bottomrule
\end{tabular}
\caption{Key properties of the (near) limiting states for $\EE=0.01$ for
  the first 5 branches shown in figure 1.  In the first column, `a'
  denotes the start of a branch while `b' denotes the end of it.}
\label{table1}
\end{center}
\end{table}

\subsection{3-fold vortex patch equilibria}
\label{subsec:num3fold}

We next turn to 3-fold symmetric vortex patch equilibria, first
studied by Deem and Zabusky in 1978 \cite{11} for the Euler equations
($\EE=0$).  Here we discuss the structure of the solution
branches for a wide range of $\EE$.\\

\begin{figure}
\begin{center}
\includegraphics[height = 0.475\textwidth]{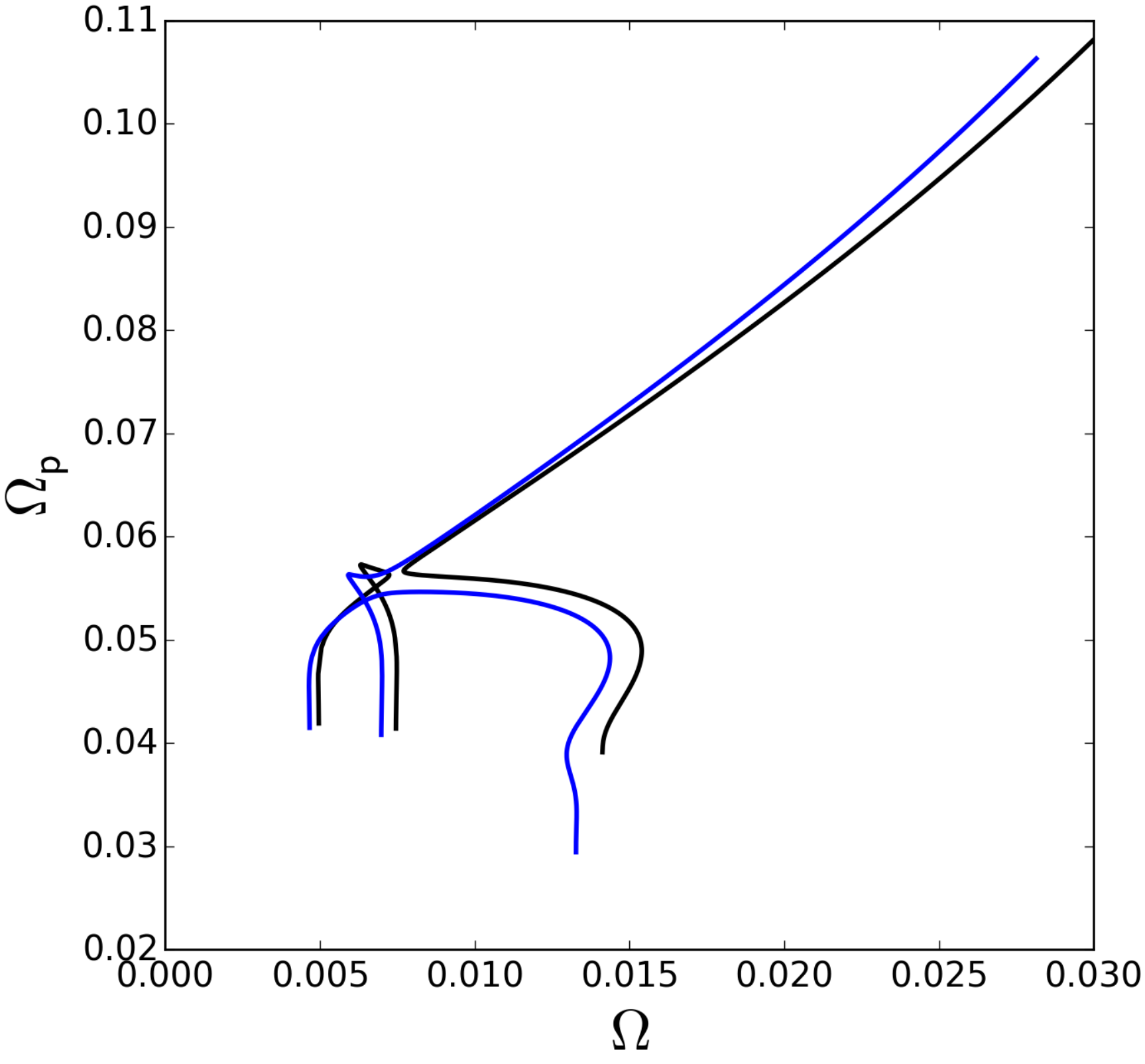}
\label{fig:3structzoom}
\caption{Solution branch structure for 3-fold vortex patch equilibria
  when $\EE=3.5$ (black) and $\EE=3.6$ (blue).  The branches start at
  the circular vortex patch solution in the upper right portion of the
  graph.}
\end{center}
\end{figure}

The most surprising result is that there is a disconnected branch of
solutions, not terminating at either end at the circular Rankine
vortex.  This was discovered by increasing $\EE$ and finding a change
in the topology of the limiting states between $\EE=3.5$ and $3.6$.
This is associated with a bifurcation in the structure of the solution
branches, as shown in 
figure 5.
For $\EE=3.5$ (black curves), the branch starting from the circular
patch in the upper right corner reaches a minimum in $\Omega$ at
$\Omega=0.007704627$, then increases and finally decreases approaching
the limiting state. (In fact $\der\Omega/\der\Omega_p$ likely changes
sign an infinite number of times before reaching the limiting state
at $\Omega_p=0$.)  This limiting state is triangular (albeit with
curved sides), and has the same form as found for the Euler equations
($\EE=0$) in \cite{11} (see below).
For $\EE=3.6$ (blue curves), the branch starting from the circular
patch in the upper right corner also reaches a minimum in $\Omega$,
but then $\Omega$ increases and limits to a significantly smaller
value as $\Omega_p\to 0$.  This limiting state resembles three
petal-like vortices connected at a single point at the origin.
This state is also the limiting state of three identical co-rotating
vortex patches, first studies in \cite{D95}.
Having determined that there is a bifurcation in the solution branch
structure between $\EE=3.5$ and $3.6$, a new branch was discovered by
taking the near limiting solution for $\EE=3.6$ and gradually decreasing
$\EE$ to $\EE=3.5$, holding the angular impulse fixed ($\Omega$ must
be allowed to vary).  In this way, we could find a solution on the
black separated branch next to the blue one (the middle pair of
curves in the lower part of the figure with $\Omega\approx 0.007$).
Having found one solution, we could then continue in both directions
to find the limiting states on this separated branch. (The blue
separated branch can be found similarly by jumping from the black
separated branch at small $\Omega$.)
One of these limiting states is the petal-like 
state just described.  The other, at the smallest $\Omega$, is a new
state consisting of a triangular central vortex attached to three
petals --- a four vortex state.  These limiting solutions are
illustrated for $\EE=3.6$ in
figure 6.\\

\begin{figure}
\begin{center}
\begin{tabular}{C{0.316\textwidth} C{0.316\textwidth} C{0.316\textwidth}} \\
(a) & (b) & (c) \\
\includegraphics[width = 0.316\textwidth]{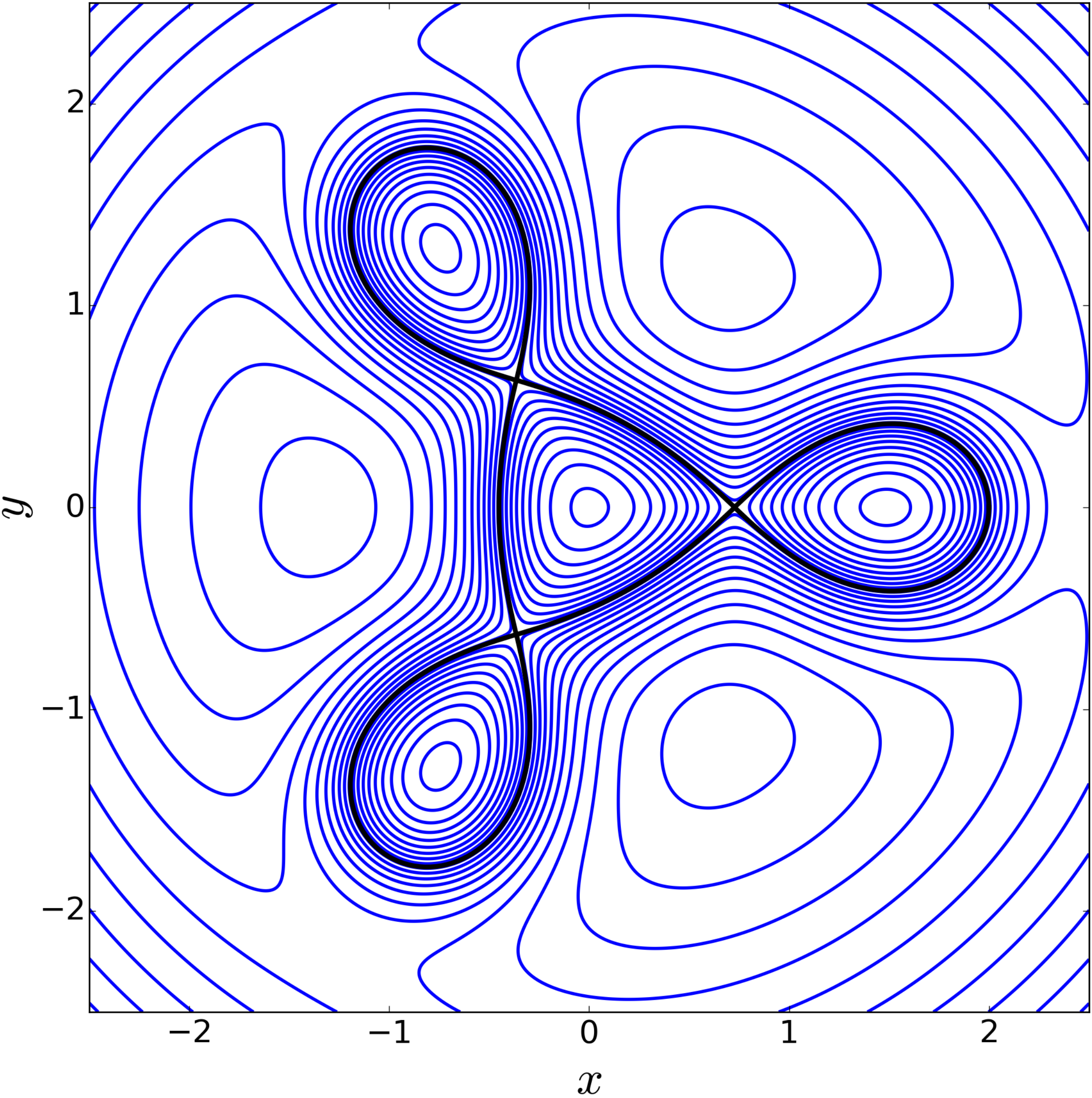} &
\includegraphics[width = 0.316\textwidth]{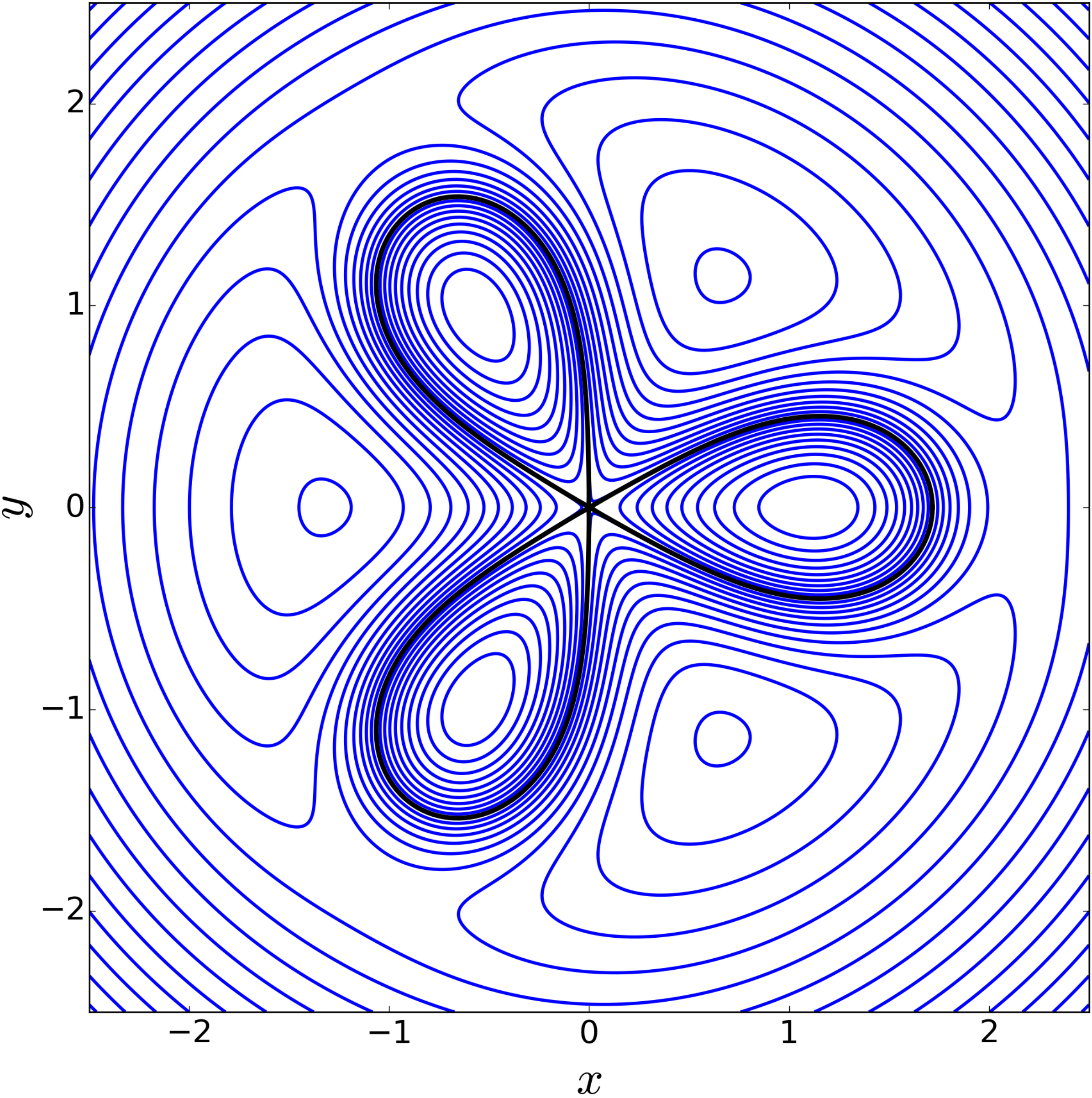} &
\includegraphics[width = 0.316\textwidth]{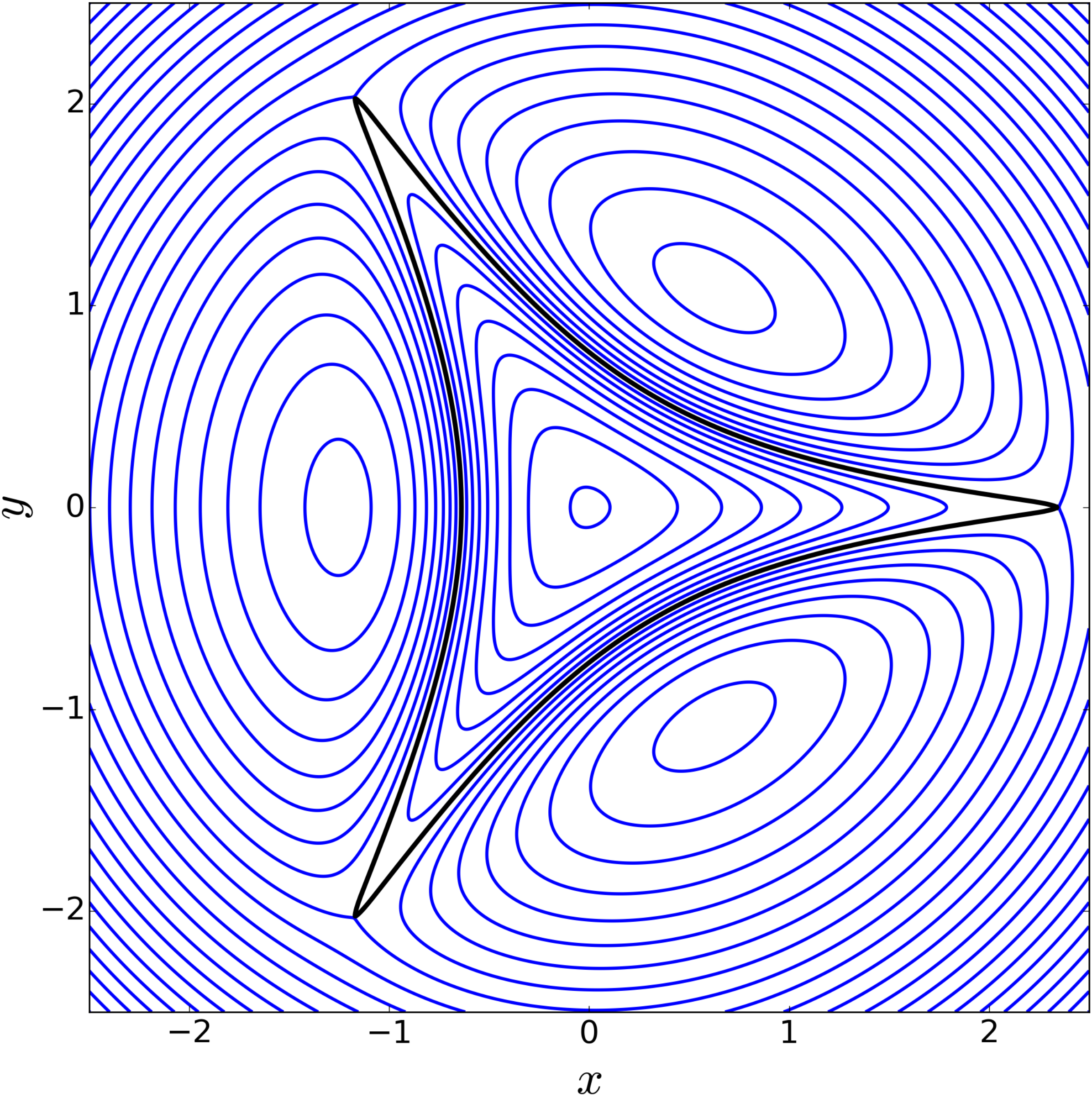}
\end{tabular}
\label{fig:3limit36}
\caption{Form of the (near) limiting solutions (black contours) and
  the co-rotating streamfunction (blue) for the state with (a) the
  smallest $\Omega$, (b) the intermediate value of $\Omega$, and
  (c) the largest $\Omega$.  Here $\EE=3.6$, corresponding to the
  blue curves in
  figure 5.}
\end{center}
\end{figure}

By continuing to jump branches from one value of $\EE$ to another,
we were able to determine that separated branches exist over a
wide range of $\EE$, and likely for all $\EE$.  As $\EE\to 0$, the
separated branch moves far from the primary branch stemming from the
circular patch, as shown in 
figure 7.
The limiting states are qualitatively similar to those for $\EE=3.6$
and are shown in 
figure 8.
Key properties of the three-fold limiting states for $\EE=0$ are
provided in Table 2.\\

\begin{table}
\begin{center}
\begin{tabular}{| r | r | r |}
  \toprule
  $\Omega$\hspace{0.6cm} & $\pi/2J$\hspace{0.3cm} & $16E/\pi$\hspace{0.3cm} \\
  \midrule
  $0.333333$ & $1.000000$ & $ 1.000000$ \\
  $0.301234$ & $0.885055$ & $ 0.835043$ \\
  $0.122420$ & $0.290564$ & $-1.002261$ \\
  $0.113192$ & $0.272230$ & $-1.110795$ \\
  \bottomrule
\end{tabular}
\caption{Key properties of the three-fold (near) limiting states for $\EE=0$.}
\label{table2}
\end{center}
\end{table}

\begin{figure}
\begin{center}
\includegraphics[height = 0.475\textwidth]{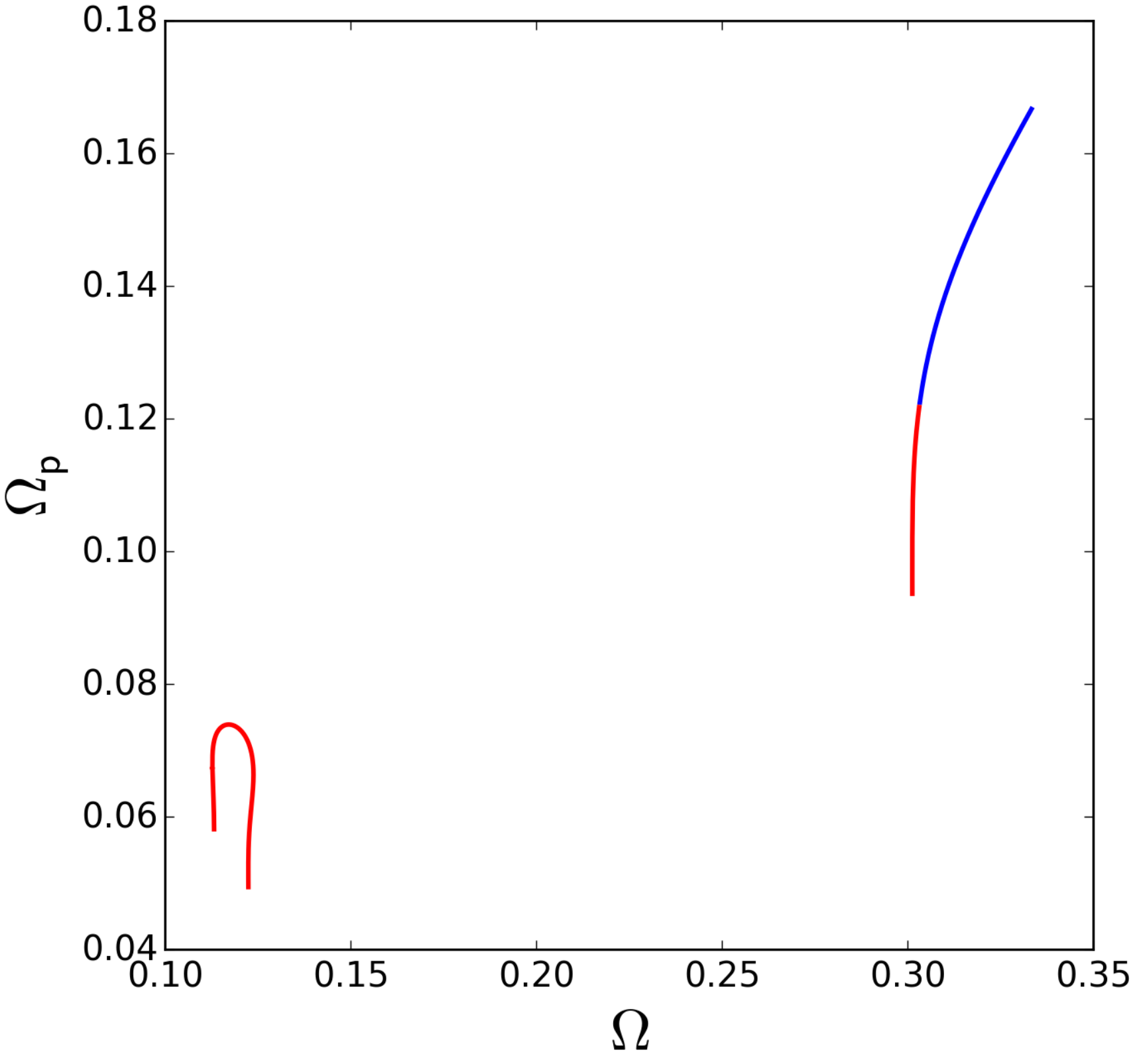}
\label{fig:3struct0}
\caption{Solution branch structure for 3-fold vortex patch equilibria
  when $\EE=0$ (the Euler equations).  The primary branch starts at
  the circular vortex patch solution in the upper right portion of the
  graph.  There, $\Omega=\tfrac13$ and $\Omega_p=\tfrac16$.  The
  separated branch is new.  The blue portion of the primary branch
  is linearly stable while the red portion and the entire separated
  branch is linearly unstable.}
\end{center}
\end{figure}

\begin{figure}
\begin{center}
\begin{tabular}{C{0.316\textwidth} C{0.316\textwidth} C{0.316\textwidth}} \\
(a) & (b) & (c) \\
\includegraphics[width = 0.316\textwidth]{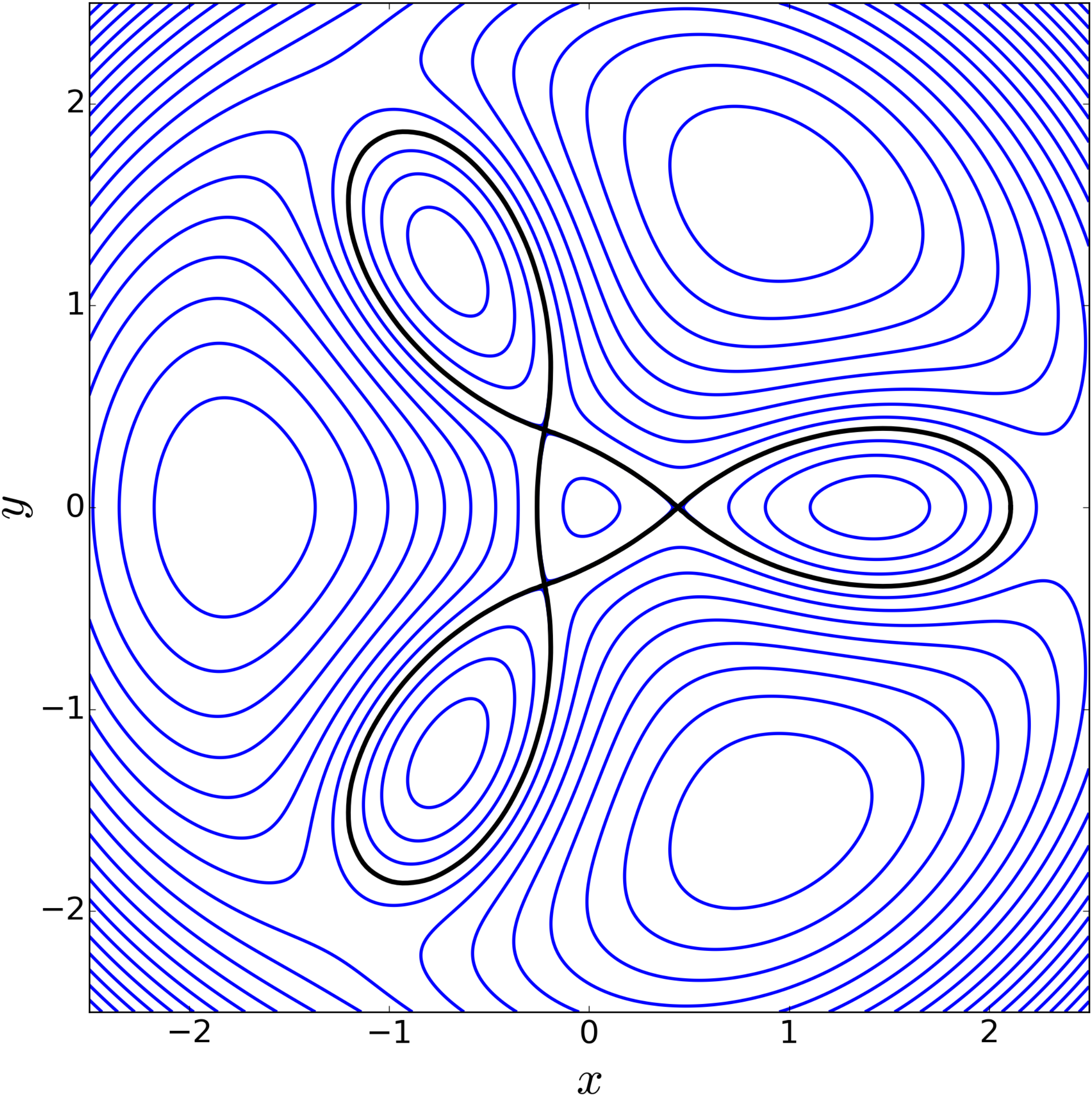} &
\includegraphics[width = 0.316\textwidth]{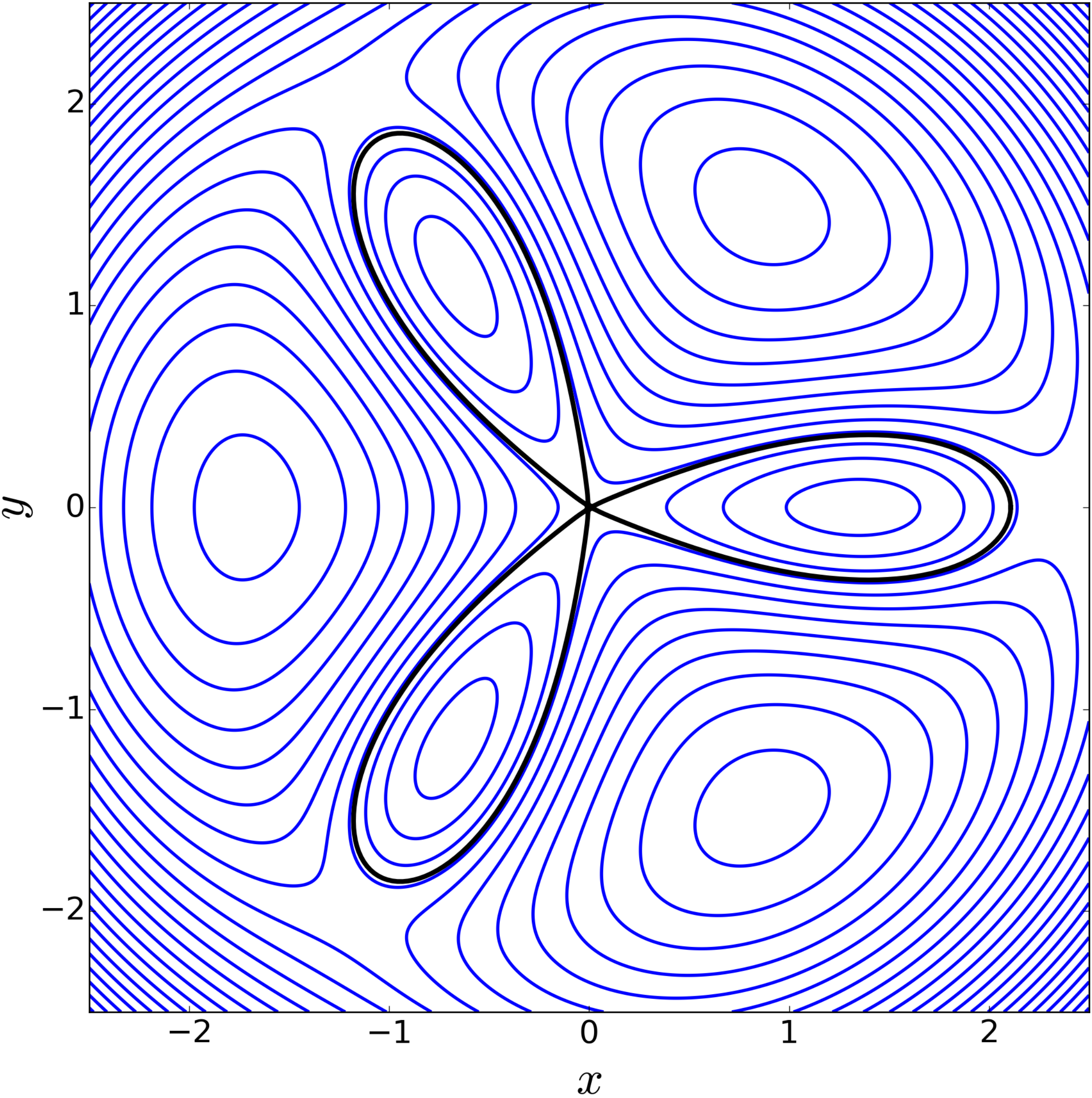} &
\includegraphics[width = 0.316\textwidth]{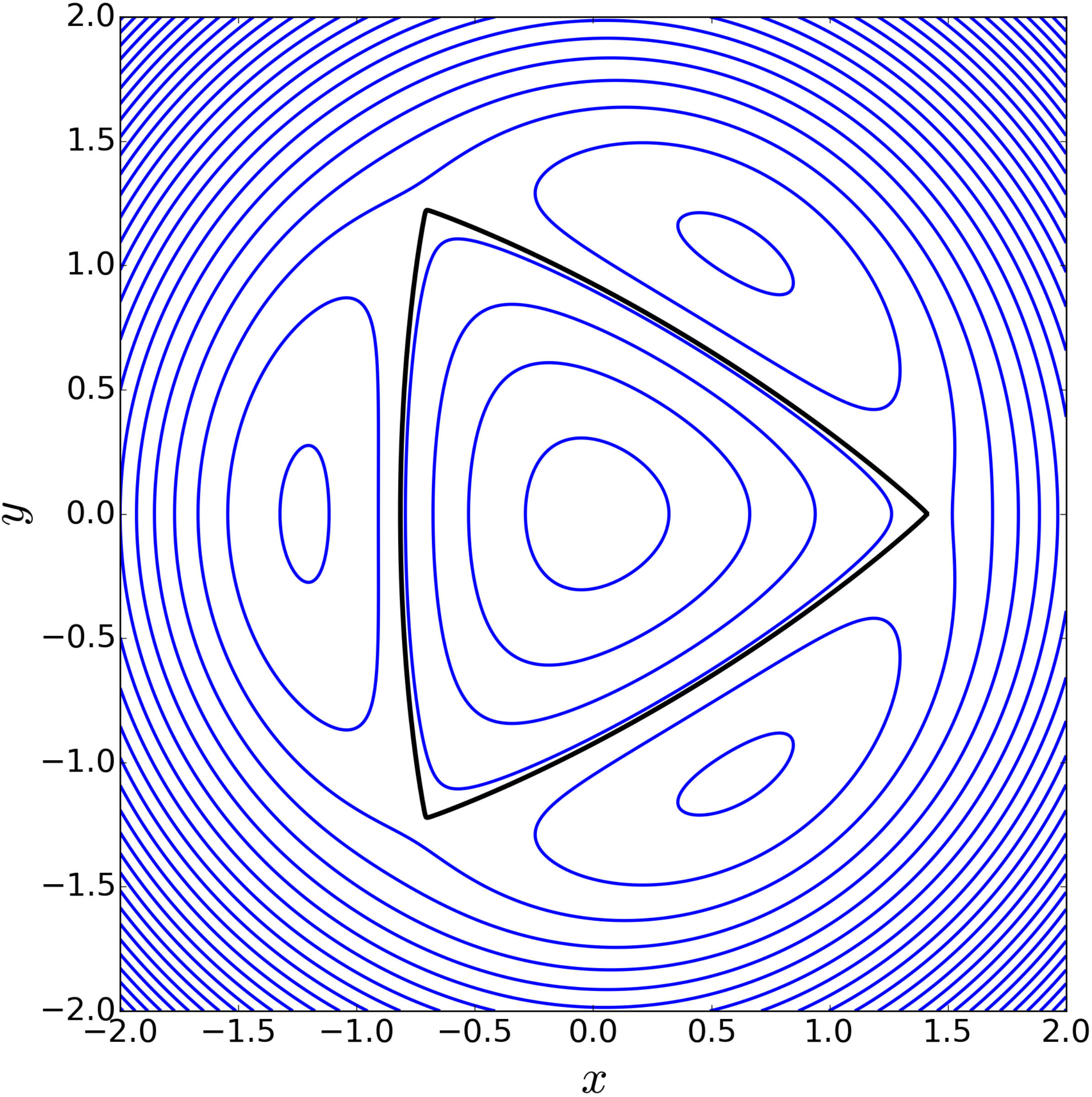}
\end{tabular}
\label{fig:3limit0}
\caption{Form of the (near) limiting solutions (black contours) and
  the co-rotating streamfunction (blue) for the state with (a) the
  smallest $\Omega$, (b) the intermediate value of $\Omega$, and (c)
  the largest $\Omega$.  Here $\EE=0$, corresponding to the Euler
  equations.  The state in (c) at the end of the primary branch of
  solutions was discovered in \cite{11}; the state in (b) was
  discovered in \cite{D95}, while that in (a) is new.}
\end{center}
\end{figure}

A more complete picture of the bifurcation structure of the three-fold
patch solutions for various values of $\EE$ is provided in
figure 9.
Each value of $\EE$ is seen to have a separated branch
which exists at values of $\Omega$ smaller than that of the limiting
circular patch solution.  In all cases, this separated branch is
linearly unstable.  Portions of the primary branch are also unstable,
though there is a window of stability for all $\EE>0$.  The nonlinear
evolution of the unstable states is deferred to another study, but the
instabilities often leave a time-dependent pulsating state.

\begin{figure}
\begin{center}
\includegraphics[height = 0.475\textwidth]{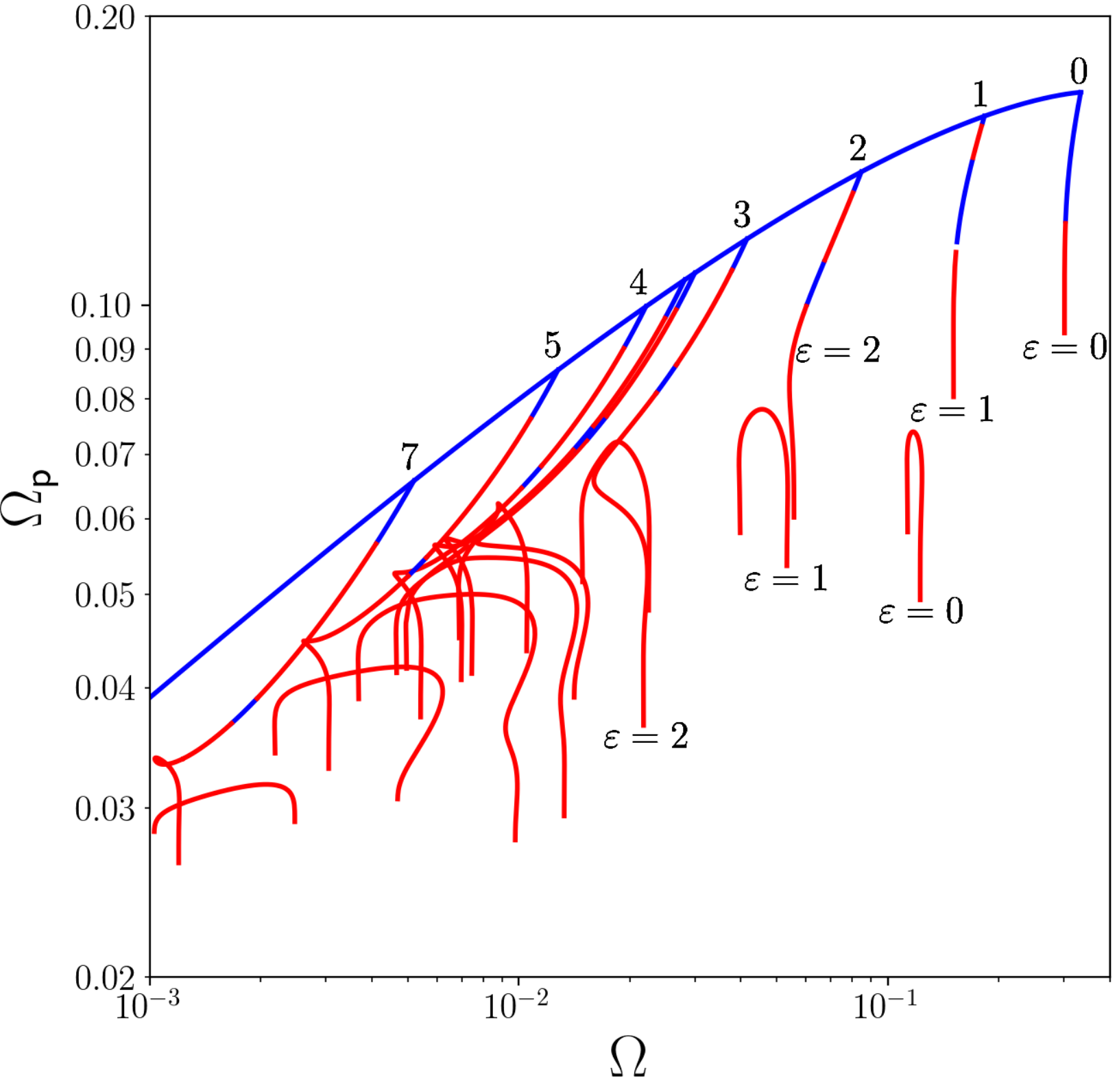}
\label{fig:3structsum}
\caption{Solution branch structure for 3-fold vortex patch equilibria
  for various $\EE$, as labelled.  Blue portions of the curves are
  linearly stable while red portions are linearly unstable, as
  determined by a full linear stability analysis (following \cite{D95}
  and \cite{pd12}).  The uppermost blue curve connects the limiting
  circular patch solutions continuously as a function of $\EE$.}
\end{center}
\end{figure}

%
\section{Tools used for the mathematical analysis}
\label{sec:tools}
The purpose of this section is to review and collect some technical
tools that are used throughout the remainder of this paper. We first
recall some simple facts about H\"older spaces on the unit
circle. Second, we discuss basic properties of modified Bessel
functions. Last we state the classical Crandall-Rabinowitz theorem and
give a generalized version with a parameter.

\subsection{Notation}
Here we introduce some notation that is used in the forthcoming sections.
 \begin{itemize}
 \item[$\bullet$] The unit disc  of the plane and its boundary will be denoted by $\D$ and $\T$, respectively.
 \item[$\bullet$] We denote by $C$ any positive constant that may change from line to line.
 \item[$\bullet$] For a given continuous function $f :\T \rightarrow \C $, we define its mean value by
 \[ \fint_{\T} f(\tau ) d\tau \triangleq \frac{1}{2i\pi} \int_{\T} f(\tau) d \tau, \]
 where $d\tau $ stands for the complex integration.
 \item[$\bullet$] Let  $X$ and $Y$ be two normed spaces. We denote by $\mathcal{L}(X,Y)$ the vector space of all the continuous linear maps  endowed with its usual strong topology.
 \item[$\bullet$] Let  $Y$ be a vector space and $R$ be a subspace, then $Y /R $ denotes the quotient space.
 \end{itemize}

\subsection{ Modified Bessel functions}\label{Bessel Function}
This section is devoted to some classical properties of Bessel
functions of imaginary argument. We start with the Bessel function
of the first kind and order $\nu$ given by the expansion
 $$
 J_\nu(z)=\sum_{m=0}^{+ \infty} \frac{(-1)^m\left(\frac{z}{2}\right)^{\nu+2m}}{m!\Gamma(\nu+m+1)}, \quad |\textnormal{arg} (z)|<\pi.
 $$
Note that this sum converges in a classical way provided that
$\Gamma(\nu+m+1)$ exists for any positive integer.  In addition, for
$\nu=n\in\Z$ it is known that Bessel functions admit the following
integral representation:
 $$
 \forall \, z\in\C,\quad  J_n(z)=\frac{1}{\pi}\int_0^{\pi}\cos( n\theta- z\sin\theta)d\theta.
 $$
Bessel functions of imaginary argument, denoted by $I_{\nu}$ and $K_{\nu}$,
are given by
  \[I_{\nu}(z)=\sum_{m=0}^{+ \infty} \frac{\left(\frac{z}{2}\right)^{\nu+2m}}{m!\Gamma(\nu+m+1)},  \quad |\textnormal{arg} (z)|<\pi
\]
and 
\[K_{\nu}(z)=\frac\pi2\frac{I_{-\nu}(z)-I_\nu(z)}{\sin(\nu \pi)},  \nu\in\mathbb{C}\backslash\mathbb{Z}\quad |\textnormal{arg} (z)|<\pi.
\]
However, for $\nu=n\in\Z$ we set  $\displaystyle{K_n(z)=\lim_{\nu\to n}K_\nu(z)}$. At this stage we recall 
useful expansions for $K_n$ that can be found for instance in \cite[p. 79-80]{Watson},

 \begin{equation}\label{Expand1} K_0(z)=-\log\bigg(\frac{z}{2}\bigg)I_0(z)+\sum_{m=0}^{\infty} \frac{(\frac{z}{2})^{2m}}{(m!)^2 }\psi(m+1),\quad K_0^\prime(z)=-K_1(z),
 \end{equation}
 where 
 \[\psi(1)=-\gamma\quad \hbox{and}\quad \forall  m\in \N^*,\quad \psi(m+1)=\sum_{k=1}^m\frac{1}{k}-\gamma.
 \]
 In addition, for $n\in \N^*$  
 \begin{eqnarray*}K_n(z)&=& (-1)^{n+1} \sum_{m=0}^{+ \infty} \frac{\left( \frac{z}{2} \right)^{n+2m}}{m!(n+m)!} \Bigg(\log\Big(\frac{z}{2}\Big)-\frac{1}{2}\psi(m+1)-\frac{1}{2}\psi(n+m+1) \Bigg)\\
 && + \frac{1}{2}\sum_{m=0}^{n-1} \frac{(-1)^m (n-m-1)!}{m! \left( \frac{z}{2}\right)^{n-2m}}.
 \end{eqnarray*}
 
 Another useful property is the positivity of $I_n$ and $K_n$. In fact, for any $n\in \N$ we have 
\begin{equation}\label{posit1}
 \forall x>0,\, I_n(x)>0\quad \hbox{and}\quad K_n(x)>0.
\end{equation}
The first one is obvious from the definition since each term in the
sum is strictly positive. As for the second one, it can be deduced
from the following integral representation found in e.g.\
\cite[p. 181]{Watson},
\begin{equation}\label{Form34}
K_\nu(z)=\int_0^{+\infty} e^{-z\cosh t}\cosh(\nu t) dt.
\end{equation}
Another useful identity found in \cite[p. 441]{Watson}
deals with Nicholson's integral representation of $I_n(z) K_n(z)$: for
$n\in \N$
 \begin{equation}\label{Nichols}
 I_n(z) K_n(z)=\frac{2(-1)^n}{\pi}\int_0^{\frac\pi2}K_0(2z\cos \theta) \cos(2n \theta) d\theta.
 \end{equation}

For the convenience of the reader, we next describe the relationship
between $K_0$ and the Green function associated with Helmholtz's operator
in two-dimensional space. Let $\varepsilon \in \R^*$ and consider in
the distribution sense the equation
 $$
( -\Delta+\varepsilon^2 )G_\EE=\delta_0, \textnormal{in}\quad \mathcal{S}^\prime(\R^2).
 $$
Then using a Fourier transform we obtain
 $$
 \widehat{G}_\EE(\xi)=\frac{1}{|\xi|^2+\varepsilon^2}, \quad\forall\, \xi\in\R^2.
 $$
Thus by a scaling argument, we have
 $$
 G_\EE(x)=G_1(\EE x), \quad\textnormal{with}\quad  G_1(x)=\frac{1}{4\pi^2}\int_{\R^2}\frac{e^{i x\cdot \xi}}{1+|\xi|^2}d\xi. 
 $$
Hence, a change of variables using  polar coordinates  yields
 \begin{eqnarray*}
 G_1(x)&=&\frac{1}{4\pi^2}\int_{0}^{+\infty}\frac{r}{1+r^2}\int_{0}^{2\pi}\cos(|x| r\cos\theta)d\theta dr\\
 &=&\frac{1}{2\pi}\int_{0}^{+\infty}\frac{r J_0(|x| r)}{1+r^2}dr\\
 &=&\frac{1}{2\pi}K_0(|x|),
 \end{eqnarray*}
where in the last line we have used an identity from
\cite[p. 425]{Watson}. As an application we show how to recover
the velocity from the domain of the patch in \eqref{sqg}. In fact, if
$D$ is a smooth bounded simply-connected domain and $q={\bf{1}}_D$,
then from the foregoing results the streamfunction $\psi,$ which is
the solution of the elliptic equation
 $$
 (\Delta-\EE^2)\psi={\bf{1}}_D 
 $$
is given explicitly by
 $$
 \psi(x)=-\frac{1}{2\pi}\int_{\R^2}K_0(|\EE | |x-y|) {\bf{1}}_D(y) dA(y)
 $$
where $dA$ denotes the planar Lebesgue measure. It follows that the
velocity induced by the patch $v=\nabla^\perp \psi$ takes the form
 
 \begin{equation}\label{BE38}
 v(x)=\frac{1}{2\pi}\int_{\partial D}K_0\big( |\EE | |x-\xi|\big)d\xi,
 \end{equation}
where the integration should be understood in the complex sense.

\subsection{Boundary equations}\label{Boundary equations}

In what follows we state the boundary equation of a rotating patch.
First, the initial data $q_0={\bf 1}_D$ generate a rotating patch
about the origin with uniform angular velocity $\Omega\in\R$ if
$$
q(t)={\bf 1}_{D_t} \quad \hbox{with}\quad D_t= e^{it \Omega} D.
$$
We may check that this is equivalent to 
$$
\big(v(x)-\Omega x^\perp\big)\cdot \vec{n}(x)=0,\quad \forall x\in \partial D
$$
with $\vec{n}(x)$ being the unit outward normal vector to the boundary
at the point $x$. The velocity $v$ induced by $q_0$ is given by
\eqref{BE38}.  Using complex notation we find that
 \[ \forall w \in \T \text{, } G(\varepsilon,\Omega,\Phi)(w)=0, \]
with 
\begin{equation*}
G(\varepsilon, \Omega,\Phi)(w)= \textnormal{Im} \left\lbrace  \Omega \Phi(w)\overline{\Phi'(w)}\overline{w}- \overline{\Phi'(w)} \overline{w} \fint_{\T} \Phi'(\tau) K_0\Big(|\EE|  \vert \Phi(w)- \Phi(\tau)\vert \Big) d\tau  \right \rbrace
\end{equation*}
and $\Phi: \mathbb{T}\to \C$ is at least a $C^1$ parametrization of
the boundary.  Actually, we may add a constant term in the kernel
$K_0$ without changing the equation. Thus according to the singularity
structure of $K_0$ near the origin detailed in \eqref{Expand1} the
suitable constant to add is $\log\Big(\frac{|\varepsilon|}{2}\Big)$.
Therefore

\begin{equation}\label{tildeG_j}
G(\varepsilon, \Omega,\Phi)(w)= \textnormal{Im} \left\lbrace  \Omega \Phi(w)\overline{\Phi'(w)}\overline{w}- \overline{\Phi'(w)} \overline{w} \fint_{\T} \Phi'(\tau) K_0^\EE\big( \vert \Phi(w)- \Phi(\tau)\vert \big) d\tau  \right \rbrace
\end{equation}
with
\[ K_0^\EE(x)\triangleq K_0(|\EE|  x)+\log\big({|\varepsilon|}/{2}\big).
\]
If we let $\EE\to 0$, then without surprise we get the vortex patch
equation associated with the Euler equations described for example in
\cite{19}:
\[G_E(\Omega, \Phi(w))= \hbox{Im}\bigg\{\Big( \Omega \overline{\Phi(w)}+\frac12 \fint_{\T} \frac{\overline{\Phi(\tau)}-\overline{\Phi(w)}}{\Phi(\tau)-\Phi(w)} \Phi'(\tau) d \tau \Big) w \Phi'(w)\bigg\}.\]
One may notice that
\begin{equation}\label{Eq71}
G(0,\Omega, \Phi(w))=- {G_E(\Omega, \Phi(w))}.
\end{equation}
Indeed, starting from the general formula
\[
\fint_{\partial D} \log|z-\xi|d\xi=-\frac12\overline{\fint_{\partial D}\frac{\overline\xi-\overline z}{\xi-z}d\xi}
\]
we find by a change of variables 
\begin{equation}\label{Flog} \fint_{\T} \log(\vert \Phi(w)-\Phi(\tau)\vert) \Phi'(\tau) d \tau =-\frac{1}{2} \overline{\int_{\T}\frac{\overline{\Phi(\tau)}-\overline{\Phi(w)}}{\Phi(\tau)-\Phi(w)} \Phi'(\tau) d \tau}. 
\end{equation}
Thus
\[G_E(\Omega, \Phi(w))= -\hbox{Im}\bigg\{\Big( \Omega {\Phi(w)}+\fint_{\T} \log(\vert \Phi(w)-\Phi(\tau)\vert) \Phi'(\tau) d \tau  \Big) \overline{w}\,\overline{ \Phi'(w)}\bigg\}.\]
It suffices now to use the expansion \eqref{Expand1} to deduce  that  
\[
\forall x\neq0,\quad \lim_{\EE\to 0}K_0^\EE(x)=-\log(x/2)
\]
and thus we find \eqref{Eq71}.

\section{Bifurcation to $m-$fold symmetric vortex patch equilibria}
\label{sec:mfold}

The main task of this section is to prove the existence of rotating
m-fold vortex patch (relative) equilibria, or `V-states' \cite{11},
for the QGSW model given by $(\ref{sqg})$. In the first section we
state our main result.  The proof is carried out in several steps and
is detailed in different sections.  The basic tool is the classical
Crandall-Rabinowitz's theorem, and for the study of the imperfect
bifurcation we need a slight generalization of this theorem.

\subsection{Main result}

We first state our principal result concerning the existence of a
countable family of bifurcating curves with m-fold symmetry from
Rankine (circular) vortices. More precisely, we obtain the following result.
\begin{theorem}\label{existence}
  Let $\EE\in\R$, then for each integer $m\geq 1$
  there exists a curve (or branch) of $m$-fold rotating vortex patches
  bifurcating from the unit disc at the angular velocity
\[ \Omega_m(\varepsilon)=I_1( |\varepsilon | )K_1( |\varepsilon | )-I_m( | \varepsilon |)K_m(| \varepsilon |).\]
Moreover, the existence is uniform for vanishing $\varepsilon$. More
precisely, there exists $a>0$ and continuous functions $\varphi:
(-a,a)^2 \rightarrow \R$, $\psi: (-a,a)^2 \rightarrow C^{1+
  \alpha}(\T)$ satisfying
$$\varphi(0,0)=\Omega_m(0)=\frac{m-1}{2m},\quad \psi(0,0)=0
$$ such that 
\[\Omega=\varphi(\EE,s),\quad \psi(\varepsilon,s,w)=\sum_{n\geq2}a_{nm-1}(\varepsilon,s) \overline{w}^{nm-1} \]
and 
\[ G\big(\varepsilon, \varphi(\varepsilon,s), w+s\overline{w}^{m-1} + s \psi(\varepsilon,s,w)\big)=0, \, \forall\, (\EE,s)\in (-a,a)^2,\quad \forall w\in \T. \]
Recall that the function $G$ defining the vortex patch equation is given
by \eqref{tildeG_j}.
\end{theorem}
\begin{remark}
On the one hand, the case $m=1$ is trivial. It corresponds simply to
the translation of the unit disc, and with the notation of the theorem
we have $\psi(\EE,s,w)=0$. On the other hand, the regularity of the
boundary is not at all optimal; like for the Euler equations we guess
that it must be analytic, see \cite{4,20}.

\end{remark}
\begin{remark}
It is known that for the Euler equations the bifurcation diagram of
simply connected vortex patches is organized around Rankine vortices
through a countable collection of pitchfork curves (one for each
symmetry).  Theorem \ref{existence} shows that locally this structure
is preserved for any perturbation size of $\EE$ and therefore there is
no symmetry breakdown. This is not the case however for the
bifurcation diagram close to Kirchhoff's ellipse, as discussed below in
Section \ref{sec:2fold}.
\end{remark}
The proof of Theorem \ref{existence} is a consequence of the materials
developed in next sections. Actually, the first part of
Theorem \ref{existence} follows from Theorem \ref{C-R0}, Proposition
\ref{string11}, and \mbox{Proposition \ref{assumptions}}.  However for
the second part dealing with the stability of Eulerian branches under
small perturbation on $\EE$, we need to make use of Theorem \ref{C-R}.

\subsection{Crandall-Rabinowitz's Theorem with a parameter}
The main objective of this section is to formulate suitable conditions
for the bifurcation from the trivial solutions of a general nonlinear
equation of the type
$$
F(\EE,\lambda,x)=0, \quad F:\R\times\R\times V\to Y
$$
with $Y$ a Banach space and $V$ a neighborhood of \,$0$ in some Banach
space $X$. We assume that $F$ is smooth enough and
$$
\forall \lambda,F(0,\lambda,0)=0.
$$
The starting point is that for $\EE=0$ we know the structure of the
bifurcation diagram near the trivial solutions and it is of interest
to understand how its geometric structure varies with respect to an
arbitrary perturbation in $\EE$.  This is called an imperfect
bifurcation. This subject is well developed in the literature starting
with the pioneering work of Golubitsky and Schaeffer \cite{Glob}, who
classify the bifurcation diagram in a general setting using tools from
the theory of singularities.  Various particular studies related to
the present study have been carried out over the last few decades, and
one may consult for instance \cite{Liu,Shi} and the references
therein. In what follows we formulate some results dealing with
imperfect bifurcations with symmetry persistence. This phenomenon
occurs especially when the trivial solutions do not vary with respect
to the parameter $\EE$, that is
$$
\forall\lambda,\,\forall \EE,\,F(\EE,\lambda,0)=0.
$$
For the steadily-rotating vortex patch solutions of \eqref{sqg}, these
are precisely the suitable conditions enabling a detailed study of the
bifurcations from the unit disc. Next we recall the
classical theorem of Crandall-Rabinowitz \cite{CR} concerning 
bifurcations from trivial solutions. This will be applied to get the
first part of Theorem \ref{existence} when $\EE$ is fixed at an
arbitrary value.
\begin{theorem}\label{C-R0} Let $X, Y$ be two Banach spaces,
  $V$ be a neighborhood of \,$0$ in $X$ and let 
$$
F : \R \times V \to Y
$$
with the following  properties:
\begin{enumerate}
\item $F (\lambda, 0) = 0$ for any $\lambda\in \R$.
\item The partial derivatives $F_\lambda$, $F_x$ and $F_{\lambda x}$ exist and are continuous.
\item $\textnormal{Ker}(\partial_xF(0, 0))=\langle x_0 \rangle$ and $Y/R(\partial_xF(0, 0))$ are one-dimensional. 
\item {\it Transversality assumption}: $\partial_\lambda\partial_xF(0, 0)x_0 \not\in R(\partial_xF(0, 0))$.
\end{enumerate}
If $\mathcal{X}$ is any complement of $\hbox{Ker }(\partial_xF(0, 0))$ in $X$, then there is a neighborhood $U$ of $(0,0)$ in $\R \times X$, an interval $(-a,a)$, and continuous functions $\psi: (-a,a) \to \R$, $\phi: (-a,a) \to Z$ such that $\psi(0) = 0$, $\phi(0) = 0$ and
$$
\big\{(\lambda,x)\in U,\, F(\lambda,x)=0\big\} =\Big\{\big(\psi(s), s x_0+s\phi(s)\big)\,;\,\vert s\vert<a\Big\}\cup\Big\{(\lambda,0)\,;\, (\lambda,0)\in U\Big\}.
$$
\end{theorem}

The next result deals with a slight generalization of the preceding
Crandall-Rabinowitz's theorem to include a parameter. This allows us to treat
the stability of the bifurcation diagram under a small
perturbation. This will be the cornerstone of the proof of the second
part of \mbox{Theorem \ref{existence}}.
\begin{theorem}\label{C-R} 
Let $X, Y$ be two Banach spaces, $V$ a neighbourhood of $0$ in $X$ and let 
$$
F : (-1,1)\times \R \times V \to Y
$$
be  a function of class $C^1$
with the following  properties:
\begin{enumerate}
\item $F (\EE,\lambda, 0) = 0$ for any $\EE\in\R$ and  $\lambda\in\R$.
\item The partial derivatives $F_{\EE}$, $F_\lambda$, $F_x$ and $F_{\lambda x}$ exist and are continuous.
\item $\textnormal{Ker}(\partial_xF(0, 0,0))=\langle x_0 \rangle$ and $Y/R(\partial_xF(0, 0,0))$ are one-dimensional. 
\item {\it Transversality assumption}: $\partial_\lambda\partial_xF(0,0, 0)x_0 \not\in R(\partial_xF(0,0, 0))$.\end{enumerate}
If $\mathcal{X}$ is any complement of $\hbox{Ker }(\partial_xF(0, 0,0))$ in $X$, then there is a neighborhood  $U$ of $(0,0,0)$, an interval $(-a,a)$, with $a>0$, and continuous functions 
$$\psi: (-a,a)^2 \to \R,\quad \phi: (-a,a)^2 \to \mathcal{X}
$$ such that $\varphi(0,0) = 0$, $\psi(0,0) = 0$ and
$$
\Big\{(\EE,\lambda,x)\in U,\, F(\EE,\lambda,x)=0\Big\}=\Big\{\big(\varepsilon, \psi(\varepsilon,s), s x_0+s \phi(\varepsilon,s)\big)\,;\,|\varepsilon|,\vert s\vert<a\Big\}\cup\Big\{(\varepsilon,\lambda,0)\,;\, (\varepsilon,\lambda,0)\in U\Big\}.
$$
\end{theorem}

\begin{proof}
The proof is a simple adaptation of \cite{CR}.  Let $\langle w_0\rangle$
be a complement of
$\mathcal{Y}\triangleq\mathcal{R}(\partial_xF(0, 0,0))$ in $Y$. Then
$$
X=\langle x_0\rangle \oplus\mathcal{X} \quad \textnormal{ and }\quad
Y=\langle w_0\rangle \oplus\mathcal{Y}.
$$

Consider  the projection $P: X\mapsto\langle x_0\rangle$ on $\langle x_0\rangle$ along $\mathcal{X}$ given by  
$$
x=s x_0+ z,\, z\in \mathcal X\Longrightarrow  Px=  sx_0
$$
and similarly define the projection   $Q:  Y\mapsto \langle w_0\rangle$ on $\langle w_0\rangle$ along $\mathcal{Y}$.
   Then the equation $F(\varepsilon,\lambda, x)=0$  is equivalent to the system
$$
F_1(\EE,\lambda,s, z)\triangleq(\hbox{Id}-Q)F(\EE,\lambda,s x_0+z)=0\quad \hbox{and}\quad  QF(\EE,\lambda, sx_0+z)=0.
$$
It is clear that for some $\eta>0$, the function $$F_1:(-1,1)\times\R\times(-\eta,\eta) \times \mathcal{U}\to\mathcal{Y}_m$$  is $C^1$ with $\mathcal{U}$ a small neighbourhood of $0$ in $\mathcal{X}$ . Moreover, it is not difficult to check that 
\begin{equation*}
\partial_z F_1(0,0,0,0)=(\hbox{Id}-Q)\partial_xF(0,0,0):\mathcal{X}\to \mathcal{Y}
\end{equation*}
is an isomorphism.
By the implicit function theorem, the solutions of the equation $F_1(\EE,\lambda,s,z)=0$ are described near the point $(0,0,0,0)$ by the parametrization $z=\varphi(\EE,\lambda, s)$ with 
$$\varphi:(-\delta ,\delta)^3\to\mathcal{X}, \quad \delta>0
$$ 
being a $C^1$ function. Therefore we obtain
\begin{equation}\label{zeroXX}
(\hbox{Id}-Q)F\big(\EE,\lambda,s x_0+\varphi(\EE,\lambda,s )\big)=0, \forall |\EE| ,|\lambda|, |s|<\delta.
\end{equation}
Consequently, solving  
 the equation  $F(\varepsilon,\lambda, x)=0$ close to $(0,0,0)$ is equivalent to  
\begin{equation*}
 F_2(\EE,\lambda,s)\triangleq QF\big(\EE, \lambda, s x_0+\varphi(\EE,\lambda,s )\big)=0, \quad \forall\,  |\EE| , |\lambda|,|s|<\delta.
\end{equation*}
Using  the assumption $F(\EE, \lambda,0)=0$, for any $\EE\in\R$ and
$\lambda\in\R$, one deduces by uniqueness  that
\begin{equation}\label{zero1}
\varphi(\EE,\lambda,0)=0,\quad\forall | \EE | , |\lambda|<\delta.
\end{equation}
Now differentiating equation \eqref{zeroXX} with respect to $s$ we get
$$
\partial_s\Big((\hbox{Id}-Q)F\big(\EE,\lambda,s x_0+\varphi(\EE,\lambda,s )\big)\Big)=0, \forall |\EE| ,|\lambda|, |s|<\delta.
$$
In particular we deduce for $s=0$
$$
(\hbox{Id}-Q)\partial_xF\big(\EE,\lambda,\varphi(\EE,\lambda,0)\big)\big(x_0+\partial_s\varphi(\EE,\lambda,0)\big)=0,
$$
which implies, in view of \eqref{zero1},  that
$$
\partial_xF\big(0,0,0\big)\big(\partial_s\varphi(0,0,0)\big)=0.
$$
This gives $\partial_s\varphi(0,0,0) \in \langle x_0\rangle $, but from the definition one has  $\partial_s\varphi(0,0,0) \in \mathcal{X}$ and consequently
\begin{equation}\label{Eq11}
\partial_s\varphi(0,0,0) =0.
\end{equation}
Hence there exists  a continuous function $\varphi_1: (-\delta,\delta)^3\to \mathcal{X}$ such that
$$
\varphi(\varepsilon,\lambda,s)= s\varphi_1(\varepsilon,\lambda,s) \textnormal{ and } \varphi_1(0,0,0)=0.
$$
Set
\begin{equation}
g(\EE,\lambda,s)\triangleq\left\{ \begin{array}{ll}
 QF\big(\EE, \lambda, s x_0+\varphi(\EE,\lambda,s )\big)/s, \quad s\neq0\\
 Q\partial_xF\big(\EE, \lambda,0)[x_0+ \partial_s \varphi(\varepsilon,\lambda,0)],\quad s=0.\end{array} \right.
\end{equation}
Note that $g$ is continuous and 
$$
g(0,0,0)= Q\partial_xF(0,0,0)x_0=0.
$$ 
Moreover, thanks to $\eqref{zero1}$ one may easily check that
\[ \partial_{\lambda} \varphi(\varepsilon, \lambda,0)=0, \, \forall \vert \varepsilon \vert, \vert \lambda \vert < \delta.\]
Consequently,  the partial derivative $\partial_\lambda g$ exists, it is continuous and satisfies
$$
\partial_\lambda g(0,0,0)=Q\partial_\lambda \partial_xF(0,0,0)x_0.
$$
From the transversality assumption we find
$$
\partial_\lambda g(0,0,0)\neq0.
$$
Hence we can use a weak version of the implicit function theorem, see Appendix A in \cite{CR}, and 
 thus find that the solutions of $g(\EE,\lambda,s)=0$ near the origin are parametrized by a $C^1$ surface  $\psi:(-a,a)^2\to \R$ such that $\lambda=\psi(\EE,s)$ and
$$
g\big(\EE,\psi(\EE,s),s \big)=0, \forall |\EE| , |s|<a, a>0.
$$
Therefore the non-trivial solutions  of the equation $F(\EE,\lambda,x)=0$ near the origin  are parametrized by 
$$
\lambda=\psi(\EE,s), \quad x=s x_0+s\varphi_1\big(\EE, \psi(\EE,s), s\big)\triangleq sx_0+s{\phi}(\EE,s),\quad \forall  |s|, |\EE| <a.
$$ 
This completes the proof of the desired result.
\end{proof}

\subsection{Function spaces I}\label{FS2}
In this section we introduce  the function spaces used below in studying the bifurcation from the unit disc.  For $\alpha\in (0,1)$, we set
 \[X= \Big\{f \in C^{1+\alpha}(\T), \textnormal{s.t.} \, \forall w\in \T,\,f(w)=\sum_{n=0}^{+ \infty} f_n \overline{w}^{n}, f_n \in \R \Big\} \]
and
 \[Y= \Big\{g \in C^{\alpha}(\T), \textnormal{s.t.} \, \forall w\in \T,\, g(w)=\sum_{n=1}^{+ \infty} g_ne_{n}(w), g_n \in \R \Big\} \textnormal{, with } \quad e_n(w)= \textnormal{Im}(w^n). \]
As discussed below, the  $m$-fold symmetric vortex patch solutions are
essentially arising from bifurcations in the more restrictive function
spaces 
  \[X_m= \Big\{ f \in X, \textnormal{s.t.} \, \forall w\in \T, \,f(w)=\sum_{n=1}^{+ \infty} f_{nm-1} \overline{w}^{nm-1}\Big\}\]
  and
 \[Y_m= \Big\{ g \in Y,\textnormal{s.t.} \, \forall w\in \T,\, g(w)=\sum_{n=1}^{+ \infty} g_ne_{nm}(w)\Big\}. \]
Of course, the spaces $X$ and $X_m$ are equipped  with  the strong topology of $C^{1+\alpha}$ whereas $Y$ and $Y_m$ are equipped with   the strong topology of $C^{\alpha}$.

Next we recall the following lemma (see e.g.\ \cite{15,20}).
 \begin{lemma}\label{LemSing}
Let $\mathbb{\Delta}=\big\{ (w, w), w \in \T \big\}$ and let $K: \T \times \T \setminus\mathbb{\Delta} \mapsto \C $ be a measurable function with the following properties. There exists $C>0$ such that,

 \[\vert K(w, \tau) \vert \leq C,\quad  \forall (w,\tau)\notin \mathbb\Delta\]
and that for each $\tau \in \T$, the function $w\in \T\setminus\{\tau\} \mapsto K(w,\tau)$ is differentiable and
\[\Big\vert \partial_w K(w,\tau)  \Big\vert \leq  \frac{C}{\vert w - \tau\vert }.\]
Then the operator
\[T\varphi(w)=\int_{\T} K(w,\tau) \varphi(\tau) d\tau,\]
sends $C^\alpha(\T)$ to $L^\infty(\T)$ for any $\alpha\in (0,1)$  with 
\[ \Vert T\varphi \Vert_{\alpha} \leq C_{\alpha} C\Vert \varphi \Vert_{\infty} \textnormal{, } \varphi \in L^{\infty}(\T),\]
where $C_{\alpha}$ depends only on $\alpha$.

\end{lemma}

\subsection{Regularity of the functional I}
The main goal of this section is to study the regularity properties
required by Theorem \ref{C-R0} and Theorem \ref{C-R} for the
functional $G$ introduced in \eqref{tildeG_j}.  Denote by $B_r$ the
ball of center $\hbox{Id}$ and radius $r$ in the space $X$ and $B_r^m$
the same ball in the \mbox{space $X_m$}.

\begin{proposition}\label{string11}
There exists  $r \in (0,1)$  such that for  any $\alpha\in (0,1)$   the following holds true.
\begin{enumerate}
\item $G: \R\times\R\times B_r\to Y$ is of class $C^1$. It is at least of class $C^3$. 
\item The restriction $G: \R\times\R\times B_r^m\to Y_m$  is well-defined.
\end{enumerate}
\end{proposition}
\begin{proof}
The proofs are classical and  can be performed in a similar way to those of  \cite{15,20}, using Lemma \ref{LemSing} in particular. Some  details will be given later in the subsection \ref{RF2}. Thus we only sketch the proof of the symmetry given in point $(2)$. 
The spaces that used are described in subsection \ref{FS2}. Recall that 
\[ B_r^m=\lbrace \Phi \in X_m , \Vert \Phi-\textnormal{Id}\Vert_{C^{1+\alpha}(\T)} \leq r \rbrace,\quad  \Phi(w)=w+\sum_{n=1}^{+ \infty} f_{nm-1} \overline{w}^{nm-1}\]
 and
 \[G(\varepsilon,\Omega,\Phi)=\textnormal{Im} \bigg\{  \Big(\Omega \Phi(w)- I(\EE,\Phi)(w) \Big) \overline{\Phi'(w)}\overline{w} \bigg\}\]

with
\[ I(\varepsilon,\Phi)(w)= \fint_{\T} \Phi'(\tau) K_0\Big(|\varepsilon| \vert \Phi(w)- \Phi(\tau)\vert \Big)d\tau. \]
We begin by checking that $G(\varepsilon,\Omega,f)$ belongs to $Y_m$. It is enough for that purpose to prove that

\[G(\varepsilon,\Omega,\Phi)\bigg(e^{\frac{2i \pi}{m}} w\bigg)=G(\varepsilon,\Omega,\Phi)( w), \; \forall w \in \T. \]
Note that 
\begin{equation}\label{derivphisymm}
\Phi\bigg(e^{\frac{2i \pi}{m}} w\bigg)= e^{\frac{2i \pi}{m}}\Phi( w),\quad \Phi'\bigg(e^{\frac{2i \pi}{m}} w\bigg)=\Phi'( w)
\end{equation}
and thus  the property is obvious for the first term $\textnormal{Im} \left\lbrace  \Omega \Phi(w)\overline{\Phi'(w)} \overline{w} \right\rbrace$. For the last term of $G$, it is enough to check the identity,
\[  I(\varepsilon,\Phi)\bigg(e^{\frac{2i \pi}{m}} w\bigg)=e^{\frac{2i \pi}{m}} I(\varepsilon,\Phi)( w), \, \forall w \in \T . \]

This follows simply by making the change of variable $\tau=e^{\frac{2i \pi}{m}} \xi$:
\begin{align*}
 I(\varepsilon,\Phi)\bigg(e^{\frac{2i \pi}{m}} w\bigg) &=\fint_{\T} e^{\frac{2i \pi}{m}}  \Phi'\Big(e^{\frac{2i \pi}{m}} \tau\Big) K_0 \Big(| \varepsilon| \Big|\Phi\Big( e^{\frac{2 i \pi}{m}} w\Big)- \Phi\Big(e^{\frac{2i \pi}{m}} \tau\Big) \Big| \Big)d\tau  \\
  &=e^{\frac{2i \pi}{m}} \fint_{\T} \Phi'(\tau)  K_0 \Big(|\varepsilon| \vert \Phi(w)- \Phi(\tau)\vert \Big)d\tau \\
  &=e^{\frac{2i \pi}{m}} I(\varepsilon, \Phi)( w).
\end{align*}
This ends the proof.
 \end{proof}

\subsection{Spectral study}
In this section we compute the linearized operator at the
trivial solution of the functional $G$ introduced in
\eqref{tildeG_j}. We prove that it acts as a Fourier multiplier
with symbol related to modified Bessel functions. This allows us to
describe the full range of $\Omega$ corresponding to non-trivial
kernels. Finally, we check that for these values of $\Omega$ all the
assumptions of Crandall-Rabinowitz's theorem are satisfied.

\subsubsection{Structure of the linearized operator}
We prove the following result. 
\begin{proposition}\label{propLin}
Let  $ \displaystyle{h: w\mapsto \sum_{n=0}^{+ \infty} a_n \overline{w}^n\in X}$, then  
\[D_fG(\varepsilon, \Omega, \textnormal{Id})(h)(w)=\sum_{n=0}^{+ \infty} a_n(n+1) \Big(\Omega_{n+1}(\EE) -\Omega\Big) e_{n+1}(w), \textnormal{ with }e_n(w)=\textnormal{Im}(w^n) \]
and
\[ \Omega_m(\varepsilon)=I_1(| \varepsilon |)K_1(| \varepsilon |)-I_m(| \varepsilon |)K_m(| \varepsilon |).\]
\end{proposition}
\begin{proof}
Without loss of generality, we may assume that $\EE>0$. Now, for given $h\in X$, one may deduce from straightforward computations that
 \begin{equation}
D_fG( \varepsilon, \Omega, \hbox{Id})(h)(w)= \mathcal{L}_0(h)(w)+\mathcal{L}_1(h)(w)+\mathcal{L}_2(h)(w)
\end{equation}
with
\[ \mathcal{L}_0(h)(w) = \Omega\,\textnormal{Im} \left\lbrace  h(w) \overline{w}+ \overline{h'(w)} \right \rbrace,\]
\begin{eqnarray}
\nonumber \mathcal{L}_1(h)(w) =\textnormal{Im} \Bigg\{- \overline{h'(w)} \overline{w} \fint_{\T}  K_0 \big(\varepsilon \vert w- \tau\vert\big) d\tau - \overline{w} \fint_{\T} h'(\tau) K_0 \big(\varepsilon  \vert w- \tau\vert \big)d\tau\Bigg\}
\end{eqnarray}
and
\[ \mathcal{L}_2(h)(w) =\EE  \,\textnormal{Im} \Bigg\{- \overline{w} \fint_{\T}  \frac{ \textnormal{Re}\Big(\big( h(w)-h(\tau)\big) \big( \overline{w}- \overline{\tau}\big) \Big)}{\vert w- \tau \vert} K_0' \big(\varepsilon \vert w- \tau \vert \big) d \tau \Bigg\}.\]

We begin with the easier term $\mathcal{L}_0(h)(w)$ whose computation is straightforward: 
\begin{equation}\label{Ni3}\mathcal{L}_0(h)(w)=-\sum_{n=0}^{+ \infty} a_n \Omega (n+1)e_{n+1}.
\end{equation}
For  $\mathcal{L}_1(h)$ we first use the change of variable $\tau\mapsto w \tau$
$$
 \mathcal{L}_1(h)(w) =\textnormal{Im} \left\lbrace- \overline{h'(w)} \fint_{\T}  K_0 \big(\varepsilon \vert 1- \tau\vert\big) d\tau -  \fint_{\T} h'(\tau w) K_0 \big(\varepsilon  \vert 1- \tau\vert \big)d\tau\right \rbrace
$$
which implies that
$$
\mathcal{L}_1(h)(w)=-\sum_{n=1}^{+ \infty} na_n  \Bigg[\fint_{\T}  K_0 \big(\varepsilon  \vert 1- \tau\vert \big)(\overline{\tau}^{n+1}-1) d\tau \Bigg] e_{n+1}. 
$$
We focus on the integral term involving in $\mathcal{L}_1(h)(w)$. By symmetry arguments we obtain
\begin{eqnarray*}\fint_{\T}  K_0 \big(\varepsilon \vert 1- \tau\vert\big)(\overline{\tau}^{n+1}-1) d\tau&=& \frac{1}{2\pi}  \int_0^{ 2\pi} K_0 \big(2 \varepsilon   \sin({\theta}/{2})\big)\big(\cos(n\theta)-\cos\theta\big)d\theta\\
&=& \frac{2}{\pi}  \int_0^{ \frac\pi2} K_0 \big(2\varepsilon  \sin{\theta}\big)\big(\cos(2n\theta)-\cos(2\theta)\big)d\theta\\
&=& \frac{2}{\pi}  \int_0^{ \frac\pi2} K_0 \big(2\varepsilon \cos{\theta}\big)\big((-1)^n\cos(2n\theta)+\cos(2\theta)\big)d\theta.\\
\end{eqnarray*}
Using \eqref{Nichols} we deduce that
$$
\fint_{\T}  K_0 \big( \varepsilon  \vert 1- \tau\vert\big)(\overline{\tau}^{n+1}-1) d\tau=I_n(\EE ) K_n(\EE )-I_1(\EE )K_1( \EE ).
$$
Therefore
\begin{equation}\label{Nie1}
\mathcal{L}_1(h)(w)=\sum_{n=1}^{+ \infty} na_n \Big(I_1( \EE )K_1( \EE )-I_n( \EE ) K_n(  \EE )\Big) e_{n+1}.
\end{equation}

For the computation of  $\mathcal{L}_2(h)(w)$, we  write
\[\mathcal{L}_2(h)(w)= - \sum_{n=1}^{+ \infty} \frac{  \varepsilon  a_n}{2}  \Bigg( \fint_{\T}  \bigg[\frac{ (\tau^n-1)(\tau-1)}{\vert 1 - \tau \vert}-\frac{ (\overline{\tau}^n-1)(\overline{\tau}-1)}{\vert 1 - \tau \vert}\bigg]K_0' \big(\varepsilon \vert 1- \tau \vert \big) d \tau \Bigg) e_{n+1}.\]

Now we compute   the  following integral term  which is more delicate
$$
d_n\triangleq\frac{ \EE }{ 2}\fint_{\T}  \bigg[\frac{ (\tau^n-1)(\tau-1)}{\vert 1 - \tau \vert}-\frac{ (\overline{\tau}^n-1)(\overline{\tau}-1)}{\vert 1 - \tau \vert}\bigg]K_0' \big( \varepsilon  \vert 1- \tau \vert \big) d \tau. 
$$
First we use the following trigonometric identity: for $\tau=e^{i\theta}, \theta\in[0,2\pi]$, one has
$$
 \textnormal{Re}\Bigg\{\bigg(\frac{ (\tau^n-1)(\tau-1)}{\vert 1 - \tau \vert}-\frac{ (\overline{\tau}^n-1)(\overline{\tau}-1)}{\vert 1 - \tau \vert}\bigg)\frac{d\tau}{2i\pi}\Bigg\}=\frac1\pi\cos(\theta/2)\Big(\sin\theta+\sin(n\theta)-\sin\big((n+1)\theta\big)\Big) d\theta.
$$ 
Thus  integration by parts yields
\begin{eqnarray*}
 d_n&=&\frac{1}{2\pi}\int_0^{2\pi}\big(\EE  \cos(\theta/2)\big)K_0^\prime\big(2\EE \sin(\theta/2)\big)\Big(\sin\theta+\sin(n\theta)-\sin\big((n+1)\theta\big)\Big)d\theta\\
 &=&-\frac{1}{2\pi}\int_0^{2\pi}K_0\big(2\EE \sin(\theta/2)\big)\Big(\cos\theta+n\cos(n\theta)-(n+1)\cos\big((n+1)\theta\big)\Big)d\theta.
\end{eqnarray*}

Performing a change of variables and invoking symmetry arguments  imply
\begin{eqnarray*}
 d_n &=&-\frac{2}{\pi}\int_0^{\frac\pi2}K_0\big(2\EE \sin(\theta)\big)\Big(\cos(2\theta)+n\cos(2n\theta)-(n+1)\cos\big(2(n+1)\theta\big)\Big)d\theta\\
 &=&\frac{2}{\pi}\int_0^{\frac\pi2}K_0\big(2\EE \cos(\theta)\big)\Big(\cos(2\theta)-n(-1)^n\cos(2n\theta)-(n+1)(-1)^n\cos\big(2(n+1)\theta\big)\Big)d\theta.
\end{eqnarray*}
Using \eqref{Nichols} we obtain
\begin{equation}
d_n=-I_1(\EE ) K_1(\EE )-nI_n(\EE ) K_n(\EE )+(n+1)I_{n+1}(\EE )K_{n+1}(\EE ).
\end{equation}
Combined with \eqref{Nie1} we find that
$$
\mathcal{L}_1(h)(w)+\mathcal{L}_2(h)(w)=  \sum_{n=0}^{+ \infty} (n+1) a_{n}\Big( I_1(\EE ) K_1(\EE )-I_{n+1}(\EE )K_{n+1}(\EE )\Big)e_{n+1}.
$$
Putting together this identity with \eqref{Ni3} gives the desired result.

\end{proof}
\subsubsection{Bifurcation assumptions}

Next we check the assumptions on the linearized operator
required by Theorem \ref{C-R0} and Theorem \ref{C-R}.
For this purpose, we introduce the  countable dispersion set
\begin{equation}\label{Disp}
 \mathbb{S}=\Big\{ \Omega_m(\EE)\triangleq I_1(|\varepsilon| )K_1(|\varepsilon |)-I_{m}(| \varepsilon |) K_m(| \varepsilon | ),\quad m\geq1\Big\}.
 \end{equation}
The main result reads as follows.
\begin{proposition}\label{assumptions}
Let $\EE \in \R $ be a fixed  real number and $G$ be the functional defined in \eqref{tildeG_j}; note that some of its properties are detailed in Proposition \ref{string11}. Then the following assertions hold.
\begin{enumerate}
\item The sequence $m\mapsto \Omega_m(\EE)$ is strictly increasing and converges to $I_1( |\EE | ) K_1( | \EE |)$.
\item The kernel of $D_fG\big(\varepsilon, \Omega,\textnormal{Id}\big)$ is non-trivial if and only if $\Omega=\Omega_m(\EE) \in \mathbb{S}$. In this case, it is one-dimensional and generated by
\[v_m:w\in\T\mapsto \overline{w}^{m-1}.\]
\item The range of $D_f G\big(\varepsilon,\Omega_m(\EE),\textnormal{Id}\big)$ is closed in Y and is of co-dimension one. It is given by
 \[\mathcal{R}(D_fG\big(\varepsilon, \Omega_m(\EE),\textnormal{Id})\big)=\Big\{ g \in C^{\alpha}(\T), g=\sum_{\underset{n=1}{n \neq m}}^{+ \infty} g_ne_{n}, \,g_n \in \R \Big\}.\]
\item Transversality assumption:
\[ \partial_{\Omega} D_fG(\varepsilon, \Omega_m,\textnormal{Id})v_m \notin \mathcal{R}(D_fG(\varepsilon,\Omega_m,Id)).\]

\end{enumerate}
\end{proposition}

\begin{proof}
$(1)$
 We use  the following inequality (see \cite{Segura}). For $\nu \geq 0$ and $x>0$
\begin{equation}\label{product}
\frac{I_{\nu+ \frac{1}{2}}(x)}{I_{\nu- \frac{1}{2}}(x)} < \frac{x}{ \nu + \sqrt{\nu^2 + x^2}} \leq \frac{K_{\nu - \frac{1}{2}}(x)}{K_{\nu+ \frac{1}{2}}(x)}.
\end{equation}
Thus using the positivity of $I_n$ and $K_n$ mentioned in
\eqref{posit1}, we find that the sequence $n\mapsto I_n(|\EE | ) K_n(|\EE |)$
is strictly decreasing. It remains to check that
$\displaystyle{\lim_{n\to\infty} I_n(|\EE |) K_n(| \EE | )=0}$. For this,
we establish a precise result on the convergence rate used below:
there exists $C>0$ such that for any real number
$\EE$ ,
\begin{equation}\label{decay1}
\forall n\in\N^\star ,\quad 0< I_n(|\EE|  ) K_n( | \EE | )\le C \frac{\ln(n+1)}{n}.
\end{equation}
Indeed, using integration by parts in \eqref{Nichols} we find
\[
I_n(|\EE | ) K_n(| \EE | )=-\frac{2(-1)^n\EE }{\pi n}\int_0^{\frac\pi2}\sin\theta\, K_1\big(2 | \EE | \cos\theta\big)\sin(2n\theta) d\theta.
\]
Thus
\[
0<I_n(| \EE | ) K_n(| \EE | )\le \frac{2 |\EE| }{\pi n}\int_0^{\frac\pi2}\sin\theta\, K_1\big(2| \EE | \cos\theta\big)|\sin(2n\theta)| d\theta.
\]
On the other hand using \eqref{Form34} we deduce by the change of variable  $\theta=\cosh t$ that for $x>0$
\begin{align*}
K_1(x)&=\int_{1}^{+\infty} e^{-x\theta}\frac{\theta}{\sqrt{\theta^2+1}}d\theta\\
&=\int_{1}^{2} e^{-x\theta}\frac{\theta}{\sqrt{\theta^2+1}}d\theta+\int_{2}^{+\infty} e^{-x\theta}\frac{\theta}{\sqrt{\theta^2+1}}d\theta\\
&\le e^{-x}+\int_{2}^{+\infty} e^{-x\theta}d\theta\\
&\le  e^{-x}+\frac{1}{x}e^{-2x}.
\end{align*}
Consequently there exists $C>0$ such that for any $x>0$,
$$
K_1(x)\le \frac{C}{x}
$$
which implies after straightforward computations related to Dirichlet kernel,
\begin{align*}
\forall n\in \N^\star,\quad I_n(\EE ) K_n(\EE )&\le \frac{C}{n}\int_0^{\frac\pi2}\frac{|\sin(2n\theta)|}{\cos\theta} d\theta\\
&\le  \frac{C}{n}\int_0^{\frac\pi2}\frac{|\sin(2n\theta)|}{\sin\theta} d\theta\\
&\le  C\frac{\ln(n+1)}{n}.
\end{align*}
$(2)$ The result follows from the structure of the linearized operator stated in Proposition \ref{propLin} and the strict monotonicity of the eigenvalues $(\Omega_m(\EE))_{m\geq1}$.

$(3)$ 
We want to prove that for any $m \geq 1$ the range of $D_fG(\varepsilon, \Omega_m(\EE),\textnormal{Id})$ coincides with
\[ Z_m \triangleq \Big\{ g \in C^{\alpha}(\T), g(w)=\sum_{\underset{n=1}{n \neq m}}^{+ \infty} g_ne_{n}, g_n \in \R \Big\}.\]
As $Z_m$ is closed in Y and of co-dimension one, it is enough to check that the range is $Z_m$.
First, it is obvious that 
\[ \mathcal{R}(D_f G(\varepsilon,\Omega_m(\EE),\textnormal{Id}) \subset Z_m\]
and it thus remains to check the reverse inclusion. Let  $\displaystyle{g=\sum_{n\geq1} g_n e_n\in Z_m}$; we want to find $h \in X$ such that
\[D_fG(\varepsilon,\Omega_m(\EE),\hbox{Id})(h)=g.\]
Set  $\displaystyle{h(w)=\sum_{n\geq0}h_n \overline{w}^n}$, then  the equation 
\[D_fG(\varepsilon,\Omega_m(\EE),\textnormal{Id})h=g\]
admits an explicit solution such that 
\[ h_{n}=\frac{g_{n+1}}{(n+1)(\Omega_{n+1}(\EE)-\Omega_{m}(\EE))}, \quad n\neq m-1 \]
and 
\[h_{m-1}=0.\]

We next check that $h \in C^{1+ \alpha}(\T)$.
Since 
\[h(w)=\sum_{\underset{n\geq0}{n \neq m-1}} \frac{g_{n+1}}{(n+1)(\Omega_{n+1}(\EE)-\Omega_m(\EE))} \overline{w}^n\]
then  it follows from Cauchy-Schwarz inequality and the Bessel identity  that
\begin{align*}
\|h\|_{L^\infty(\T)}&\leq C_0\sum_{n\geq1}\frac{|g_{n+1}|}{n+1}\\
&\leq C\|g\|_{L^2(\T)}\\
&\leq  C\|g\|_{C^{\alpha}(\T)}
\end{align*}
where $C_0$ is the inverse of the  distance between $\Omega_m(\EE)$ and $\mathbb{S}\backslash\{\Omega_m(\EE)\}$.
$C_0$ is finite due to the monotonicity of the eigenvalues.
We now prove that the derivative $h'$ belongs to  $C^{\alpha}$. It is obvious that
\[h'(w)=\sum_{\underset{n=1}{n \neq m-1}}^{+ \infty} \frac{ng_{n+1}}{(n+1)(\Omega_{m}(\EE)-\Omega_{n+1}(\EE))} \overline{w}^{n+1},\]
which can be split as follows
\begin{align*}
h'(w)&=\sum_{\underset{n=2}{n \neq m}}^{+ \infty} \frac{g_{n}}{\Omega_{m}(\EE)-\Omega_{n}(\EE)} \overline{w}^{n}+\sum_{\underset{n=2}{n \neq m}}^{+ \infty} \frac{g_{n}}{n(\Omega_n(\EE)-\Omega_{m}(\EE))} \overline{w}^{n}\\
=&-\sum_{\underset{n=2}{n \neq m}}^{+ \infty} \frac{g_{n}}{K_m(\EE )I_m(\EE )} \overline{w}^{n} -\sum_{\underset{n=2}{n \neq m}}^{+ \infty} g_{n} \left[\frac{1}{\Omega_n(\EE)-\Omega_m(\EE)}-\frac{1}{K_m(|\EE | )I_m(| \EE | )} \right] \overline{w}^{n}\\
&+\sum_{\underset{n=2}{n \neq m}}^{+ \infty} \frac{g_{n}}{nK_m(| \EE |)I_m(| \EE  |)} \overline{w}^{n}+\sum_{\underset{n=2}{n \neq m}}^{+ \infty} \frac{g_{n}}{n} \left[\frac{1}{\Omega_n( \EE |)-\Omega_m(\EE)}-\frac{1}{K_m(| \EE | )I_m(| \EE | )} \right] \overline{w}^{n}.
\end{align*}
Set
\[\chi(w)=\sum_{\underset{n=2}{n \neq m}}^{+ \infty} g_{n} \overline{w}^n,\quad H_1(w)=\sum_{\underset{n=2}{n \neq m}}^{+ \infty}\left[\frac{1}{\Omega_n(\EE)-\Omega_m(\EE)}-\frac{1}{K_m(| \EE |)I_m(| \EE |)} \right] \overline{w}^{n},\quad H_2(w)=\sum_{\underset{n=2}{n \neq m}}^{+ \infty}\frac{\overline{w}^n}{n}\]
and
\[H_3(w)=\sum_{\underset{n=2}{n \neq m}}^{+ \infty}\frac{1}{n}\left[\frac{1}{\Omega_n(\EE)-\Omega_m(\EE)}-\frac{1}{K_m(| \EE |)I_m(| \EE| )} \right] \overline{w}^{n}.\]
Then
\[h'(w)=-\frac{1}{K_m(|\EE| )I_m(| \EE |)} \chi(w) - \chi \ast H_1(w)+\frac{1}{K_m(| \EE | )I_m(| \EE| )} \chi \ast H_2(w)+ \chi \ast H_3(w).\] 
As $\chi(w)=\Pi_+(2ig(w))-2ig(w)$, with $\Pi_+$ being the Szeg\"o projection that  sends  continuously   $C^\alpha(\T)$ to itself, we deduce that $\chi\in C^\alpha(\T)$. Hence in order to ensure $h^\prime\in  C^\alpha(\T)$ it is enough to prove that $H_j\in L^1(\T), j\in\{1,2,3\}$.
 Let us start with $H_1$. It is obvious that
\[\left\vert \frac{1}{\Omega_n(\EE)-\Omega_m(\EE)}-\frac{1}{K_m(| \EE | )I_m(| \EE| )} \right\vert =\left\vert \frac{K_n(| \EE | )I_n( | \EE | )}{K_m(| \EE | )I_m( |\EE | ) \big( K_m(| \EE |)I_m(l \EE |)-K_n(| \EE | )I_n(| \EE | )\big)} \right\vert. \]
Hence using \eqref{decay1} we find  a constant $C$ depending on $m$ and $\EE$ such that for any $n\neq m$ 
\begin{align*}
\left\vert \frac{1}{\Omega_n(\EE)-\Omega_m(\EE)}-\frac{1}{K_m(|\varepsilon |)I_m( |\varepsilon |)}  \right\vert & \leq C \,K_n(| \varepsilon |)I_n(| \varepsilon |)\\
&\le C\frac{\ln(n+1)}{n}.
\end{align*}
According to the Parseval identity, this proves  that $H_1\in L^2(\T)$  and by the usual embedding we find $H_1\in L^1(\T)$. It is simple to check that $H_2, H_3$ belong to $L^2(\T)$ and so to $L^1(\T)$ which completes the desired result.

$(4)$ For the transversality assumption, it is obvious that for any $h\in X$
\[ \partial_\Omega D_{f }G\big(\varepsilon,\Omega_m(\EE),\hbox{Id}\big) h= \textnormal{Im} \left\lbrace h(w)\overline{w}+\overline{h'(w)} \right\rbrace.\]
Therefore, for $v_m(w)=\overline{w}^{m-1}$ 
\[ \partial_\Omega D_{f }G\big(\varepsilon,\Omega_m(\EE),\hbox{Id}\big) v_m=-m e_m \notin \mathcal{R}(D_fG \big(\varepsilon,\Omega_m(\EE),\textnormal{Id})\big)\]
and consequently the transversality condition is verified.

\end{proof}
%

\section{Imperfect bifurcation close to the branch of Kirchhoff ellipses }
\label{sec:2fold}
This section is devoted to the study of the global structure of the
two-fold branch. According to \mbox{Theorem \ref{existence}} we know there
exists
a local branch close to Rankine vortices that bifurcates at the
point $\Omega_2(\EE)$. For $\EE=0$ the full branch is explicitly
described by Kirchhoff ellipses, and according to \cite{4,19,Kam,Luz}
we known that from this branch a countable family of bifurcating
curves emerges at the Love instability points \cite{Love}.  Notice
that these new curves model alternating one/two-fold V-states and the
two-fold V-states are characterized by an odd frequency perturbation
of the conformal mapping of the ellipse $w\in\T\mapsto
w+Q\overline{w}$.\\

We investigate below the `imperfect'
bifurcation, that is, the behavior of the solution branch structure subject to
a small perturbation in $\EE$. We prove that the scenarios of
persistence/breakdown symmetry occur simultaneously close to the
Kirchhoff ellipse branch.  Indeed, we prove by using perturbation theory,
see Theorem \ref{thmF2}, that far from the second bifurcating
point
the local structure of the two-fold branch persists and varies
continuously with respect to a small perturbation in $\EE$.  However
around the singularity set the issue depends on the symmetry of the
V-states. In fact, we show in Theorem \ref{thmF2} that the diagram
structure around the one-fold bifurcating curves is not destroyed and
is similar to the Euler one. However, and this is only proved for the
$m=4$ Love instability point, the symmetry is broken down around the
first bifurcating curve of the two-fold V-states (see Theorem
\ref{thmbreak}). This is a kind of resonance phenomenon between the two
branches with the same symmetry leading to a separation of the
singularity and a loss of the connectedness. Numerically, in Section
\ref{sec:num2} 
this behavior is observed for the first    two-fold branches emerging from
the ellipse, but from an analytical standpoint the problem is
difficult due to the cumbersome computations required for higher elliptical azimuthal wavenumbers.

\subsection{Function spaces II}\label{FS1}
We first introduce the function spaces suitable for studying
the bifurcation from the two-fold branch. We draw attention to the
fact that we use the same notation as in the Section \ref{sec:mfold}
dealing with the $m$-folds structure but with a different meaning.  For
$\alpha\in (0,1)$, we set
\begin{equation}\label{spaceX}
 X= \bigg\{f \in C^{1+\alpha}(\T), f(w)=\sum_{n=2}^{+ \infty} f_n w^{n}, f_n \in \R \bigg\}
 \end{equation}
 and
 \begin{equation}\label{spaceY}
 Y=\bigg\{g \in C^{\alpha}(\T), g(w)=\sum_{n=1}^{+ \infty} g_ne_{n}(w), g_n \in \R \bigg\} \textnormal{, with }\quad e_n(w)= \textnormal{Im}(w^n). 
 \end{equation}
 
\subsection{Summary of the bifurcations from Kirchhoff ellipses}
The results of this section were obtained in \cite{19} and for the
commodity of the presentation we briefly  recall them.
Since ellipses are explicit rotating solutions for the Euler equations,
then from \cite{19} one finds that
\[  G\big(0, {\scriptstyle{\frac{1-Q^2}{4}}}, \alpha_Q\big)=0,\quad  \forall Q \in (0,1) \, \]
with  $\alpha_Q \colon w\in \T\mapsto w+ {Q}\overline{w}$ being  the conformal parametrization of the ellipse centered at the origin  and with semi-axes $1\pm Q$ and $Q\in[0,1)$. Notice that Kirchhoff discovered that such ellipses
rotate at the angular velocity  $\frac{1-Q^2}{4}$.
Introducing  
\begin{equation}\label{Eq8}
F(\varepsilon, Q ,f)=G\big(\varepsilon,{\scriptstyle{\frac{1-Q^2}{4}}}, \alpha_Q+f\big),\quad f\in X
\end{equation}
where  the space $X$ is described in  \eqref{spaceX},
it is plain that
\[F(0, Q ,0)=0,\quad  \forall Q \in [0,1).\]
From \eqref{Eq71} one obtains
\begin{equation}\label{diff1}
\mathcal{L}_Q\triangleq D_f F(0,Q,0)=-D_f G_E(\frac{1-Q^2}{2}, \alpha_Q).
\end{equation}
Let   $m\geq3$ be an integer and denote by $Q_m$ the unique solution in $[0,1)$ of the equation
\begin{equation}\label{equaQ4}
 1+Q^m-{\scriptstyle{\frac{1-Q^2}{2}}} m=0,
\end{equation}
and set 
$$
\mathcal{S}\triangleq \big\{Q_m,\,  m\geq 3\big\}.
$$
At various points in the argument,
we need to distinguish between the following  two subsets of $\mathcal{S}$:
\begin{equation}\label{Resset}
 \mathcal{S}_{\textnormal{reso}}\triangleq \big\{Q_{2m},\,  m\geq 2\big\}\quad \hbox{and}\quad \mathcal{S}_{\textnormal{Nreso}}\triangleq \big\{Q_{2m+1},\,  m\geq 1\big\}.
\end{equation}
The first one is called the `resonant set' and the second is the
`non-resonant set'.
Note that  from \cite{19} we know that the sequence $(Q_m)_{m\geq3}$ is strictly increasing with 
$$
\lim_{m\to+\infty} Q_m=1.
$$
The following result dealing with the structure of the linearized operator $\mathcal{L}_Q$ was proved \mbox{in \cite{19}}. This was used to prove the
existence of bifurcations from Kirchhoff ellipses using the Crandall-Rabinowitz theorem.
\begin{proposition}\label{prop-spec11}
Let  $X$  and $Y$ be the spaces  introduced in \eqref{spaceX} and \eqref{spaceY}. Then the following assertions hold true.
\begin{enumerate}
\item Let $h(w)=\displaystyle{\sum_{n\geq2} a_n w^n\in X}$, then
$$
\mathcal{L}_Qh=\frac{1}{2} \sum_{n\geq1} g_{n+1}e_n; \quad e_n(w)=\textnormal{Im}(w^n),
$$
with
\begin{eqnarray*}
g_2&=&\frac12(1+Q)^2 a_2,\\
 g_3&=&2Q^2a_3,\\
g_{n+1}&=& \Big(1+Q^n-\frac{1-Q^2}{2}n\Big)\big(a_{n+1}-Q a_{n-1}\big), \quad \forall n\geq 3.
\end{eqnarray*}
\item The kernel of $\mathcal{L}_Q$ is non-trivial if and only if $Q=Q_m\in \mathcal{S}$ and it is a one-dimensional vector space generated by
$$
v_m(w)=\frac{w^{m+1}}{1-Qw^2}.
$$ 
 \item The range of $\mathcal{L}_Q$ is of co-dimension one in $Y$ and it is  given by
 $$
 R(\mathcal{L}_Q)=\Big\{ g\in C^\alpha(\mathbb{T}), g= \sum_{n\geq1\\\atop n\neq m}g_{n+1}e_n, \quad g_n \in \mathbb{R} \Big\}.
 $$
 \item Transversality assumption: for any $Q=Q_m\in\mathcal{S},$
 $$
 \partial_Q\mathcal{L}_Q v_m\notin R(\mathcal{L}_Q).
 $$
\end{enumerate}
\end{proposition}
\subsection{Regularity of the  functional II}\label{RF2}
The main goal of this section is to study the regularity properties required by Theorem \ref{C-R0} for the functional $F$ introduced in \eqref{Eq8}.  
\begin{proposition}\label{biendef}
Let  $\alpha\in (0,1),\, \mu \in (0,1)$  and  set $r_{\mu}=\frac{1-\mu}{2}$. Then we have

 \[ \begin{array}{lll}
    F: (-1,1) \times (0, \mu)\times B_{r_{\mu}}    &\longrightarrow &Y\\
         (\varepsilon,Q,f)& \longmapsto& F(\varepsilon, \Omega,f)
    \end{array} \]
    is well-defined and of class $C^1$, and $\partial_Q\partial_f F$ exists and is continuous on $(-1,1) \times (0, \mu)\times B_{r_{\mu}},$ where $B_{r_{\mu}}=\big\{f\in X,\,  \text{ }\Vert f \Vert_{C^{1+\alpha}} \leq r_{\mu}\big\}.$   Moreover for any $i, j\in\N, i+j\leq3$ the function  $\partial_{\Omega}^i \partial_f^j F(\varepsilon,.,.)$ is continuous.
    \end{proposition}

\begin{proof}
We only sketch the basic steps of the proof which closely parallels
the proof developed in \cite{19}. For more details we refer the reader
to this reference. 
First, we write
\begin{eqnarray*}
F(\varepsilon,Q,f(w))&=& \textnormal{Im} \Big\{\scriptstyle{\frac{1-Q^2}{4}} \Big(1+Q \overline{w}^2+\overline{w} f(w)\Big)\big(1-Qw^2+\overline{f'(w)}\big)\\
 &-& \overline{\Phi'(w)} \overline{w} \fint_{\T} \Phi'(\tau)  K_0^\EE\big( \vert \Phi(w)- \Phi(\tau)\vert \big) d\tau  \Big\}
\end{eqnarray*}
with the notation $\Phi(w)= \alpha_q(w)+f(w)$ and
\[
K_0^\EE(x)= K_0\big(|\varepsilon| x \big)+\log\big({\scriptstyle{\frac{|\varepsilon|}{2}}}\big). \]
Since  $C^{\alpha}$ is an algebra,  $f \in C^{1+\alpha}$ and $f^\prime \in C^{\alpha}$, then the  first function 
\[w\in\T\mapsto   {\scriptstyle{\frac{1-Q^2}{4}}}\hbox{Im}\Big\{ [1+Q \overline{w}^2+\overline{w} f(w)][1-Qw^2+\overline{f'(w)}]\Big\}
\]
 belongs to $C^{\alpha}$ and its Fourier coefficients  are all real; therefore, it belongs to the space $Y$. The $C^1$  regularity with respect to $(Q,f)$ is elementary and was discussed in \cite{19}.  
For the second term, using the results in subsection $\ref{Bessel Function}$, one may write 
\begin{equation}\label{FF1}K_0^\varepsilon (x)=-\log(x)-\Big(\log\Big({\scriptstyle\frac{|\varepsilon|}{2}}\Big)+\log(x)\Big)\frac{\varepsilon^2 x^2}{4}\mathcal{K}_1(\varepsilon^2 x^2)+\mathcal{K}_2(\varepsilon^2x^2)
\end{equation}
where
\[\mathcal{K}_1(z)=\sum_{m=1}^{+ \infty} \frac{\left(\frac{z}{4}\right)^{m-1}}{(m!)^2}\]
and
\[\mathcal{K}_2(z)=\sum_{m=0}^{+ \infty} \left(\frac{z}{4} \right)^{m} \frac{\psi(m+1)}{(m!)^2}.\]
Consequently,
\[-\fint_{\T} \Phi'(\tau) K_0^\EE( \vert \Phi(w)-\Phi(\tau)\vert )d \tau= T_0 \Phi '(w)+T_1 \Phi '(w)+T_2\Phi'(w)\]
where 
\[ T_0\varphi(w)=\fint_{\T} \log(\vert \Phi(w)-\Phi(\tau)\vert ) \varphi(\tau) d \tau,\]
\begin{eqnarray*}T_1\varphi(w)&\triangleq&\fint_{\T} \widehat{K_1}(\tau, w,\varepsilon) \varphi(\tau) d \tau 
\end{eqnarray*}
and
\[ T_2\varphi(w) \triangleq- \fint_{\T} \mathcal{K}_2(\varepsilon^2 \vert \Phi(w)-\Phi(\tau)\vert^2)\varphi(\tau) d \tau \] 
with
\[
\widehat{K_1}(\tau, w,\varepsilon)\triangleq \frac{\EE^2}{4} \Bigg(\log \bigg(\frac{ \vert \varepsilon \vert }{2}\bigg)+\log(\vert \Phi(w)-\Phi(\tau)\vert)\Bigg)  |\Phi(w)-\Phi(\tau)|^2\mathcal{K}_1(\varepsilon^2 \vert \Phi(w)-\Phi(\tau)\vert^2 )
\]

Moreover, we have seen in \eqref{Flog}  that
\[T_0 \Phi '(w)=\overline{\widehat{T}_0\Phi '(w)}\]
with 
\[ \widehat{T_0}\varphi(w)=-\frac{1}{2}{\fint_{\T} \frac{\overline{\Phi(w)}-\overline{\Phi(\tau)}}{\Phi(w)-\Phi(\tau)} \varphi(\tau) d \tau }.\]
Let  $Q \in (0,\mu)$ and take  $r_{\mu}=\frac{1-\mu}{2}$. By the mean value theorem, there exists a constant $C_{\mu}$ such that   for all  $f \in B_{r_{\mu}}$ 
\[{\scriptstyle\frac{1-\mu}{2}} \vert\tau-w \vert \leq \vert \Phi(w)-\Phi(\tau) \vert \leq C_{\mu}|\tau-w|,\quad \forall \tau, w\in \T.
\]
In addition, we may easily check that the kernel $K(\tau,w)=\frac{\overline{\Phi(w)}-\overline{\Phi(\tau)}}{\Phi(w)-\Phi(\tau)}$ satisfies the assumptions of Lemma \ref{LemSing} and thus
\[
\|T_0\Phi^\prime\|_{C^\alpha(\T)}\le C \|\Phi^\prime\|_{L^\infty}\le C_0.
\]
Note that according to \cite{19} we also have that $(\EE,Q, f)\mapsto T_0\Phi^\prime$ is of class $C^1$ from $(-1,1) \times (0,\mu)\times B_{r_\mu}$ to $C^\alpha(\T)$.
As for  $T_2$,  the kernel is not singular and one may easily check that
\[
|\widehat{K}_1(\tau,w,\EE)|+|\partial_w\widehat{K}_1(\tau,w,\EE)|\leq C_0, \forall \tau, w\in \T.
\]
Consequently we may use once again Lemma \ref{LemSing} and deduce that $(\EE,Q, f)\mapsto T_1\Phi^\prime$ is well-defined.  Moreover the Fourier coefficients of $T_1\Phi^\prime$ are real  which follows from the general fact
\[
\overline{ T_1\varphi(w)}= T_1\varphi(\overline{w}), \forall \varphi \in X, \forall w\in \T.
\]
By straightforward arguments we can also prove that  $(\EE,Q,f)\mapsto T_1\Phi^\prime$ is of class $C^1$.
Observe that the regularity with respect to $\EE$ comes in particular from the fact that  the function $\EE\in (-1,1)\mapsto \EE^2 \log\EE $ is $C^1$. The same analysis can be implemented for the last term $T_2\Phi^\prime$ and this concludes the $C^1$ regularity of $(\EE,Q,f)\mapsto F(\EE,Q, f)$. 
Concerning the existence and the regularity of $\partial_Q\partial_f F$ it can be proved similarly to the case $\EE=0$ discussed in \cite{19}.
 \end{proof}

\subsection{Bifurcation diagram far from the resonant set}
The main goal of this section is study the structure of the
bifurcation diagram far from the resonant set
${\mathcal{S}_{\textnormal{reso}}}$ defined in \eqref{Resset}. We
prove its persistence for small perturbations. This is 
done in two different subsections. First we prove the stability
of the Kirchhoff ellipse branch under small perturbations, leading to the
existence of a two-fold branch for \eqref{sqg} living close to the
ellipse branch. Second we explore the bifurcation of one-fold
curves from the two-fold branch close to the non-resonant set. This
proves the persistence of the bifurcation diagram of the Euler equations
under small perturbations in $\EE$ but far from the resonant set.

\subsubsection{Structure of the two-fold curve}

The aim in this subsection is to construct two-fold V-states close to
Kirchhoff ellipses $\mathcal{E}_Q$ parametrized by
$w\in \T\mapsto w+Q\overline w$, with $Q\in[0,1)$.
We first study the case where $Q$ is far from the resonant set
$\mathcal{S}_{\textnormal{reso}}=\big\{Q_{2m}, m\in\N^\star\big\}$. We
prove that a one-dimensional continuous curve can be constructed away
from this set and which remains close to Kirchhoff ellipses for small
values of $\EE$. For this purpose we introduce the spaces

\begin{equation}\label{spaceX2}
X_2=\Big\{ f\in C^{1+\alpha}(\T), f(w)=\sum_{n\in \N^*} f_n w^{2n+1}, f_n\in \R\Big\}
\end{equation}
and
\begin{equation}\label{spaceY2}
Y_2=\Big\{ g\in C^{\alpha}(\T), g(w)=\sum_{n\in \N^*} g_n e_{2n}, g_n\in \R\Big\}, \textnormal{ with } e_n(w)= \hbox{Im}(w^n).
\end{equation}
Note that  a domain whose boundary is parametrized by   $\Phi(w)= w+Q\overline w+f(w), w\in \T$ with $f\in X_2$ is two-fold.
The main goal is to prove the following.
\begin{theorem}\label{thmF2}
Consider the V-state equation  \eqref{Eq8} and let
$m\in \N^*$, $\delta<\frac{Q_{2m+2}-Q_{2m}}{2}$.
Define $I_{m,\delta}=[Q_{2m}+\delta, Q_{ 2m+2}-\delta]$.
Then there exists $\EE_0>$ and a function 
\begin{align*}
 f\colon  [-\EE_0,\EE_0]\times I_{m,\delta}&\longrightarrow  X_2\\
  (\EE,Q)& \longmapsto f(\EE,Q).
\end{align*}
of class $C^1$ such that
\[F\big(\varepsilon, Q ,f(\EE,Q)\big)=0, \quad \forall \, (\EE,Q)\in  [-\EE_0,\EE_0]\times I_{m,\delta}.
\]
In particular the curve $Q\in I_{m,\delta}\mapsto \alpha_Q+ f(\EE,Q)$
describes rotating patches with two-fold symmetry living close to
Kirchhoff ellipses.
\end{theorem}
\begin{proof}
The proof relies on the use of  the implicit function theorem.  First notice from \mbox{Proposition \ref{biendef}} that
for any   $\mu \in (0,1)$, the functional

 \begin{align*}
    F \colon (-1,1) \times (0, \mu)\times B_{r_{\mu}}^2    &\longrightarrow Y_2\\
         (\varepsilon,Q,f) &\longmapsto F(\varepsilon, Q,f)
 \end{align*}
    is well-defined and of class $C^1$, where 
    \begin{equation}\label{ballz}
    B_{r_{\mu}}^2=\big\{f\in X_2,\,  \text{ }\Vert f \Vert_{C^{1+\alpha}} \leq r_{\mu}\big\} \quad \hbox{and}\quad r_{\mu}=\frac{1-\mu}{2}.
 \end{equation}
 We point out that the persistence of two-fold symmetry follows from  Proposition \ref{string11}.
In addition, $D_f F(0,Q,0)$ is given by the restriction of the  operator $\mathcal{L}_Q$ described by \eqref{diff1} on the sub-space $X_2$. As we have seen in Proposition \ref{prop-spec11}, the kernel of $\mathcal{L}_Q$ is non-trivial if and only if $Q\in \mathcal{S}$. Since $Q\in I_{m,\delta}$ then $\hbox{Ker } D_f F(0,Q,0)$ is trivial for any $Q\neq Q_{2m+1}$, and for $Q=Q_{2m+1}$ the kernel is one-dimensional and generated by the vector $v_{2m+1}(w)=\frac{w^{2m+2}}{1-Q w^2}$.  However this vector does not belong to $X_2$ and consequently  $\hbox{Ker } D_f F(0,Q_{2m+1},0)$ is also trivial. Therefore for any $Q\in I_{m,\delta}$ the linear operator $D_f F(0,Q,0)\in \mathcal{L}( X_2, Y_2)$  is one-to-one. We check that it is also onto. Let $g=\sum_{n\geq1} g_n e_{2n}\in Y_2$ and consider finding the pre-image by $D_f F(0,Q,0)$. Then according to Proposition \ref{prop-spec11}, $h(w)=\sum_{n\geq1} h_n w^{2n+1}$ satisfies $D_f F(0,Q,0)h=g $ if and only if
\begin{equation}\label{Solvsyst}
g_1=2 Q^2 h_1\quad \hbox{and}\quad g_n=\Big(1+Q^{2n}-(1-Q^2) n \Big) \big(h_n-Q h_{n-1}\big), \forall n\geq 2.
\end{equation}
Note that for each $n$ the number $\big((1-Q^2) n-1-Q^{2n}\big)$ vanishes if and only if $Q=Q_{2n}$, and thus for $Q\in I_{m,\delta}$ this coefficient
does not vanish uniformly in $n$. One can see from the recursion relation that
$$
h(w)=\frac{h_1 w^3+G(w)}{1-Q w^2} \quad\hbox{with}\quad G(w)=\sum_{n\geq2}\frac{g_n}{1+Q^{2n}-(1-Q^2) n} w^{2n+1}.
$$
Since $Q\in (0,1)$ and $\frac{1}{1-Q w^2}$ is $C^\infty(\T)$, then $h\in X_2$ if and only if $G\in C^{1+\alpha}(\T)$. Thus it suffices to establish  $G^\prime\in C^{\alpha}(\T)$ or equivalently 
$$
\chi\, \colon w\in \T\mapsto  \sum_{n\geq2}\frac{n(1-Q^2)\, g_n}{1+Q^{2n}-(1-Q^2) n} w^{2n}\in C^\alpha(\T).
$$ 
It is plain that
\begin{eqnarray*}
\chi(w)&=&-\sum_{n\geq2} g_n w^{2n}+\sum_{n\geq1}\frac{1+Q^{2n}}{1+Q^{2n}-(1-Q^2) n} g_n w^{2n}\\
&=&-\Pi^+\big( 2ig(w)-g_1w^2\big)+K\star\Pi^+\big(2ig(w)-g_1w^2\big)
\end{eqnarray*}
with $\Pi^+$ being the Szeg\"{o} projection and 
$$
K(w)=\sum_{n\geq2}\frac{1+Q^{2n}}{1+Q^{2n}-(1-Q^2) n}  w^{2n}.
$$
It is easy to prove the existence of a constant $C>0$ such that for any $Q\in I_{m,\delta}$
$$
\Big|\frac{1+Q^{2n}}{1+Q^{2n}-(1-Q^2) n}\Big|\le \frac{C}{n}.
$$
Therefore $K\in L^2(\T)\subset L^1(\T)$, combined with the fact  $\Pi^+ g\in C^\alpha(\T)$, implies that  $\chi\in C^\alpha(\T)$. Finally we see that $D_f F(0,Q,0)$ is onto and thus it is an isomorphism from $X_2$ to $Y_2$. By the implicit function theorem and a standard compactness argument we conclude  the  existence of a unique surface of solutions 
\[F\big(\varepsilon, Q ,f(\EE,Q)\big)=0, \quad \forall \, (\EE,Q)\in  [-\EE_0,\EE_0]\times I_{m,\delta}.
\]
This achieves the proof of Theorem \ref{thmF2}.
\end{proof}
\subsubsection{Bifurcation from the two-fold curve}
In this subsection we prove the bifurcation of countable family of
one-dimensional curves of V-states from the curve constructed in
Theorem \ref{thmF2} at some points which are close to the points of
the non-resonant set $\mathcal{S}_{\textnormal{Nreso}}=\big\{Q_{2m+1},
m\geq 1\big\}$.

 The main result may be stated as follows.


\begin{theorem}
Let $m\geq3$ be an odd number. There exists $\EE_0>0$  such that for any $\EE\in [-\EE_0,\EE_0]$, there exists a one-dimensional curve of one-fold  V-states bifurcating from the two-fold branch constructed in Theorem $\ref{thmF2}$ at a point $Q_{\EE,m}$ close to $Q_{m}$.

\end{theorem}

\begin{proof}

The proof follows the same lines of Theorem \ref{C-R} with slight
modifications using the Lyapunov-Schmidt reduction in an important
way.  First recall from Proposition \ref{biendef} that the functional
$F: (-1,1)\times(0,\mu)\times B_{r_\mu}\to Y$ is well-defined and is
of class $C^1,$ with
$$
    B_{r_{\mu}}=\big\{f\in X,\,  \text{ }\Vert f \Vert_{C^{1+\alpha}} \leq r_{\mu}\big\} \quad \hbox{and}\quad r_{\mu}=\frac{1-\mu}{2}.
$$
According to  Proposition  \ref{prop-spec11}, for $Q=Q_m$ the kernel of $\mathcal{L}_{Q_m}$ is generated by the vector $v_m(w)=\frac{w^{m+1}}{1-Q w^2}$. Denote by   $\mathcal{X}$    a  complement of $v_m$ in $X$ such that
$$
X_2\subset \mathcal{X}.
$$ 
This last fact follows since $m$ is odd and therefore the function
$v_m$ is even; consequently, we can choose a complement containing odd
functions which is exactly the space $X_2$. Recall that the spaces $X$
and $X_2$ were introduced in \eqref{spaceX} and \eqref{spaceX2}. The
range of $\mathcal{L}_{Q_m}$, denoted by $\mathcal{Y}$, is of
co-dimension one and we may choose a complement generated by the
vector $w_m=e_m$.  Then

$$
X=\langle v_m\rangle \oplus\mathcal{X}\quad  \hbox{and} \quad Y=\langle w_m\rangle \oplus\mathcal{Y}.
$$

Let  $\Pi_1: X\mapsto\langle v_m\rangle$ be  the projection along $\mathcal{X}$ onto $\langle v_m\rangle$. Thus 
$$
f=sv_m+ g, \quad g\in \mathcal X\Longrightarrow  \Pi_1f= sv_m,
$$
and similarly define the projection
$\Pi_2:  Y\mapsto \langle w_m\rangle$ along
$\mathcal{Y}$ onto $\langle w_m\rangle $. 
The V-state equation $F(\varepsilon,Q, f)=0$  is then equivalent to the system
$$
F_1(\EE,Q,s, g)\triangleq(\hbox{Id}-\Pi_2)F(\EE,Q,s v_m+g)=0
$$
and
$$
F_2(\EE,Q,s, g)\triangleq \Pi_2F(\EE,Q, sv_m+g)=0.
$$

The function $F_1:(-1,1)\times (0,\mu)\times(-\eta,\eta) \times B_r \to\mathcal{Y}$  is $C^1$, with $B_r$ a small ball in  $\mathcal{X}$  centered  at  $0$, and $\eta>0$ such that for any $s\in (-\eta,\eta)$ and for any $g\in B_r$ we have $sv_m+g\in B_{r_{\mu}}$. Moreover, 
$$
F_1(0,Q_m,0,0)=0
$$
and it is not difficult to check that 
\begin{equation*}
\partial_g F_1(0,Q_m,0,0)=(\hbox{Id}-\Pi_2)\partial_fF(0,Q_m,0):\mathcal{X}\to \mathcal{Y}
\end{equation*}
is an isomorphism.
By the implicit function theorem the solutions of the equation $F_1(\EE,Q,s,g)=0$ are described near the point $(0,Q_m,0,0)$ by the parametrization $g=\varphi(\EE,Q, s)$ with 
$$\varphi \colon (-\delta ,\delta)\times(Q_m-\delta, Q_m+\delta)\times(-\delta,\delta)\to\mathcal{X}, \quad \delta>0
$$ 
being a $C^1$ function. By virtue of  Theorem \ref{thmF2}, we know the existence of a function  $$(\EE,Q)\mapsto f(\EE,Q)\in X_2\subset \mathcal{X}$$ such that
\begin{equation}\label{Eq0xxF}
F(\EE,Q,f(\EE,Q))=0
\end{equation} 
and so in particular
$$
F_1(\EE,Q,0,f(\EE,Q))=0.
$$
Thus by uniqueness we get
\begin{equation}\label{Eq3F}
\varphi(\EE,Q,0)=f(\EE,Q), \quad \forall (\EE,Q)\in  (-\delta ,\delta)\times(Q_m-\delta, Q_m+\delta).
\end{equation}
As Kirchhoff ellipses are exact solutions for $\EE=0$, we obtain
\begin{equation}\label{Eq5F}
\varphi(0,Q,0)=0, \quad \forall Q\in  (Q_m-\delta, Q_m+\delta).
\end{equation}
The equation of $F_2$  in a neighbourhood of $(0,Q_m,0,0)$ takes the form 

\begin{equation*}
 \widehat{F_2}(\EE,Q,s)\triangleq \Pi_2F\big(\EE, Q, s v_m+\varphi(\EE,Q,s )\big)=0, \quad \forall\,  |\EE| , |Q-Q_m|,|s|<\delta.
\end{equation*}
From the relations \eqref{Eq0xxF}  and \eqref{Eq3F} we deduce that 
$$
\widehat{F_2}(\EE,Q,0)=0, \quad \forall |\EE| \le\delta, \forall \,|Q-Q_m|\le\delta.
$$
Set
\begin{equation}
\widehat{g}(\EE,Q,s)\triangleq\left\{ \begin{array}{ll}
 \frac{\widehat{F_2}(\EE,Q,s)}{s}, \quad s\neq0\\
 \Pi_2\partial_fF\big(\EE, Q,\varphi(\EE,Q,0)\big)\big(v_m+ \partial_s \varphi(\varepsilon,Q,0)\big),\quad s=0.\end{array}\right.
\end{equation}
Then the function  $\widehat{g}$ is continuous and 
\begin{eqnarray*}
\widehat{g}(0,Q_m,0)&=& \Pi_2\partial_fF(0,Q_m,0)\big(v_m+\partial_s \varphi(0,Q_m,0)\big)\\
&=&0.
\end{eqnarray*}
Indeed, by differentiating with respect to $s$ the following equation 
\[ F_1(\varepsilon, Q,s,\varphi(\varepsilon,Q,s))=0, \, \forall \vert \varepsilon \vert, \vert Q-Q_m \vert, \vert s \vert \leq \delta\]
 at the point  $(0,Q_m,0)$, we find
 
 \[ (\hbox{Id}-\Pi_2)\partial_f F(0,Q_m,0)(v_m+\partial_s \varphi(0,Q_m,0))=0.\]
 Consequently, $\partial_s \varphi(0,Q_m,0) \in \textnormal{Ker}(\mathcal{L}_{Q_m}) \cap \mathcal{X}$, and therefore  
 \[\partial_s \varphi(0,Q_m,0) =0.\]
 Thanks to $(\ref{Eq5F})$, we obtain
 \[\partial_Q \varphi(0,Q_m,0)=0.\]
Moreover,  $\widehat{g}$ is differentiable with respect to $Q$ and
\begin{eqnarray*}
\partial_{Q} \widehat{g}(0,Q_m,0)&=&\Pi_2\partial_Q\partial_fF(0,Q_m,0)\big(v_m+\partial_s \varphi(0,Q_m,0)\big)\\
&+&\Pi_2\partial_fF(0,Q_m,0)\big(\partial_Q{v_m}_{|Q=Q_m}+\partial_Q\partial_s \varphi(0,Q_m,0)\big)\\
&+&\Pi_2 \partial_f^2 F(0,Q_m,0)\big(v_m+\partial_s \varphi(0,Q_m,0), \partial_Q \varphi(0,Q_m,0)\big)\\
&=&\Pi_2\partial_Q \partial_f F(0,Q_m,0)(v_m).
\end{eqnarray*}

From the transversality assumption proved in Proposition \eqref{prop-spec11}, we obtain
\[
\partial_{Q} \widehat{g}(0,Q_m,0)\neq0.
\]
Hence using a weak version of the implicit function theorem, see \cite{CR}, we deduce   that the solutions of $\widehat{g}(\EE,Q,s)=0$ near the point $(0,Q_m,0)$ are parametrized by a continuous  surface  $\gamma:(-\EE_0,\EE_0)^2\to \R$ such that $Q=\gamma(\EE,s)$ and
$$
\widehat{g}\big(\EE,\gamma(\EE,s),s\big)=0, \forall |\EE| \le\EE_0,\,\forall  |s|\le\EE_0,\quad \hbox{with}\quad  \EE_0>0.
$$
Therefore the solutions  of the equation $F(\EE,Q,f)=0$ near the point $(0,Q_m,0)$  are given by the union $\mathcal{C}_1\cup\mathcal{C}_2$ where 
$$
\mathcal{C}_1=\Big\{\big(\EE,Q,\varphi(\EE,Q,0)\big), |\EE| \le \EE_0, |Q-Q_m|\leq\EE_0\Big\}
$$
corresponding to the two-fold V-states constructed in Theorem \ref{thmF2} and 
$$
\mathcal{C}_2=\Big\{\big(\EE,\gamma(\EE,s), sv_m+\varphi\big(\EE,\gamma(\EE,s),s\big)\big), |\EE| \leq\EE_0, |s|\leq  \EE_0\Big\}.
$$
Note that the curve $\mathcal{C}_2$ is different from $\mathcal{C}_1$ since for  $s\neq0$ the V-state parametrized by 
$$w\in\mathbb{T}\mapsto  w+Q\overline{w}+s v_m(w)+\varphi\big(\EE,\gamma(\EE,s),s\big)
$$
is not two-fold because $m$ is odd and therefore $v_m$ is a  non-vanishing even function. In addition the curve $\mathcal{C}_2$ intersects $\mathcal{C}_1$ at $s=0$. This achieves the proof.
\end{proof}

\subsection{Breakdown of the bifurcation diagram close to the resonant set }
The numerical study conducted in Section \ref{sec:num2} shows that,
contrary to what occurs in the Euler equations, the two-fold branch is
never connected for small $\varepsilon$ and is split into countable
disjoint connected components or branches. The separation of the
singularity set seems to happen around the resonant set
$\mathcal{S}_{\textnormal{reso}}=\big\{Q_{2m}, m\geq2\big\}$ due to
the resonance between branches with the same symmetry. We provide 
an analytical confirmation of this behavior only around the point
$Q_4$ which is more tractable than the remaining cases $Q_{2m},
m\geq3$. We point out that the separation of the two-fold branch
around $Q_{4}$ is only proved locally in the bifurcation diagram.
The global structure of this separation is much more
complicated and may require more elaborate tools.\\

More precisely, we prove the following result.
\begin{theorem}\label{thmbreak}
Consider  the V-state equation  \eqref{Eq8}. There exists $\varepsilon_0>0$ such that for any $\varepsilon\in (-\EE_0,\varepsilon_0)\backslash\{0\}$ there exists $r_\varepsilon>0$ such that the set
$$
\Big\{F\big(\varepsilon, Q ,f\big)=0, |Q-Q_4|<r_\varepsilon, f\in X_2, \|f\|_{C^{1+\alpha}}<r_\varepsilon\big\}
$$
is given by the union of two disjoint  one-dimensional curves.

\end{theorem} 
\begin{proof}
The proof is based on studying the local structure of the V-state
equation \eqref{tildeG_j} through the associated quadratic form. At
first sight there is a logarithmic singularity in $\EE$ at second order
which could present difficulties for understanding the local
structure. However, as shown below, this term may be combined
with the rotation term and therefore it does not contribute at the
nonlinear level. To show this, we first make some transformations 
using new unknowns. From \eqref{FF1} we may write for $x>0$ the
expansion
\[K_0^\EE(x)=\psi(1)-\log(x)-\frac{\EE^2}{4} x^2\log(x)+\frac{\EE^2}{4}\big(\psi(2)-\log(|\EE | /2)\big) x^2+ \EE^4\log\EE\,  \mathcal{R}^\EE(x)
\]
where $\mathcal{R}^\EE$ is at least of class $C^3$ in the variable $x$ and analytic in the variable $\EE$.
It follows that 
\begin{eqnarray*}
\fint_{\T}\Phi^\prime(\tau)K_0^\EE(|\Phi(\tau)-\Phi(w)|) d\tau&=&-\fint_{\T}\Phi^\prime(\tau)\log(|\Phi(\tau)-\Phi(w)|) d\tau\\
&+&\frac{\EE^2}{4}\big(\psi(2)-\log(| \EE | /2)\big) \fint_{\T}\Phi^\prime(\tau)|\Phi(\tau)-\Phi(w)|^2 d\tau\\
&-&\frac{\EE^2}{4} \fint_{\T}\Phi^\prime(\tau)\log(|\Phi(\tau)-\Phi(w)|) |\Phi(\tau)-\Phi(w)|^2 d\tau\\
&+&\EE^4\log|\EE|   \fint_{\T}\Phi^\prime(\tau)\mathcal{R}^\EE(|\Phi(\tau)-\Phi(w)|) d\tau.
\end{eqnarray*}
Denote $\Gamma=\Phi(\T)$; then it is simple to obtain, by a change of variables and the residue theorem,
\begin{eqnarray*}
 \fint_{\T}\Phi^\prime(\tau) |\Phi(\tau)-\Phi(w)|^2 d\tau&=&\fint_{\Gamma}|\xi-\Phi(w)|^2 d\xi\\&=&
 \fint_{\Gamma}\big(|\xi|^2-\Phi(w)\overline{\xi}\big) d\xi.
\end{eqnarray*}
From Cauchy-Pompeiu's formula we find
\[\fint_{\Gamma}|\xi|^2 d\xi=\frac1\pi\int_{D}\xi d A,
\]
and since the domain $D$ delimited by the curve $\Gamma$ is two-fold and  centered at the origin then
\[\fint_{\Gamma}|\xi|^2 d\xi=0.
\]
Hence
\begin{eqnarray*}
 \fint_{\T} |\Phi(\tau)-\Phi(w)|^2 \Phi^\prime(\tau) d\tau&=& -\Phi(w)\fint_{\T}\overline{\Phi(\tau)}\Phi^\prime(\tau) d\tau.
\end{eqnarray*}
Thus we obtain
\begin{eqnarray*}
\fint_{\T}K_0^\EE(|\Phi(\tau)-\Phi(w)|) \Phi^\prime(\tau) d\tau&=&-\fint_{\T}\log(|\Phi(\tau)-\Phi(w)|) \Phi^\prime(\tau) d\tau\\
&-&\frac{\EE^2}{4}\big(\psi(2)-\log(| \EE | /2)\big) \Phi(w)\fint_{\T}\overline{\Phi(\tau)}\Phi^\prime(\tau) d\tau\\
&-&\frac{\EE^2}{4} \fint_{\T}\log(|\Phi(\tau)-\Phi(w)|) |\Phi(\tau)-\Phi(w)|^2 \Phi^\prime(\tau)d\tau\\
&+&\EE^4\log|\EE|   \fint_{\T}\mathcal{R}^\EE(|\Phi(\tau)-\Phi(w)|)  \Phi^\prime(\tau)d\tau.
\end{eqnarray*}
Inserting this identity into \eqref{tildeG_j} and using  \eqref{Eq71} we find
\begin{eqnarray*}
G(\EE,\Omega, \Phi)&=&-G_E(\Omega_\EE,\Phi)+\frac{\EE^2}{4} \hbox{Im}\{ G_1(\Phi)\}+\EE^4\log|\EE| \,\hbox{Im}\{G_2(\EE,\Phi)\}
\end{eqnarray*}
with
\[\Omega_\EE\triangleq\Omega+\frac{\EE^2}{4}\big(\psi(2)-\log(|\EE | /2)\big) \fint_{\T}\overline{\Phi(\tau)}\Phi^\prime(\tau) d\tau,
\]
\[
G_1(\Phi)\triangleq\overline{w}\,\overline{\Phi^\prime(w)}  \fint_{\T}\log(|\Phi(\tau)-\Phi(w)|) |\Phi(\tau)-\Phi(w)|^2 \Phi^\prime(\tau)d\tau
\]
and
\[G_2(\EE,\Phi)\triangleq  -\overline{w}\,\overline{\Phi^\prime(w)}\fint_{\T}\Phi^\prime(\tau)\mathcal{R}^\EE(|\Phi(\tau)-\Phi(w)|) d\tau.
\]
As before we look for solutions of the form
\[
\Phi=\alpha_Q+ f,\,  f\in X_2.
\]
Note that the space $X_2$ was introduced previously in  \eqref{spaceX2}.
We now impose the constraint
\[
\Omega_\EE=\frac{1-Q^2}{4}.
\]
Set 
\[
 F_E(Q,f)\triangleq-G_E\big({(1-Q^2)}/{4}
,\alpha_Q+f\big),\quad  F_1(Q,f)\triangleq G_1(\alpha_Q+f)
\]
and
\[
  F_2(\EE,Q,f)\triangleq G_2(\EE,\alpha_Q+f).
\]
Then the V-state equation is equivalent to
\begin{equation}\label{V-st2}
\widehat{F}(\EE,Q,f)\triangleq F_E(Q,f)+\frac{\EE^2}{4}\hbox{Im}\{F_1(Q,f)\}+\EE^4\log(|\EE|)\,  \hbox{Im}\{F_2(\EE,Q,f)\}=0.
\end{equation}
 Following the same lines of  Proposition \ref{biendef} we can check that 
\begin{align*}
    \widehat{F} \colon (-1,1) \times (0, \mu)\times B_{r_{\mu}}    &\longrightarrow Y_2\\
         (\varepsilon,Q,f) &\longmapsto \widehat{F}(\varepsilon, Q,f)
 \end{align*}
    is well-defined and of class $C^1$, where the ball $ B_{r_{\mu}}$ was previously defined in \eqref{ballz}.
    In fact, we can easily check that $\widehat{F}$ is at least $C^3$. Moreover, $\partial_Q \widehat{F}$ and $D_f  \widehat{F}$ are also $C^3$.
Note that from \eqref{Eq8} we have
\[
D_f \widehat{F}(0,Q,0)=D_fF_E(Q,0)=\mathcal{L}_Q,
\]
and the full structure of $\mathcal{L}_Q$ was previously detailed in Proposition \ref{prop-spec11}. Recall in particular that for $Q=Q_4$ one has
\[\hbox{Ker} \mathcal{L}_{Q_4}=\langle v_4\rangle,\, v_4(w)=\frac{w^5}{1-Q_4 w^2}
\]
and
\[
 R(\mathcal{L}_{Q_4})=\Big\{ g\in C^\alpha(\mathbb{T}), g= \sum_{n\geq1\\\atop n\neq 2}g_{n}e_{2n}, \quad g_n \in \mathbb{R} \Big\}.
\]
Let $\widehat{X}_2$ be a complement of $v_4$ in $X_2$, that is,
\[
X_2=\widehat{X}_2\oplus \langle v_4\rangle.
\] 
Thus any $f\in X_2$ admits  a unique decomposition in the form
$f=s v_4+g$ with $g\in \widehat{X}_2$ and $ s\in\R$.
Denote by $\widehat{Y}_2$ the space $R(\mathcal{L}_{Q_4})$, then 
\[
Y_2=\widehat{Y}_2\oplus \langle e_4\rangle,
\]
and let $\Pi: Y_2\to \langle e_4\rangle$ be the canonical projection along $\widehat{Y}_2$. Thus equation \eqref{V-st2} is equivalent to 
\begin{equation}\label{Eqw12}
H_1(\EE,Q,s,g) \triangleq \big(\hbox{Id}-\Pi\big)\widehat{F}(\EE,Q,s v_4+g)=0\quad\hbox{and}
\quad H_2(\EE,Q,s,g) \triangleq \Pi\widehat{F}(\EE,Q,s v_4+g)=0.
\end{equation}
Now we can directly observe that
\[
\partial_g H_1(0,Q_4,0,0)=\big(\hbox{Id}-\Pi\big)\mathcal{L}_{Q_4},
\]
and that $\partial_g H_1(0,Q_4,0,0): \widehat{X}_2\to \widehat{Y}_2$ is an isomorphism.  Thus, by the implicit function theorem, the solutions of the equation $H_1(\EE,Q,s,g)=0$ are described near the point $(0,Q_4,0,0)$ by the parametrization $g=\varphi(\EE,Q, s)$ with 
\begin{equation}\label{phi}
\varphi \colon (-\delta ,\delta)\times(Q_4-\delta, Q_4+\delta)\times(-\delta,\delta)\to\widehat{X}_2, \quad \delta>0
\end{equation}
being a $C^1$ function. As the functions defining the V-states are smooth enough, $\varphi$ is in fact at least of class $C^3$ . Notice that the defect of regularity  comes only from the variable $\EE$ and thus one gets more smoothness in the remaining variables. Indeed, we have the following. 
\begin{remark}\label{rmkk1}
The functions $\partial_Q \varphi$ and $\partial_s  \varphi$ are at least $C^3$.
To prove this, we just differentiate  the  equation
\[ \big(\textnormal{Id}-\Pi\big)\widehat{F}(\EE,Q,sv_4+ \varphi(\EE,Q,s))=0, \; \forall ( \EE,Q,s)\in (-\delta ,\delta)\times(Q_4-\delta, Q_4+\delta)\times(-\delta,\delta)
\]
and argue by induction.
\end{remark}
  Now since the ellipses are solutions, then by uniqueness we obtain
\begin{equation}\label{zz123}\varphi(0,Q,0)=0,\forall Q\in (Q_4-\delta, Q_4+\delta).
\end{equation}
This implies that for any $k\in\{1,2,3\}$ 
\begin{equation}\label{Eqsw136}
\partial_Q^k\varphi(0,Q,0)=0, \forall Q\in (Q_4-\delta, Q_4+\delta).
\end{equation}
On the other hand
differentiating the first equation in \eqref{Eqw12} with respect to $s$ we
obtain
\begin{eqnarray*}
\big(\hbox{Id}-\Pi\big)\big(\mathcal{L}_{Q_4}(v_4+\partial_s\varphi(0,Q_4,0))
&=&0
\end{eqnarray*}
and as $\partial_s\varphi(0,Q_4,0)\in\widehat{X}_2$  we deduce 
\begin{equation}\label{Eqsw135}
\partial_s\varphi(0,Q_4,0)=0.
\end{equation}
Similarly, differentiating the first equation of \eqref{Eqw12}  with respect to $\EE$, we obtain (due to \eqref{V-st2})
 \[
 \big(\hbox{Id}-\Pi\big)\mathcal{L}_{Q_4}\partial_\EE\varphi(0,Q_4,0)=0
 \]
 which implies that
 \begin{equation}\label{Eqsw134}
 \partial_\EE\varphi(0,Q_4,0)=0.
 \end{equation}
Now  the V-state equation reduces in this small neighbourhood to the finite-dimensional equation
\begin{equation}\label{Eqww1}
\psi(\EE,Q,s)\triangleq H_2\big(\EE,Q,s,\varphi(\EE,Q, s)\big)=0.
\end{equation}
It is obvious from  \eqref{V-st2}
 that 
\begin{equation}\label{Idenz1}
\psi(\EE,Q,s)=\psi_E(\EE,Q,s)+\frac{\EE^2}{4} \psi_1(\EE,Q,s)+\EE^4\log|\EE|\, \psi_2(\EE,Q,s)
\end{equation}
with
\[
\psi_E(\EE,Q,s) \triangleq \Pi F_E\big(Q,sv_4+\varphi(\EE,s,Q)\big),\,\, \psi_1(\EE,Q,s)\triangleq \Pi\,  \hbox{Im}\big \{F_1\big(Q,sv_4+\varphi(\EE,Q,s)\big)\big\}
\]
and
\[
\psi_2(\EE,Q,s)\triangleq\Pi\, \hbox{Im}\, \big\{F_2\big(\EE,Q,sv_4+\varphi(\EE,Q,s)\big)\big\}.
\]
Moreover,
\[\psi(0,Q,0)=0,\forall Q\in (Q_4-\delta, Q_4+\delta)
\]
which implies that for any $k\in\{1,2,3,4 \}$
\begin{equation}\label{Eqsw137}
\partial_Q^k\psi(0,Q,0)=0, \forall Q\in (Q_4-\delta, Q_4+\delta).
\end{equation}
Moreover straightforward computations yield, in view of \eqref{Eqsw134} and the structures of $\Pi$ and $\mathcal{L}_{Q_4}$,
\[
\partial_\EE\psi_E(0,Q_4,0)=\Pi \mathcal{L}_{Q_4}\partial_\EE\varphi(0,Q_4,0)=0.
\]
Using once again  \eqref{Eqsw134} we find
\begin{eqnarray*}
\partial_\EE^2\psi_E(0,Q_4,0)&=&\Pi\partial^2_fF_E(Q_4,0)\big(\partial_\EE\varphi(0,Q_4,0),\partial_\EE\varphi(0,Q_4,0)\big)+\Pi \mathcal{L}_{Q_4}\partial_\EE^2\varphi(0,Q_4,0)\\
&=&0.
\end{eqnarray*}
Hence we obtain
\begin{eqnarray*}
\partial_\EE^2\psi(0,Q_4,0)&=&\frac12\psi_1(0,Q_4,0)\\
&=&\frac12\Pi \, \hbox{Im}\{F_1(Q_4,0)\}.
\end{eqnarray*}
To compute the projection, we need to calculate  the coefficients of $w^4$ and $\overline{w}^4$ in the Fourier expansion of  $F_1(Q_4,0)$. First we have
\begin{eqnarray*}
F_1(Q_4,0)&=&\overline{w}\,\overline{\alpha_{Q_4}^\prime(w)} I(w)\\
&=&\big(\overline{w}-Q_4w\big) I(w)
\end{eqnarray*}
with 
\begin{eqnarray*}
I(w)&\triangleq&  \fint_{\T}\log(|\alpha_{Q_4}(\tau)-\alpha_{Q_4}(w)|) |\alpha_{Q_4}(\tau)-\alpha_{Q_4}(w)|^2 \alpha_{Q_4}^\prime(\tau)d\tau.
\end{eqnarray*}
Using the  identity
\[
|\alpha_{Q_4}(\tau)-\alpha_{Q_4}(w)|=|\tau-w||1-Q_4\tau w|, \quad \forall \tau, w\in\T
\]
we find after 
straightforward computations 
\begin{eqnarray}\label{Iden11}
 |\alpha_{Q_4}(\tau)-\alpha_{Q_4}(w)|^2 \alpha_{Q_4}^\prime(\tau)&=&\sum_{k=0}^2\alpha_k(w)\tau^k+\sum_{k=1}^4\beta_k(w)\overline{\tau}^k
 \end{eqnarray}
 with
 \begin{eqnarray*}
\alpha_0(w)\triangleq2+Q_4^2+Q_4(w^2+\overline{w}^2),\quad \alpha_1(w)&\triangleq&-2Q_4w-(1+Q_4^2)\overline{w},\quad \alpha_2(w)\triangleq Q_4,
 \end{eqnarray*}

\begin{eqnarray*}
\beta_1(w)\triangleq (Q_4^2-1)\big(Q_4 \overline{w}+w\big),\quad \beta_2(w)&\triangleq&-Q_4(2Q_4^2+1)-Q_4^2 w^2-Q_4^2\overline{w}^2,
 \end{eqnarray*}
 
 \begin{eqnarray*}
\beta_3(w)\triangleq2Q_4^2 \overline{w}+Q_4(1+Q_4^2)w\quad \hbox{ and} \quad \beta_4(w)&\triangleq&-Q_4^2.
 \end{eqnarray*}
 Now we compute for $n\in\Z$
\begin{eqnarray*}
  \fint_{\T}\log(|\alpha_{Q_4}(\tau)-\alpha_{Q_4}(w)|\tau^nd\tau&=& \fint_{\T}\log(|\tau-w|)\tau^nd\tau+ \fint_{\T}\log(|1-Q_4\tau w|)\tau^nd\tau\\
  &\triangleq&I_n(w)+J_n(w).
 \end{eqnarray*}
 First note that by a change of variables
 \[
 I_n(w)=w^{n+1}\frac{1}{2\pi}\int_{0}^{2\pi}\log|1-e^{i\theta}| e^{i(n+1)\theta} d\theta.
 \]
 From elementary trigonometric identities we write
 \[
  I_n(w)=w^{n+1}\frac{1}{4\pi}\int_{0}^{2\pi}\log\big(4\sin^2(\theta/2) \big) e^{i(n+1)\theta} d\theta.
 \]
 Using Lemma A.3 of \cite{4} we get
 \begin{equation*}
I_n(w)=\left\lbrace
\begin{array}{l}
-\frac{1}{2|n+1|}w^{n+1},\quad \hbox{if}\quad n\in\Z\backslash\{-1\}
 \\
0,\quad \hbox{if}\quad n=-1.
\end{array}
\right.
\end{equation*}
In addition
 \begin{eqnarray*}
 J_n(w)&=&\overline{w}^{n+1}\frac{1}{2\pi}\int_{0}^{2\pi}\log|1-Q_4 e^{i\theta}| e^{i(n+1)\theta} d\theta\\
 &=&\overline{w}^{n+1}\frac{1}{4\pi}\int_{0}^{2\pi}\log\big|1+Q_4^2-2Q_4 \cos(\theta)\big| e^{i(n+1)\theta} d\theta.
 \end{eqnarray*}
 Again from Lemma A.4 \cite{4} we obtain
  \begin{equation*}
J_n(w)=\left\lbrace
\begin{array}{l}
-\frac{1}{2|n+1|} Q_4^{|n+1|}\overline{w}^{n+1},\quad \hbox{if}\quad n\in\Z\backslash\{-1\}
 \\
0,\quad \hbox{if}\quad n=-1.
\end{array}
\right.
\end{equation*}
Putting together the preceding identities one finds
\begin{equation}\label{Idz9}
\fint_{\T}\log(|\alpha_{Q_4}(\tau)-\alpha_{Q_4}(w)|\tau^nd\tau=\left\lbrace
\begin{array}{l}
-\frac{1}{2|n+1|}\big(w^{n+1}+ Q_4^{|n+1|}\overline{w}^{n+1}\big),\quad \hbox{if}\quad n\in\Z\backslash\{-1\}
 \\
0,\quad \hbox{if}\quad n=-1.
\end{array}
\right.
\end{equation}
Using \eqref{Iden11} we find
\begin{eqnarray*}
I(w)&=&\sum_{k=0}^2\alpha_k(w)\big(I_k(w)+J_k(w)\big)+\sum_{k=1}^4\beta_k(w)\big(I_{-k}(w)+J_{-k}(w)\big)\\
&=&\sum_{k=0}^2\alpha_k(w)\big(I_k(w)+J_k(w)\big)+\sum_{k=2}^4\beta_k(w)\big(I_{-k}(w)+J_{-k}(w)\big).
\end{eqnarray*}
From straightforward calculation, and using the fact that $Q_4$ is a solution of $(\ref{equaQ4})$ with $m=4$, we obtain
\[
\Pi\,\hbox{Im}\Big\{\overline{w}\,\overline{\alpha_{Q_4}^\prime(w)} \sum_{k=0}^2\alpha_k\big(I_k+J_k\big)\Big\}=\frac{Q_4^2}{12}(5-Q_4^2)  e_4.
\]
Similarly we obtain
\[
\Pi\,\hbox{Im}\Big\{ \overline{w}\,\overline{\alpha_{Q_4}^\prime(w)} \sum_{k=2}^4\beta_k\big(I_{-k}+J_{-k}\big)\Big\}=\frac{Q_4^6-3Q_4^4-2Q_4^2}{12}  e_4.
\]
Consequently
\[
\Pi\, \hbox{Im} \Big\{F_1(Q_4,0) \Big\}=\frac{Q_4^2}{12}(Q_4^4-4Q_4^2+3) e_4.
\]
Therefore
\begin{equation}\label{Eqdev1}
\partial_\EE^2\psi(0,Q_4,t)=\frac{Q_4^2}{24}(Q_4^4-4Q_4^2+3) e_4.
\end{equation}
We next compute $\partial_s\partial_\EE\psi(0,Q_4,0)$. From \eqref{Idenz1} we write
\begin{eqnarray*}
\partial_s\partial_\EE\psi(0,Q_4,0)&=&\Pi\partial_f^2F_E(Q_4,0)[\partial_\EE\varphi(0,Q_4,0),v_4+\partial_s\varphi(0,Q_4,0)]\\
&+&\Pi\mathcal{L}_{Q_4}\partial_\EE\partial_s\varphi(0,Q_4,0).
\end{eqnarray*}
From \eqref{Eqsw134} and the structure of $\Pi$ we find
\begin{equation}\label{Eqszz1}
\partial_s\partial_\EE\psi(0,Q_4,0)=0.
\end{equation}
Similarly we find
\begin{eqnarray*}
\partial_Q\partial_s\psi(0,Q_4,0)&=&\Pi\partial_Q\partial_sF_E(0,Q_4,0)\\
&=&\partial_Q\big\{\Pi\partial_f F_E(Q,\varphi(0,Q,0))(v_4+\partial_s\varphi(0,Q,0))\big\}_{Q=Q_4}.
\end{eqnarray*}
Using \eqref{zz123} and \eqref{Eqsw136} we deduce that
\begin{eqnarray*}
\partial_Q\partial_s\psi(0,Q_4,0)&=&\Pi\big\{\partial_Q\mathcal{L}_Qv_4\big\}_{Q=Q_4}.
\end{eqnarray*}
This is the transversality assumption in the Crandall-Rabinowitz theorem. According to \cite{19} we have
\[
\Pi\big\{\partial_Q\mathcal{L}_Qv_4\big\}_{Q=Q_4}=4(Q_4+Q_4^3) e_4
\]
and thus
\begin{equation}\label{Zqsr1}
\partial_Q\partial_s\psi(0,Q_4,0)=4(Q_4+Q_4^3) e_4.
\end{equation}
To compute $\partial_Q\partial_\EE\psi(0,Q_4,0)$ we note from \eqref{zz123}, \eqref{Eqsw134} and the identity $\Pi\mathcal{L}_{Q_4}=0$ that 
\begin{eqnarray}\label{Ze1}
\nonumber \partial_Q\partial_\EE\psi(0,Q_4,0)&=&\Pi\partial_Q\big\{\partial_fF_E(Q,\varphi(0,Q,0))\partial_\EE\varphi(0,Q,0)\big\}_{Q=Q_4}\\
&=&0.
\end{eqnarray}

The computations of $\partial_s^2\psi(0,Q_4,0)$ can be performed using the formula
\begin{eqnarray*}
\partial_s^2\psi(0,Q_4,0)&=&\partial_{f}^2F_E(Q_4,0)[v_4,v_4]\\
&=& \left. \Big\{\frac{d^2}{ds^2}F_E(Q_4,sv_4)\Big\}\right|_{s=0}.
\end{eqnarray*}
Observe  from \eqref{Eq71} that
\begin{eqnarray*}
F_E\big(Q_4,sv_4\big)&=&G\big(0,{(1-Q_4^2)}/4,\alpha_{Q_4}+sv_4\big)\\
&=&-G_E\big({(1-Q_4^2)}/4,\alpha_{Q_4}+s v_4\big)
\end{eqnarray*}
with
\[G_E\big({1-Q_4^2}/4,\alpha_{Q_4}+s v_4\big)= \hbox{Im}\Big\{\Big( \frac{1-Q_4^2}{4} \big(\overline{\alpha_{Q_4}(w)}+s\overline{v_4(w)}\big)+I(s) \Big) w \big(\alpha_{Q_4}^\prime(w)+s v_4^\prime(w)\big)\Big\},
\]
\[I(s)\triangleq \frac12 \fint_{\T} \frac{\overline{A}+s\overline{B}}{A+sB}\big(\alpha_{Q_4}^\prime(\tau)+s v_4^\prime(\tau)\big) d \tau\]
and
\begin{eqnarray*}
A&=&\alpha_{Q_4}(\tau)-\alpha_{Q_4}(w)\\
B&=&v_4(\tau)-v_4(w).
\end{eqnarray*}
It is easy to obtain
\begin{eqnarray*}
\Big\{\frac{d^2}{ds^2}G_E({1-Q_4^2}/4,\alpha_{Q_4}+s v_4)\Big\}_{s=0}&=& \frac{1-Q_4^2}{2}\hbox{Im}\Big\{\overline{v_4(w)} w v_4^\prime(w)\Big\}+\hbox{Im}\Big\{ 2I^\prime(0) w v_4^\prime(w)+I^{\prime\prime}(0) w\alpha_{Q_4}^\prime(w)\Big\}\\
&\triangleq&\frac{1-Q_4^2}{2}\hbox{Im}\big\{I_1(w)\big\}+\hbox{Im}\big\{I_2(w)\big\}+\hbox{Im}\big\{I_3(w)\big\}.
\end{eqnarray*}
We start by computing $\Pi\,  \hbox{Im}\big\{I_1(w)\big\}$. It is
straightforward to show
\[
\overline{v_4(w)} w v_4^\prime(w)=\frac{w^2}{w^2-Q_4}\frac{-3Q_4 w^2+5}{(1-Q_4w^2)^2}.
\]
Set $z=w^2$. Then 
\[
\overline{v_4(w)} w v_4^\prime(w)=\frac{z}{z-Q_4}\frac{-3Q_4 z+5}{(1-Q_4z)^2}.
\]
Note that $z\mapsto \frac{z}{z-Q_4}\frac{-3Q_4 z+5}{(1-Q_4z)^2}$ is
holomorphic in the annulus of small radius $Q_4$ and large radius
$\frac{1}{Q_4}$, and therefore it admits a Laurent expansion in this
domain. To evaluate $\Pi\, \hbox{Im}\big\{I_1(w)\big\}$ it suffices to
compute the coefficients of $z^2$ and $\frac{1}{z^2}$ in that
expansion using the residue theorem. The coefficient of $z^2$ is given by
\[
a\triangleq \fint_{\T}\frac{1}{z-Q_4}\frac{-3Q_4 z+5}{(1-Q_4z)^2}\frac{dz}{z^2}.
\]
Using the change of variable $z\mapsto \frac1z$ we obtain
\begin{eqnarray*}
a&=&\fint_{\T}\frac{z^2}{1-Q_4z}\frac{5 z-3Q_4}{(z-Q_4)^2}{dz}\\
&\triangleq&\fint_{\T}\frac{g(z)}{(z-Q_4)^2}{dz}\\
&=&g^\prime(Q_4)
\end{eqnarray*}
with
\[g(z)\triangleq\frac{z^2(5 z-3Q_4)}{1-Q_4z}.
\]
Thus we obtain, 
\[
a=\frac{-7Q_4^4+9Q_4^2}{(1-Q_4^2)^2}.
\]
Now we move on to the coefficient of $\frac{1}{z^2}$ given by  the formula
\[
b\triangleq \fint_{\T}\frac{z^2}{z-Q_4}\frac{-3Q_4 z+5}{(1-Q_4z)^2}{dz}.
\]
Using the residue theorem we obtain
\[
b=Q_4^2\frac{-3Q_4^2+5}{(1-Q_4^2)^2} .
\]
Consequently
\begin{eqnarray}\label{Fg1}
\nonumber \frac{1-Q_4^2}{2}\Pi\,  \hbox{Im}\big\{I_1(w)\big\}&=&\frac{1-Q_4^2}{2}(a-b)e_4\\
&=&2Q_4^2 e_4.
\end{eqnarray}
Next we compute $\Pi \,\hbox{Im}\{I_2(w)\}$. First notice that
\[
2 I^\prime(0)=\fint_{\T} \frac{\overline{A}}{A} v_4^\prime(\tau) d \tau+\fint_{\T} \frac{A\overline{B}-B\overline{A}}{A^2}\alpha_{Q_4}^\prime(\tau) d \tau.
\]
We rewrite $I_2(w)$ as follows,
\[ I_2(w)=I_{21}(w)+I_{22}(w)+I_{23}(w),\]
where
\[I_{21}(w)=wv_4^\prime(w)\fint_{\T} \frac{\overline{A}}{A} v^\prime_4(\tau) d \tau\]
\[I_{22}(w)=wv_4^\prime(w)\fint_{\T} \frac{\overline{B}}{A}\alpha_{Q_4}^\prime(\tau) d \tau\]
and
\[I_{23}(w)=-wv_4^\prime(w)\fint_{\T} \frac{B\overline{A}}{A^2}\alpha_{Q_4}^\prime(\tau) d \tau.\]
Straightforward computations imply
\[
\fint_{\T} \frac{\overline{A}}{A} v_4^\prime(\tau) d \tau=\overline{w}\fint_{\T} \frac{Q_4w \tau-1}{\tau-Q_4\overline{w}} \frac{-3Q_4 \tau^6+5\tau^4}{(1-Q_4\tau^2)^2} d \tau.
\]

Hence by the residue theorem we obtain
\[
\fint_{\T} \frac{\overline{A}}{A} v_4^\prime(\tau) d \tau=\overline{w}^5({Q_4^2-1}) \frac{-3Q_4^7 \overline{w}^2+5Q_4^4}{(1-Q_4^3\overline{w}^2)^2}.
\]
Since we can extend $z\mapsto I_{21}(z)$ to a holomorphic function in the
annulus of small radius $Q$ and large radius $\frac{1}{Q_4}$, then $I_{21}$
admits a Laurent expansion in this domain. As before, to evaluate
$\Pi\, \hbox{Im}\big\{I_{21}(w)\big\}$ it suffices to compute the
coefficients of $z^4$ and $\frac{1}{z^4}$ in that expansion using
the residue theorem. The coefficient of $z^4$ is given by
\[\tilde{a}=\fint_{\T} w v_4^\prime(w)\overline{w}^5({Q_4^2-1}) \frac{-3Q_4^7 \overline{w}^2+5Q_4^4}{(1-Q_4^3\overline{w}^2)^2} \frac{dw}{w^5}.\]
Moreover, the coefficient of $\frac{1}{z^4}$ is given by
\[\tilde{b}=\fint_{\T}  w v_4^\prime(w)\overline{w}^5({Q_4^2-1}) \frac{-3Q_4^7 \overline{w}^2+5Q_4^4}{(1-Q_4^3\overline{w}^2)^2} w^3 dw.\]
One may thus deduce the following equality,
\begin{eqnarray}\label{bout1}
\nonumber \Pi \,\textnormal{Im} \lbrace I_{21}(w) \rbrace &=&(\tilde{a}-\tilde{b})e_4 \\
&=& \frac {{Q_4}^{6} \left( 45-58\,{Q_4}^{4}+21\,{Q_4}^{8} \right) }{({Q_4}^{4}-1)({Q_4}^{2}+1)}e_4.
\end{eqnarray}
As for  $I_{22}$, using the change of variable $\tau\mapsto \overline{\tau}$,
we obtain by the residue theorem
\begin{eqnarray*}
\fint_{\T} \frac{\overline{B}}{A}\alpha_{Q_4}^\prime(\tau) d \tau&=&-\overline{w}\fint_{\T}\frac{v_4(\tau)-v_4(\overline{w})}{\tau-\overline{w}}\frac{1-Q_4\tau^2}{1-Q_4\overline{w}\, \tau}\frac{d\tau}{\tau}\\&=&-v_4(\overline{w}).
\end{eqnarray*}
Again,  we can extend $z\mapsto I_{21}(z)$ to  a holomorphic function in the same annulus and thus we just need to compute the coefficients of $z^4$   and $\frac{1}{z^4}$ denoted  by  $\tilde{c}$ and  $\tilde{d}$, respectively: 
\[\tilde{c}=-\fint_{\T}  w v_4^\prime(w) v_4(\overline{w})\frac{dw}{w^5}\]
and
\[\tilde{d}=-\fint_{\T}  w v_4^\prime(w)v_4(\overline{w}) w^3 dw.\]
According to the residue theorem we obtain
\begin{eqnarray}\label{bout2}
\nonumber \Pi\,  \textnormal{Im} \lbrace I_{22}(w)\rbrace &=&(\tilde{c}-\tilde{d})e_4 \\
&=& \frac {4{Q_4}^{2} }{({Q_4}^{2}-1)}e_4.
\end{eqnarray}
Now we move on to the last term $I_{23}(w)$. The computations are very
tedious and we use the Maple symbolic manipulation package to obtain the
following expressions
\[-\fint_{\T} \frac{B\overline{A}}{A^2}\alpha_{Q_4}^\prime(\tau) d \tau=\frac{\sum_{i=0}^5 \tilde{\alpha}_{2i}(Q_4)w^{2i}}{\omega^3(\omega^2-Q_4^3)(Q_4^3+Q_4 \omega^4-(1+Q_4^4)\omega^2)}\]
where
\[\tilde{\alpha}_{0}(Q_4)\triangleq 3Q_4^7-4Q_4^9, \quad\tilde{\alpha}_{2}(Q_4)\triangleq4Q_4^{10}-4Q_4^8+7Q_4^6-5Q_4^4, \]
\[\tilde{\alpha}_4(Q_4)\triangleq Q_4^9-8Q_4^7+7Q_4^5-Q_4^3, \quad \tilde{\alpha}_6(Q_4)\triangleq Q_4^8-3Q_4^6+3Q_4^4-Q_4^2, \]
\[\tilde{\alpha}_8(Q_4)\triangleq3Q_4^3-2Q_4^5-Q_4\quad  \text{ and }\quad  \tilde{\alpha}_{10}(Q_4)\triangleq Q_4^2-1.\]
%

Let $\tilde{e}$ and $\tilde{f}$ be the coefficient of $z^4$ and $\frac{1}{z^4}$ in the Laurent expansion of $I_{23}(z)$.
Again using Maple, we obtain
\begin{eqnarray}\label{bout3}
\nonumber \Pi\, \textnormal{Im} \lbrace I_{23}(w) \rbrace &=&(\tilde{e}-\tilde{f})e_4 \\
 &=&-{\frac {Q_4^{2} \left( 21\,Q_4^{12}-65\,Q_4^{8}-12\,Q_4^{6}+47\,Q_4^{4}+12\,Q_4^{2}+5 \right) }{(Q_4^{4}-1)(Q_4^{2}+1)}}e_4.
\end{eqnarray}
Now using $(\ref{bout1})$, $(\ref{bout2})$ and $(\ref{bout3})$ we find
\begin{eqnarray}\label{boutI2}
  \Pi \,\textnormal{Im} \big\{ I_2(w) \big\}&=&\frac{1}{2}\frac {Q_4^{2} \left( 7\,Q_4^{4}-2\,Q_4^{2}-1 \right) }{ \left( Q_4^{2}-1 \right)  }e_4 .
 \end{eqnarray}

Now we compute $\textnormal{Im}\Big\{I_3(w)\Big\}$. First we notice that
\[I^{\prime\prime}(0)=\fint_{\T} \frac{B^2 \overline{A}-B \overline{B}A}{A^3} \alpha_{Q_4}^\prime(\tau) d\tau + \fint_{\T} \frac{\overline{B}A-B\overline{A}}{A^2} v_4^\prime(\tau) d\tau.  \]
Consequently, we can split  $I_3$ as follows,
\[I_3(w) =I_{31}(w)+I_{32}(w)+I_{33}(w)+I_{34}(w)\]
where
\[I_{31}(w)=w \alpha_{Q_4}^\prime(w)\fint_{\T} \frac{B^2 \overline{A}}{A^3} \alpha_{Q_4}^\prime(\tau) d\tau,\]
\[I_{32}(w)=w \alpha_{Q_4}^\prime(w)\fint_{\T} \frac{\overline{B}}{A} v_4^\prime(\tau) d\tau,\]
\[I_{33}(w)=-w \alpha_{Q_4}^\prime(w) \fint_{\T} \frac{B\overline{A}}{A^2} v_4^\prime(\tau) d\tau\]
and
\[I_{34}(w)=-w \alpha_{Q_4}^\prime(w)\fint_{\T} \frac{B \overline{B}}{A^2} \alpha_{Q_4}^\prime(\tau) d\tau.\]
In the following, we denote by  $a_{3i}$ and $b_{3i}$ the coefficient in front of $z^4$ and $\frac{1}{z^4}$  in the Laurent expansion of $I_{3i}(z)$ on the same annulus as before. We obtain the following expressions using Maple,
\begin{eqnarray}
\nonumber a_{31}&=&-\frac {Q_4^{2} \left( -12\,Q_4^{6}+8\,Q_4^{2}+26\,Q_4^{4}+3+12\,Q_4^{12}+4\,Q_4^{10}-33\,Q_4^{8} \right) }{(Q_4^{4}-1)^2  \left( Q_4^{2}+1 \right) }
\end{eqnarray}
and
\begin{eqnarray}
\nonumber  b_{31}&=&-\frac {Q_4^{6} \left( 78\,Q_4^{4}+8\,Q_4^{2}+3+56\,Q_4^{12}-129\,Q_4^{8}-20\,Q_4^{6}+12\,Q_4^{10} \right) }{ \left( Q_4^{4}-1 \right)^2  \left( Q_4^{2}+1 \right) }.
\end{eqnarray}
Therefore,
\begin{eqnarray}\label{bout4}
\nonumber \Pi\, \textnormal{Im} \lbrace I_{31}(w) \rbrace &=&(a_{31}-b_{31})e_4 \\
 &=& \frac { \left( 56\,Q_4^{12}+12\,Q_4^{10}-85\,Q_4^{8}-12\,Q_4^{6}+26\,Q_4^{4}+8\,Q_4^{2}+3 \right) Q_4^{2}}{ \left( Q_4^{2}+1 \right)  \left( -1+Q_4^{4} \right) }e_4. 
\end{eqnarray}

Similarly we obtain using Maple, 
\[\fint_{\T} \frac{\overline{B}}{A} v_4^\prime(\tau) d\tau=\frac{\sum_{i=0}^{5}\beta_{2i}(Q_4)w^{2i}}{w^5(w^2-Q_4^3)^2(w^2-Q_4)(Q_4^2-1)^2}\]
where

\[\beta_0(Q_4):=-6Q_4^9+3Q_4^7+3Q_4^{11}, \quad \beta_2(Q_4):=3Q_4^{10}-11Q_4^8+13Q_4^6-5Q_4^4,\]
\[\beta_4(Q_4):=3Q_4^9-11Q_4^7+13Q_4^5-5Q_4^3, \quad \beta_6(Q_4):=-6Q_4^6+13Q_4^4-5Q_4^2,\]
\[\beta_8(Q_4):=-Q_4^5+6Q_4^3-5Q_4 \text{ and }\beta_{10}(Q_4):=3Q_4^2-5.\]
The coefficients in the Laurent expansion have the following expressions,
\[a_{32}=0\]
and
\[b_{32}=\frac { \left( 7\,Q_4^{6}-9\,Q_4^{4}+3\,Q_4^{2}-5 \right) Q_4^{2}}{(\,Q_4^{2}-1)^2}.\]
Consequently,
\begin{eqnarray}\label{bout5}
\nonumber \Pi \textnormal{Im} \lbrace I_{32}(w) \rbrace &=&(a_{32}-b_{32})e_4 \\
&=& -\frac { \left( 7\,Q_4^{6}-9\,Q_4^{4}+3\,Q_4^{2}-5 \right) Q_4^{2}}{(\,Q_4^{2}-1)^2}e_4.
\end{eqnarray}

Thanks to Maple, one may find the following expression
\[I_{33}(w)=\frac{(w^2-Q_4)Q_4^4 \big(\sum_{i=0}^{6}\gamma_{2i}(Q_4)w^{2i}\big)}{w^6 (w^2-Q_4^3)(-(Q_4^{10}+3Q_4^6)w^2+3Q_4^3(Q_4^4+1)w^4-(3Q_4^4+1)w^6+Q_4^9+Q_4w^8)}\]

where 
\[\gamma_0(Q_4)=18Q_4^{12}-15Q_4^{10}, \quad \gamma_2(Q_4)=30Q_4^{11}-68Q_4^9+48Q_4^7-{18}Q_4^{13},\]
\[\gamma_4(Q_4)=-15Q_4^{12}-105Q_4^8+90Q_4^6+80Q_4^{10}-45Q_4^4, \quad \gamma_6(Q_4)=-12Q_4^{11}-136Q_4^7+122Q_4^5+66Q_4^9-40Q_4^3,\]
\[\gamma_8(Q_4)=-35Q_4^2+104Q_4^4+46Q_4^8-112Q_4^6, \quad \gamma_{10}(Q_4)=-30Q_4-64Q_4^5+86Q_4^3,\]
and
\[\gamma_{12}(Q_4)=30Q_4^2-25.\]

This allows one to obtain
\[a_{33}=\frac {Q_4^{6} \left( 9\,Q_4^{8}-26\,Q_4^{4}+25 \right) }{ \left( Q_4^{2}+1 \right)  \left( -1+Q_4^{4} \right) ^{2}}\]
and
\[b_{33}=\frac {Q_4^{6} \left( -131\,Q_4^{8}+56\,Q_4^{12}+78\,Q_4^{4}+5+12\,Q_4^{2}-24\,Q_4^{6}+12\,Q_4^{10} \right) }{ \left( Q_4^{2}+1 \right)  \left( -1+Q_4^{4} \right) ^{2}}.\]
Finally we obtain
\begin{eqnarray}\label{bout6}
\nonumber \Pi\, \textnormal{Im} \lbrace I_{33}(w) \rbrace &=&(a_{33}-b_{33})e_4 \\
&=& -4\,{\frac {Q_4^{6} \left( 14\,Q_4^{8}+3\,Q_4^{6}-21\,Q_4^{4}-3\,Q_4^{2}+5 \right) }{ \left( Q_4^{2}+1 \right)  \left( -1+Q_4^{4} \right) }}e_4.
\end{eqnarray}
For the last term, we  use Maple to obtain
\[-\fint_{\T} \frac{B \overline{B}}{A^2} \alpha_{Q_4}^\prime(\tau) d\tau= -\frac{\sum_{i=0}^5 \xi_{2i}(Q_4)w^{2i}}{w^5(w^2-Q_4^3)\big((Q_4^4+Q_4^2+1)w^2(Q_4-w^2)+Q_4(w^6-Q_4^3)\big)}\]

where
\[\xi_{0}(Q_4)=-3Q_4^7, \quad \xi_{2}(Q_4)=3Q_4^8-3Q_4^6+5Q_4^4,\]
\[\xi_{4}(Q_4)=3Q_4^7-8Q_4^5+5Q_4^3, \quad \xi_{6}(Q_4)=3Q_4^6-8Q_4^4+5Q_4^2,\]
\[\xi_{8}(Q_4)=-6Q_4^3+5Q_4 \text{ and } \xi_{10}(Q_4)=2Q_4^2+5.\]
Thus we deduce using Maple once again 
\[a_{34}=-\frac{Q_4^2(3Q_4^2-5)}{(Q_4^2-1)^2}\]
and
\[b_{34}=-\frac{Q_4^2(4Q_4^6-4Q_4^4+3Q_4^2-5)}{(Q_4^2-1)^2}.\]
Therefore we find
\begin{eqnarray}\label{bout7}
\nonumber \Pi\, \textnormal{Im} \lbrace I_{34}(w) \rbrace &=&(a_{34}-b_{34})e_4 \\
&=& \frac{4Q_4^6}{(Q_4^2-1)}e_4.
\end{eqnarray}
Finally, using $(\ref{bout4})$, $(\ref{bout5})$,$(\ref{bout6})$ and $(\ref{bout7})$ one finds
\begin{equation}\label{boutI3}
 \Pi\, \textnormal{Im} \lbrace I_{3}(w) \rbrace =-2\,{\frac {Q_4^{2} \left( 2\,Q_4^{6}-4\,Q_4^{4}+Q_4^{2}-1 \right) }{ \left( -1+Q_4^{2} \right) ^{2}}}e_4.\\
\end{equation}
Now combining  $(\ref{Fg1})$,$\ref{boutI2})$ with $(\ref{boutI3})$  and
simplifying the polynomial equation of $Q_4$, we obtain
\begin{eqnarray*}
\partial_t^2\psi(0,Q_4,0)&=&-  \frac {Q_4^2 \left( 3\,Q_4^{6}+\,Q_4^{4}-5\,Q_4^{2}+5 \right) }{ \left( -1+Q_4^{2} \right) ^{2}}e_4\\
\end{eqnarray*}
To summarize, up to this point we have proved that the V-state equation \eqref{Idenz1} reduces in the small neighbourhood $I_\delta\triangleq (-\delta,\delta)\times (-\delta+Q_4,Q_4+\delta)\times (-\delta,\delta)$ to the finite-dimensional equation
\begin{equation}\label{zero}
 \psi(\EE,Q,s)=0.
 \end{equation}
As $\psi$ is at least $\mathcal{C}^3$, we can use a Taylor expansion with the integral  form for the remainder around the point $(0,Q_4,0)$,
\begin{eqnarray}
\nonumber \psi(\EE,Q,s)&=&\psi(0,Q_4,0)+s\partial_s\psi(0,Q_4,0)+\EE \partial_{\EE}\psi(0,Q_4,0)+(Q-Q_4)\partial_Q\psi(0,Q_4,0)+\frac{s^2}{2}\partial_{ss}^2 \psi(0,Q_4,0)\\
\nonumber &+&\frac{\EE^2}{2}\partial_{\EE}^2 \psi(0,Q_4,0)+\frac{(Q-Q_4)^2}{2}\partial_{Q}^2 \psi(0,Q_4,0)+\EE(Q-Q_4)\partial_{\EE Q}^2 \psi(0,Q_4,0)+\EE s \partial_{\EE s}^2 \psi(0,Q_4,0)\\
\nonumber&+&(Q-Q_4) s \partial_{Q s}^2 \psi(0,Q_4,0)+\tilde{\EE}(\EE,Q,s)e_4
\end{eqnarray}
where
\[ \tilde{\EE}(\EE,Q,s) e_4=\int_0^1 \frac{(1-\theta )^2}{2!}D^3\psi\big(\theta\EE,Q_4+\theta(Q-Q_4),\theta s\big)\big(\EE,Q-Q_4,s\big)^3d\theta.\]
For a given vector $h$  we   use the notation $h^3$ to denote $(h,h,h)$.
Consequently, using the preceding  computations concerning the quadratic expansion, we obtain for any $( \EE,Q,s)\in I_\delta$
\[ \psi(\EE,Q,s)=\Big[a\EE^2+b\,s(Q-Q_4)+cs^2 + \tilde{\EE}(\EE,Q,s) \Big] e_4\]
with
\[a=\frac{Q_4^2(Q_4^4-4Q_4^2+3)}{48} ,\; b=4Q_4(Q_4^2+1)\]
and
\[c=-  \frac {Q_4^2 \left( 3\,Q_4^{6}+\,Q_4^{4}-5\,Q_4^{2}+5 \right) }{ \left( -1+Q_4^{2} \right) ^{2}}.\]
Since $Q_4=\sqrt{\sqrt{2}-1}$, one may easily check that
\[a>0, \; b>0 \text{ and } c<0.\]
We introduce the parameters $\tilde{c}=-c$, $\tilde{b}=-\frac{b^2}{4c}$ and $d=-\frac{b}{2c}$, which are three positive constants, and write equation \eqref{zero} in the new variables  ${\bf{Q}}=Q-Q_4$ 	and $X=s-d{\bf{Q}}$. In this
way, we find in a small neighbourhood of zero,
\begin{equation}\label{RI}
 \hat{\psi}(\EE,{\bf{Q}},X) \triangleq \Big[a\EE^2-\tilde{c}X^2+\tilde{b}{\bf{Q}}^2 + \hat{\EE}(\EE,{\bf{Q}},X) \Big] e_4=0
 \end{equation}
 where 
\[ \hat{\EE}(\EE,{\bf{Q}},X)=\tilde{\EE}(\EE,Q,s).\]
For the quadratic equation
\[a\EE^2-\tilde{c}X^2+\tilde{b}{\bf{Q}}^2=0\]
we find for given $\EE \neq 0$ two disjoint curves 
\[X= \pm \sqrt{\frac{a}{\tilde{c}}\EE^2+\frac{\tilde{b}}{\tilde{c}}{\bf{Q}}^2}.\]
We prove that this structure persists for the full equation,
\begin{equation}\label{RI2}
a\EE^2-\tilde{c}X^2+\tilde{b}^2 {\bf{Q}}^2+ \hat{\EE}(\EE,{\bf{Q}},X)=0.
\end{equation}
To this end, we check that the solutions have  the form
\[X= \sqrt{\frac{a}{\tilde{c}}\EE^2+\frac{\tilde{b}}{\tilde{c}}{\bf{Q}}^2}+y\]
where $y$ is a small correction described shortly below.
From $(\ref{RI2})$  we  deduce that $y$ satisfies the equation
$$
G({\bf{Q}},y)=y,
$$
with
 $$G({\bf{Q}},y)=-\frac{y^2}{2\sqrt{\frac{a}{\tilde{c}}\EE^2+\frac{\tilde{b}}{\tilde{c}}{\bf{Q}}^2}}+\frac{\hat{\EE}\big(\EE,{\bf{Q}},\sqrt{\frac{a}{\tilde{c}}\EE^2+\frac{\tilde{b}}{\tilde{c}}{\bf{Q}}^2}+y\big)}{2 \tilde{c} \sqrt{\frac{a}{\tilde{c}}\EE^2+\frac{\tilde{b}}{\tilde{c}}{\bf{Q}}^2}}.$$
      We prove the following lemma,
      
 \begin{lemma}\label{existence-y}
 There exist two   strictly positive constants $\eta$ and $\EE_0<1$ such that   for any  $ 0<\EE  \leq \EE_0$
 \[ \begin{array}{lll}
    G: &\overline{B}_{\eta \EE^{\frac{5}{6}}} \times  \overline{B}_{\eta \EE^{\frac{3}{2}}}     &\longrightarrow\overline{B}_{\eta \EE^{\frac{3}{2}}} \\
         &({\bf{Q}},y) &\longmapsto G({\bf{Q}}, y)
    \end{array} \]
is well-defined. Moreover, for any ${\bf{Q}}\in \overline{B}_{\eta \EE^{\frac{5}{6}}}$,  $G$ admits a unique  fixed point $y({\bf{Q}})$ which depends continuously on ${\bf{Q}}$.
 \end{lemma}
\begin{proof}
We begin with the simple  inequality,

\[\vert G({\bf{Q}},y) \vert \leq \frac{y^2}{2\sqrt{\frac{a}{\tilde{c}}\EE^2+\frac{\tilde{b}}{\tilde{c}}{\bf{Q}}^2}} +\left| \frac{\hat{\EE}(\EE,{\bf{Q}},\sqrt{\frac{a}{\tilde{c}}\EE^2+\frac{\tilde{b}}{\tilde{c}}{\bf{Q}}^2}+y)}{2\sqrt{\frac{a}{\tilde{c}}\EE^2+\frac{\tilde{b}}{\tilde{c}}{\bf{Q}}^2}}\right| .\]
It is plain that there exists  $C>0$  such that,
\[\frac{y^2}{2 \tilde{c} \sqrt{\frac{a}{\tilde{c}}\EE^2+\frac{\tilde{b}}{\tilde{c}}{\bf{Q}}^2}} \leq C \eta^2 \EE^2.\]
For the second  term we use the cubic form of the remainder and its continuity,
\[\Big|\hat{\EE}(\EE,{\bf{Q}},\sqrt{\frac{a}{\tilde{c}}\EE^2+\frac{\tilde{b}}{\tilde{c}}{\bf{Q}}^2}+y)\Big|\leq C \eta^3 \EE^{\frac{5}{2}}.\]
It follows that,
\[ \vert G({\bf{Q}},y)\vert \leq C \eta^2 \EE^2 + C \eta^3 \EE^{\frac{3}{2}}.\]
Choosing $\eta$ such that \begin{equation}\label{eta}
C \eta+ C \eta^2 \leq 1
\end{equation}
we guarantee  that  $G$ is  well-defined. In order to apply the Banach fixed point theorem with a parameter, we need only check that $G$ is a contraction. Let $y$ and $\tilde{y}$ be two elements of
$\overline{B}_{\eta \EE^{\frac{3}{2}}}$, then we have
\begin{eqnarray}
\nonumber \vert G({\bf{Q}},y)-G({\bf{Q}},\tilde{y}) \vert &\leq &\left|  \frac{(y+\tilde{y})(\tilde{y}-y)}{2\sqrt{\frac{a}{\tilde{c}}\EE^2+\frac{\tilde{b}}{\tilde{c}}{\bf{Q}}^2}}\right| \\
\nonumber & + &\left| \frac{\hat{\EE}\bigg(\EE,{\bf{Q}},\sqrt{\frac{a}{\tilde{c}}\EE^2+\frac{\tilde{b}}{\tilde{c}}{\bf{Q}}^2}+y\bigg)-\hat{\EE} \bigg(\EE,{\bf{Q}},\sqrt{\frac{a}{\tilde{c}}\EE^2+\frac{\tilde{b}}{\tilde{c}}{\bf{Q}}^2}+\tilde{y}\bigg)}{2 \tilde{c} \sqrt{\frac{a}{\tilde{c}}\EE^2+\frac{\tilde{b}}{\tilde{c}}{\bf{Q}}^2}}\right| .
\end{eqnarray}
For the first term we write 
\[ \left|  \frac{(y+\tilde{y})(\tilde{y}-y)}{2\sqrt{\frac{a}{\tilde{c}}\EE^2+\frac{\tilde{b}}{\tilde{c}}{\bf{Q}}^2}}\right|  \leq   C  \eta  \EE^{\frac{1}{2}}\vert \tilde{y}-y\vert.\]

As for the second term we split it into two parts
$$
 \hat{\EE}\Big(\EE,{\bf{Q}},{\scriptstyle{\sqrt{\frac{a}{\tilde{c}}\EE^2+\frac{\tilde{b}}{\tilde{c}}{\bf{Q}}^2}+y}}\Big)-\hat{\EE}\Big(\EE,{\bf{Q}},{\scriptstyle{\sqrt{\frac{a}{\tilde{c}}\EE^2+\frac{\tilde{b}}{\tilde{c}}{\bf{Q}}^2}+\tilde{y}}}\Big)= T_1(\EE,{\bf{Q}},\tilde{y},y)+T_2(\EE,{\bf{Q}},\tilde{y},y)
$$
with
$$
 T_1(\EE,{\bf{Q}},\tilde{y},y)=\int_0^1{\scriptstyle{ \frac{(1-\theta )^2}{2!}}}\Big[D^3\psi\big(\beta(\theta,\EE,{\bf Q}, y)\big)-D^3\psi\big(\beta(\theta,\EE,{\bf Q}, \tilde y)\big)\Big]\big[v(\EE,{\bf Q}, y)\big]^3d\theta
$$
where
$$
\beta(\theta,\EE,{\bf Q}, y)=\bigg(\theta \EE,Q_4+\theta {\bf{Q}},\theta\Big({\scriptstyle{\sqrt{\frac{a}{\tilde{c}}\EE^2+\frac{\tilde{b}}{\tilde{c}}{\bf{Q}}^2}+y+d {\bf{Q}}}} \Big)\bigg), \quad v(\EE,{\bf Q}, y)=\begin{pmatrix} 
\EE  \\
{\bf{Q}}\\
{\scriptstyle{\sqrt{\frac{a}{\tilde{c}}\EE^2+\frac{\tilde{b}}{\tilde{c}}{\bf{Q}}^2}+y+d {\bf{Q}}}},
\end{pmatrix}
$$
and
\begin{eqnarray*}
\nonumber T_2(\EE,{\bf{Q}},\tilde{y},y)&=&\int_0^1 {\scriptstyle{\frac{(1-\theta)^2}{2!}}}\Big(D^3\psi\big({\scriptstyle{\beta(\theta,\EE,{\bf Q}, \tilde y)}}\big)
\big[v(\EE,{\bf Q}, y)\big]^3d\theta-D^3\psi\big({\scriptstyle{\beta(\theta,\EE,{\bf Q}, \tilde y}})\big)
\big[v(\EE,{\bf Q}, \tilde y)\big]^3\Big)d\theta.
\end{eqnarray*}
According to Remark \ref{rmkk1} and \eqref{Eqww1} one has that  $\partial_yD^3 \psi$  is continuous. This implies by virtue of the  mean value theorem   that \begin{eqnarray*}
\left| T_1(\EE,Q,\tilde{y},y) \right| &\leq& C|y-\tilde y|\big(\EE^3+|{\bf Q }|^3+|y|^3\big)\\
&\le&C\vert y - \tilde{y}\vert\big( \EE^3+ \eta^3 \EE^{\frac{5}{2}}+\eta^3 \EE^{\frac{9}{2}}\big)\\ 
&\le&C\vert y - \tilde{y}\vert\big( \EE^3+ \eta^3 \EE^{\frac{5}{2}}\big). 
\end{eqnarray*}
For the term $T_2$ we use the multi-linear structure of $D^3\psi$ which gives
\begin{eqnarray*} \left| T_2(\EE,Q,\tilde{y},y) \right| &\leq& C\big(\EE^2+{\bf Q}^2+y^2\big)|y-\tilde y|\\
&\le&C\big(\EE^2+ \eta^2 \EE^{\frac{5}{3}}\big) \vert y - \tilde{y}\vert. 
\end{eqnarray*}

Finally, we have the following inequality, since $ \EE  , \eta\in[0,1]$,
\begin{eqnarray*} \vert G(Q,y)-G(Q,\tilde{y})\vert &\leq &C\vert y - \tilde{y}\vert\big( \EE^3+ \eta^3 \EE^{\frac{5}{2}}+\eta^2 \EE^{\frac{5}{3}}\big)\\
&\le& C\vert y - \tilde{y}\vert\big( \EE^3+\eta^2 \EE^{\frac{5}{3}}\big).
\end{eqnarray*}
For the choice of  $\eta$ made before in \eqref{eta}, it suffices to fix $\EE_0$ such that 
\begin{equation*}
C\big( \EE_0^3+\eta^2 \EE_0^{\frac{5}{3}}\big)<1
\end{equation*}
in order to guarantee  that $G$ is a contraction and hence admits a unique fixed point $y({\bf{Q}})\in \overline{B}_{\eta \EE^{\frac{3}{2}}}$. The continuity dependence with respect to ${\bf Q}$ is classical and follows from the fixed point theorem with a parameter. 
This achieves the proof of the lemma.
\end{proof}
Therefore equation \eqref{RI} admits a solution in the form
$$
X=X_+({\bf Q})=\sqrt{\frac{a}{\tilde{c}}\EE^2+\frac{\tilde{b}}{\tilde{c}}{\bf{Q}}^2}+y({\bf{Q}}).
$$

Reproducing the same analysis  we can prove that  the equation \eqref{RI} admits another solution of the form 
$$
X_{-}({\bf Q})=-{{\sqrt{\frac{a}{\tilde{c}}\EE^2+\frac{\tilde{b}}{\tilde{c}}{\bf{Q}}^2}}}+\widehat{y}({\bf{Q}}),\quad \widehat{y}({\bf Q}) \in \overline{B}_{\eta \EE^{\frac{3}{2}}}.
$$
It remains only to check that the curves ${\bf Q}\in \overline{B}_{\eta \EE^{\frac{5}{6}}}\mapsto  X_{\pm}$ are disjoint graphs. For this we write, by the triangular inequality,
\begin{eqnarray*}
 |X_+({\bf{Q}})-X_{-}({\bf Q})| & \geq &2 \sqrt{\frac{a}{\tilde{c}}\EE^2+\frac{\tilde{b}}{\tilde{c}}{\bf{Q}}^2}-\big(|y({\bf{Q}})|+|\widehat{y}({\bf Q}) |\big)\\
&\geq& C\big( \vert \EE \vert   - 2\eta\vert \EE \vert ^{\frac{3}{2}}\big)\\
& > & C\EE 
\end{eqnarray*}
if $\EE $ is small enough. Coming back to the initial unknowns,  we find two disjoint curves of solutions to equation \eqref{zero}. The proof of Theorem \ref{thmbreak} is now complete.
\end{proof}

\section{Conclusions}
\label{sec:conc}
In this paper we have described both numerically and analytically the
bifurcation structure of two-fold and three-fold symmetric
singly-connected vortex patch equilibria for the Quasi-Geostrophic
Shallow-Water equations.  The numerical results reveal that the branch
of solutions for two-fold symmetric equilibria, consisting of a main
branch of Kirchhoff elliptical vortices and secondary branches
bifurcating at the Love instability points for the Euler equations,
likely separates into an infinite set of disjoint branches for any
finite value of the Rossby deformation length $\EE^{-1}$.  This is
confirmed for $\EE\ll 1$ by mathematical analysis for the first
separation near the Love instability point for elliptical azimuthal
wavenumber $m=4$.\\

The numerical results for the three-fold symmetric equilibria also
reveal a separated branch, in this case existing for all $\EE$.  The
mathematical analysis confirms that solutions exist for small $\EE$
near the Euler limit ($\EE=0$); these solutions are near circular in
form.  The separated solutions are beyond the scope of this analysis.
In a future work, we will report on the stability and nonlinear
evolution of these solutions.\\

\begin{ackname}
DGD received support for this research from the UK Engineering
and Physical Sciences Research Council (grant number EP/H001794/1).
TH is partially supported by the the ANR project Dyficolti ANR-13-BS01-0003- 01.\end{ackname}

\end{document}